\documentclass[a4paper,reqno]{amsart}

\usepackage[utf8]{inputenc}
\usepackage{amssymb}
\usepackage{enumitem}
\usepackage{mathrsfs}
\usepackage{mathtools}
\usepackage{amsmath}
\usepackage{slashed}
\usepackage{xcolor}
\usepackage[margin=1in]{geometry}
\usepackage[colorlinks=true]{hyperref}
\hypersetup{linkcolor=violet,urlcolor=cyan,citecolor=cyan}
\usepackage{tikz}
\usepackage[normalem]{ulem}

\definecolor{greeen}{rgb}{0.0, 0.4, 0.0}

\usepackage{titletoc} 

\setlist[enumerate]{label=\emph{(\roman*)}}

\usepackage{environ}

\newtheorem{theorem}{Theorem}[section]

\newtheorem{lemma}[theorem]{Lemma}
\newtheorem{proposition}[theorem]{Proposition}

\theoremstyle{definition}
\newtheorem{definition}[theorem]{Definition}
\newtheorem{remark}[theorem]{Remark}
\numberwithin{equation}{section}

\def \C {{\mathbb{C}}}
\def \d {{\rm{d}}}

\def \Hcalf{{\mathcal{H}_\tau^f}}
\def \Hcali{{\mathcal{H}_\tau^i}}
\def \mainfield{\Phi}
\begin{document}

	\title[Supercritical wave in a large data regime]
	{Global dynamics of a supercritical wave equation \\  in a large data regime}
	
		\author[S.~Dong]{Shijie Dong}
	\address{Southern University of Science and Technology, Shenzhen International Center for Mathematics, and Department of Mathematics, 518055 Shenzhen, China.}
\email{dongsj@sustech.edu.cn, shijiedong1991@hotmail.com}
	
		\author[Z. Wyatt]{Zoe Wyatt}
	\address{Department of Pure Mathematics and Mathematical Statistics, University of Cambridge,
		Wilberforce Road, CB3 0WB, Cambridge, UK}
	\email{zoe.wyatt@maths.cam.ac.uk}

	\author[J. Zhao]{Jingya Zhao}
	\address{Southern University of Science and Technology, Department of Mathematics, 518055 Shenzhen, China.}
	\email{12231282@mail.sustech.edu.cn}
	\begin{abstract}
	We prove  the existence of global solutions to the nonlinear wave equation in $\mathbb{R}^{1+3}$ $$\Phi_{tt} - \Delta \Phi \pm \Phi|\Phi|^{p-1} = 0$$ in the energy-supercritical regime $p>5$, for a class of large initial data. Our initial data can be decomposed into two pieces, one which is dispersed in the sense of having large $L^2$ norm, while the other piece takes a localised short-pulse form. Consequently, we can obtain global existence for a class of initial data which is  large  in every homogeneous Sobolev norm $\dot{H}^s_x$ with $s \geq 0$. 
	\end{abstract}
	\maketitle
\section{Introduction}
We consider the semilinear wave equation in $\mathbb{R}^{1+3}$
\begin{subequations}\label{intro-eq} 
\begin{equation}\label{eq:model-sup}
\mainfield_{tt} - \Delta \mainfield  \pm \mainfield|\mainfield|^{p-1} =0, \qquad p>5,
\end{equation}
with real-valued solutions. 
The equation is called focusing or defocusing according to whether the sign of the nonlinearity is negative or positive. 
The initial conditions are prescribed at $t_0 = 1$ by 
\begin{align}\label{eq:ID-sup}
(\mainfield, \partial_t \mainfield)(t_0) = (\mainfield_0, \mainfield_1).
\end{align}
\end{subequations}

Our goal in the present article is to prove global existence for \eqref{intro-eq} for a class of data which is large in every homogeneous Sobolev space $\dot{H}^s_x$ with $s \geq 0$. This is best interpreted through the scaling symmetry
\begin{equation}\label{eq:scaling}
\mainfield(t,x) \mapsto \lambda^{2\over p-1} \mainfield(\lambda t, \lambda x),
\end{equation}  
which preserves the space of solutions to \eqref{intro-eq}. In three space dimensions,  the homogeneous Sobolev norm $\dot{H}_x^{s_c} \times \dot{H}_x^{s_c-1}$, where $s_c =\frac32 - \frac2{p-1}$, is called \textit{critical} since it is invariant under the rescaling \eqref{eq:scaling}. 
Consequently, the equation \eqref{intro-eq} is called energy-supercritical since the
solution cannot be controlled in the energy norm as $1<s_c<\frac32$. Indeed all conservation laws and monotonicity formula have scaling below the critical regularity and so cannot be used in this supercritical setting. Moreover, equation \eqref{intro-eq}  can be regarded as a model problem relevant to the Navier-Stokes equations in the sense that the nonlinearities in both equations are supercritical compared with their conserved quantity. 

The class of initial data that we consider in the current article can be decomposed as
\begin{align*}
 (\mainfield_0, \mainfield_1)
 =
 (u_0, u_1) + (w_0, w_1) \,.
\end{align*}
Our intuition is that the first piece $(u_0, u_1)$ is `dispersed' in the sense of having large $L^2$ norm. An example of such data is achieved by supposing that 
\begin{align*}
(u_0, u_1) \sim 
(\epsilon^{1\over 2} g_0(\epsilon x), \epsilon^{{3\over 2}} g_1(\epsilon x)),
\end{align*}
where $g_0, g_1$ are smooth functions decaying sufficiently fast at infinity. This ensures that the norms $\|(u_0, u_1)\|_{\dot{H}^s}$ are large for any $s 
\in [0,1]$. Our next assumption is that the second piece $(w_0, w_1)$ is  localised short-pulse data of the form
\begin{equation}\label{eq:shortpulse} \begin{split}
(w_0, w_1) &= \big(\epsilon^{1\over 2} h_0(\tfrac{1-r}{\epsilon}, \omega), \epsilon^{-{1\over 2}} h_1(\tfrac{1-r}{\epsilon}, \omega)\big),
\qquad w_1+\partial_r w_0 \sim \mathcal{O}(\epsilon^{1\over 2}),
\end{split}\end{equation}
where $h_0(\rho, \omega), h_1(\rho, \omega)$ are smooth  functions compactly supported away from the origin with respect to their first argument $\rho$.
Such data is inspired by Christodoulou’s characteristic initial data in his work  on the formation of black holes \cite{Christodoulou09} and shock formation for the irrotational relativistic Euler equations \cite{Christodoulou07}. The choice \eqref{eq:shortpulse} ensures that the norms $\|(w_0, w_1)\|_{\dot{H}^s}$ are large with $s 
\geq 1$. 

We can now state a rough version of our main result. The precise statement  is given in Theorem \ref{thm:main2}.

\begin{theorem}
    For a class of initial data that are large in any $\dot{H}^s_x(\mathbb{R}^3)$ with $s\geq 0$, the energy supercritical nonlinear wave equation \eqref{intro-eq} admits a classical global solution.
\end{theorem}

\subsection{Previous work}
The energy subcritical and energy critical versions of equation \eqref{intro-eq} (when $p \leq 5$ in 3D) have been extensively studied in the dispersive PDE community. Situations of global regularity or finite time blow-up can occur depending on whether the nonlinearity is focusing or defocusing. 
We refer to \cite{Taonotes, Yang2} and references within for further information on the large body of work in this direction. 

Our focus, however, in the present article is on the energy supercritical setting, and so we now review progress in this direction. On the one hand, global smooth solutions are known for small data by John and Lindblad--Sogge in \cite{John79, LS05}, and global weak solutions (including for large data) are constructed by Segal in \cite{Segal}. 
The works by Kenig--Merle and Killip--Visan in \cite{KiVi, KenigMerle}, analyse the nature of blow-up in the sense that a solution to \eqref{intro-eq} can only blow-up
in finite time, or be global but fail to scatter, if its critical Sobolev norm diverges. See also a related blow-up result for a supercritical nonlinear wave system with a particular  defocusing smooth potential by Tao in \cite{Tao}. 
In terms of large data results for the energy supercritical problem, much less is known. In both \cite{KriegerSchlag, BeSo}, equation \eqref{intro-eq} is studied under radial symmetry. In the first work \cite{KriegerSchlag} by Krieger and Schlag, the authors show global existence and restricted stability for a class of solutions with infinite critical Sobolev norm. In the latter work  \cite{BeSo} by Beceanu and Soffer, the authors establish global existence for  a class of outgoing radial data with infinite critical Sobolev norm. 

In the present article, we take the alternative perspective of using ``short pulse data” as introduced by Christodoulou in \cite{Christodoulou07, Christodoulou09}. The ``pulse'' terminology is because the data has large amplitude relative to its support. Generally speaking in the short pulse method, smallness restrictions are still imposed on the data along directions tangential to the outgoing light cone
${t - r=\text{constant}}$, however certain orthogonal directions involving $\underline{L} = \partial_t - \partial_r$ are allowed to be large. Similar to the product of ``good'' and ``bad'' derivatives in the null condition, the short pulse profile is assumed to be large in certain components in such a way that in the nonlinear interactions, large
components are always coupled with small components so that the terms are overall controllable.  

Motivated by this milestone method of Christodoulou, many works have since established global existence for nonlinear wave equations subject to short pulse initial data. We start with results in 3D, where  results showing global existence  given short pulse data have been given for a nonlinear wave equation satisfying the null condition by Miao--Pei--Yu in \cite{MPY}; see also \cite{Ding6} by Ding--Lu--Yin. Complementary to these global regularity results, we mention also the shock formation result by Miao and Yu in \cite{MiaoYu}. Note also several global existence results for
3D nonlinear hyperbolic equations with large data that is not of short pulse form in \cite{Yang, LuOhYa, LuOh, DoLiZh}. 
Results showing global existence for 2D nonlinear wave equations given short pulse data have been established for the relativistic membrane equations by Wang and Wei in \cite{Wang}, the isentropic and irrotational Euler equations for  Chaplygin fluids by Ding--Xin--Yin \cite{Ding3}, and a null form nonlinearity by the first author and Ding--Xu in \cite{DDX}. 

In the case where  $\mainfield$ is allowed to be complex valued, Shao--Wei--Zhang \cite{ShWeZh} have shown that the defocusing equation \eqref{intro-eq} admits finite time blow-up solutions arising from
smooth initial data. This extended ideas from the breakthrough  \cite{FRRS}, which proved the existence of  blow-up solutions for the energy-supercritical defocusing nonlinear Schr\"odinger, to the setting of the defocusing equation \eqref{intro-eq}. 

Finally, we recall some existence results for energy-supercritical wave-type equations with large data. 
In \cite{LiDuke}, Li developed an original non-local energy bootstrap strategy and established large data global well-posedness of hedgehog solutions for the Skyrme model which is an energy-supercritical problem. More recently, Wang in \cite{WangXuecheng} proved global existence for the supercritical Vlasov-Maxwell system with cylindrical large data.

\subsection{Main result and the proof}
The precise version of our main theorem is as follows. 
\begin{theorem}\label{thm:main2}
    Let $p>5$ and consider the Cauchy problem \eqref{intro-eq} with initial data at $t=t_0$ of the form 
    \begin{align*}
 (\Phi_0, \Phi_1)
 =
 (u_0, u_1) + (w_0, w_1) \end{align*}
 where 
     \begin{align*} w_0(r, \omega)= \epsilon^{1\over 2} h_0(\tfrac{1-r}{\epsilon}, \omega), \quad w_1(r, \omega)= \epsilon^{-{1\over 2}} h_1(\tfrac{1-r}{\epsilon}, \omega)\,.
\end{align*}
Then, there exists an $\epsilon_0 >0$ such that for all $0<\epsilon<\epsilon_0$ and for all functions $u_0(x), u_1(x)$ satisfying for $k=0,1,2$:
    \begin{align}
       \sum_{|I|\leq 6-|J|, \, |J|= k}&\| \langle x\rangle^{|I|} \nabla^{I}\nabla^{J} u_0 \| \notag
        \\& +\sum_{|I|\leq 5-\max\{|J|-1,0\}, \, |J|= k}\| \langle x\rangle^{|I|+\max\{1-|J|,0\}} \nabla^{I}\nabla^{\max\{|J|-1,0\}} u_1 \|
        \leq
        \epsilon^{k-1} \,, \label{eq:ID-disp}
    \end{align}
    and all functions $h_0(x), h_1(x)$ supported in $\{ 1\leq r \leq 2 \}$ satisfying 
    \begin{subequations}\label{eq:ID-shortp-intro}
     \begin{align}
        \sum_{|I|\leq 6}|\nabla^I h_0(x)| 
        + \sum_{|J|\leq 5}|\nabla^J h_1(x)|
        &\leq 1,\label{eq:ID-shortp1}
        \\
        |w_1 + \partial_r w_0| 
        &\leq \epsilon^{1\over 2}, \label{eq:ID-shortp2}
    \end{align}       
    \end{subequations}
    there exists a classical global solution to the Cauchy problem \eqref{intro-eq} satisfying 
    \begin{align}
        |\Phi(t,x)|&\lesssim \epsilon^{-1}\langle t+r\rangle^{-1}\langle t-r\rangle^{-1/2}.
     \end{align}
\end{theorem}

\begin{remark}
Note that in Theorem \ref{thm:main2} we do not need to assume that $L\mainfield|_{t=t_0}$ is small, where $L$ is the null derivative that is tangential to the forward Minkowski lightcone. We do however need to assume in \eqref{eq:ID-shortp2}  that the $L$ derivative of the short pulse component $(w_0, w_1)$ is small. 
\end{remark}

\begin{remark}
Note that the critical Lebesgue norm for \eqref{intro-eq} is $\mainfield_0\in L^{p_c}$ with $p_c = 3(p-1)/2$. While our data are small in this norm, and certain smallness conditions arise naturally in the short-pulse method,  our data are
large at the level of energy and its Sobolev norms cannot all be made uniformly small using the scaling invariance of the equation.  
\end{remark}
\paragraph{\textbf{Outline and novel ideas in the proof}}
The starting point in our proof of Theorem \ref{thm:main2} is to analyse  the PDE \eqref{eq:model-sup} through two auxiliary PDEs which arise from the decomposed nature of our initial data. \smallskip

\textit{Part I: Global existence for dispersed data.}
We first prove the global existence of large data solutions to the auxiliary PDE
\begin{equation}\begin{split}\label{eq:intro-dispeq}
\psi_{tt} - \Delta \psi &= \pm |\psi|^{p-1} \psi,
\\
(\psi, \partial_t \psi)(t_0)  &= (u_0, u_1),
\end{split}\end{equation}
with dispersed initial data in the sense of \eqref{eq:ID-disp}. We use a standard continuity argument relying on the natural and conformal $L^2(\mathbb{R}^3)$-based energy functionals:
\[
\mathcal{E}(t, \psi) = \frac12 \big(\| \partial_t \psi\|^2+ \| \partial_{\vec{x}} \psi\|^2\big), \quad 
\mathcal{E}_{con}(t, \psi)= \frac12 \big( \| L_0 \psi\|^2 + \sum_a\|L_a \psi \|^2 + \| \Omega \psi\|^2 + \| \psi\|^2 \big)\,.
\]
In the above  $L_a$ denotes the Lorentz boosts, $\Omega$ the rotations and $L_0$ denotes the conformal Killing vector field $L_0$, also called the scaling vector field. We use $\Gamma$ to denote the full Lorentz group of Minkowski Killing vector fields together with $L_0$. Precise notation is given in Section \ref{preliminaries}.

Due to the initial data assumptions in \eqref{eq:ID-disp}, at lowest order the solution behaves as $\|\psi\| \sim \epsilon^{-1}$, while higher derivatives of the data lead to increasing smallness: $\|\partial \psi\| \sim 1,\, \| \partial^{\geq 2} \psi \|\sim \epsilon$. This poses challenges to controlling the nonlinearity $|\psi|^{p-1} \psi$ where no derivatives appear. For this reason, we use a bootstrap which involves both the natural and conformal energies. 
For our pointwise estimates (at lowest order), a naive application of the Klainerman-Sobolev inequality to estimate $|\psi|$ requires control on $\mathcal{E}_{con}(t, \Gamma^{\leq 2}\psi)$, and all together produces decay like $|\psi| \sim \epsilon^{-1} t^{-1}$.   However by using the fundamental theorem of calculus, and the decay estimate derived from $\mathcal{E}(t, \Gamma^{\leq 2}\psi)$, we interpolate our decay estimates together to gain either a smallness coefficient $|\psi|\lesssim \epsilon^{\frac12-\delta}$ or a time decay factor $|\psi| \lesssim t^{-\frac12+\delta}$. Together with the Hardy inequality, this is sufficient to close the bootstrap assumption provided $\delta \ll p-5$.
\medskip

\textit{Part II: Global existence for short-pulse data.}
We next treat the second auxiliary PDE
\begin{equation}\begin{split}\label{eq:intro-shortpeq}
\phi_{tt} - \Delta \phi &= \pm\Big(|\psi+\phi|^{p-1} \phi + (|\psi+\phi|^{p-1} - |\psi|^{p-1})\psi \Big) =: \pm \, \mathcal{N},
\\
(\phi, \partial_t \phi)(t_0)  &= (w_0, w_1),
\end{split}\end{equation}
with short-pulse initial data in the sense of \eqref{eq:ID-shortp-intro}. Due to the initial data assumptions, the hierarchy is now reversed in the sense that higher derivatives generate largeness:
\[
	\|\phi\| \lesssim \epsilon, \quad \|\partial \phi\| \lesssim 1,\quad\|\partial^2\phi\|\lesssim\epsilon^{-1}.
	\]
To be more precise, smallness is present in certain directions (in particular the angular derivatives  $\Omega \phi$ and the null derivative $L \phi$), but not for full derivatives. For instance as in \cite{Ding3, DDX}, but unlike some other short-pulse works such as \cite{MPY}, we only need to assume control of one $L$ derivative of $w$ in \eqref{eq:shortpulse}.  

As before, one of the main issues is that we do not have smallness on $\phi$ using the Klainerman-Sobolev inequality since the estimate yields only $|\phi|\lesssim\|\Gamma^{\leq2}\phi\|\lesssim\epsilon^{-1}$. This again causes problems to control the nonlinearity.
To address this, we decompose spacetime into a short-time slab of size $O(\epsilon)$ and two global regions (an exterior $\mathcal{D}^{\text{ex}}$ and interior $\mathcal{D}^{\text{in}}$ region) separated by a boundary null hypersurface $t-r=4\epsilon$ denoted by $\mathcal{B}_{4\epsilon}$. See Figure \ref{fig:spacetime} for the precise decomposition.  The standard approach is to exploit the different dispersion mechanisms in these three different regions. \smallskip

\textit{(i) Local in time existence for short-pulse data.}
We first establish local existence on the time interval $[t_0=1, t_1=t_0+2 \epsilon]$. We make bootstrap assumptions for  natural and conformal energies which depend on the number of coordinate derivatives $\partial$  being taken:
\[
\mathcal{E}(t, \partial^I \Omega^J \phi)^{1/2} \sim \epsilon^{-|I|}\,, \quad \mathcal{E}_{con}(t, \partial^I \Omega^J\phi) ^{1/2}\sim \epsilon^{1-|I|} \,, \qquad |I|+|J|\leq 5\,.\]
In contrast to the standard Klainerman--Sobolev inequality (see \eqref{est:global_KS}) which fails to yield any smallness when estimating $\phi$, the smallness of the time interval in the local-in-time slab enables us to derive an improved local Sobolev inequality of the form $|\phi|\lesssim\epsilon^{\frac{1}{2}}\|\partial\Omega^{\leq2}\phi\|\lesssim\epsilon^{\frac{1}{2}}$. 
This refined inequality converts geometric smallness coming from the size of the support, into analytic smallness in the $L^\infty$ norm. 

Turning next to the nonlinearity appearing in \eqref{eq:intro-shortpeq}, we note that
\begin{align*}
\big|  |\psi+\phi|^{p-1} - |\psi|^{p-1} \big|
\lesssim
|\psi|^{p-2}|\phi| + |\phi|^{p-1}.
\end{align*}
A crucial observation is that  at least one factor of $\phi$ is always present in the terms on the RHS. This  ensures that the nonlinear terms inherit the compact support of $\phi$. This  key property facilitates our subsequent estimates by confining the influence of the nonlinear interactions to a bounded spacetime region, thereby avoiding unbounded contributions from the large derivatives discussed earlier.
For higher-order estimates on the nonlinearity, numerous terms arise from the Leibniz rule. For example, when taking one derivative we have
    \begin{align*} |\Gamma \mathcal{N}| &\lesssim |\phi|^{p-1}|\Gamma \phi| + |\phi|^{p-1}|\Gamma \psi| + |\psi|^{p-1}|\Gamma \phi| + |\phi||\psi|^{p-2}|\Gamma \psi|\,.
    \end{align*}
Some care is needed here when $\Gamma = \Omega$. While  $\Omega \phi$ is not large, the angular derivative exhibits unfavorable behavior on the dispersed part since $\|\Omega \psi\|\lesssim \epsilon^{-1}$. 
We counteract this large growth  via the identity $ \Omega\psi\sim r\partial\psi\sim\partial\psi$, where in the last step we used that $r\sim1$ in the local-in-time region. 
By combining the pointwise smallness of $\phi$, the compact support, and analysing the differentiated nonlinearities, we are able to control all such terms and close the bootstrap for the local-in-time evolution.

Finally, at the boundary of the local region $t=t_1$ we identify additional hidden smallness by expressing the wave equation \eqref{eq:intro-shortpeq} in a null frame and integrating in the $u$- or $v$-direction. Specifically, we find  
$|L^{\leq2}\phi|\lesssim \epsilon^{\frac{1}{2}}$ holds for values of $r$ which propagate into the exterior region. While this is one derivative more than we assumed in \eqref{eq:ID-shortp2}, this procedure relies on the fundamental theorem of calculus and the $\epsilon$-size support of the solution. Thus we have lost regularity here (i.e. we cannot estimate $L^2 \Gamma^{\geq 1} \phi$) due to this $L^\infty-L^\infty$ argument.  Similarly, we obtain an estimate on $|\underline{L}^{\leq2}\phi|\lesssim \epsilon^{\frac{1}{2}}$ for values of $r$ which  propagate into the interior region. This explicit quantitative smallness then yields suitable initial data for the global analysis.

\medskip

\textit{(ii) Global existence in the exterior region.} We next prove global existence in the exterior region $\mathcal{D}^{\text{ex}}$ for the PDE \eqref{eq:intro-shortpeq} with initial data posed at $t=t_1$. Due to the compact support of short-pulse data, the exterior region consists of the domain $\mathcal{D}^{\text{ex}}=\{(t,x): t\ge t_1,\ t-5\epsilon\le r\le t-\epsilon\}$ which we foliate using constant $t-$time slices $\Sigma^{ex}_{t}$. We rely on natural energies to control the solution and the results from \textit{(i)} inform our bootstrap which again depends on the number of coordinate derivatives $\partial$:
\[ \|\partial \partial^I \Omega^J \phi\|_{L^2(\Sigma^{ex}_{t})} \lesssim \epsilon^{-|I|}    , \qquad |I|+|J|\leq 5\,.
\]
For brevity here, we do not state also the bootstrap on the boundary  $\mathcal{B}_{4\epsilon}$.
We can gain an additional smallness factor on the field in $L^2$ by using a standard Hardy inequality  involving a weight depending on the distance from the lightcone
\[
\|(t-r)^{-1} \phi \|_{L^2(\Sigma^{ex}_{t})} \lesssim \|\partial_r \phi\|_{L^2(\Sigma^{ex}_{t})}\,.
\]
Since  $t-r \sim \epsilon$ in the exterior region,  we have $\|\phi\|_{L^2(\Sigma^{ex}_{t})} \lesssim \epsilon \| \partial_r \phi\|_{L^2(\Sigma^{ex}_{t})} \lesssim \epsilon$.
Notably, nonlinear terms without derivatives can be directly controlled by this inequality.
For the $L^\infty$ norm of $\phi$, we follow the same strategy as in the local case, obtaining from a standard Sobolev inequality the estimate $|\phi| \lesssim r^{-1} \epsilon^{\frac{1}{2}}$. 

Turning next to the nonlinear term $\mathcal{N}$, similar problematic terms as in \textit{(i)} occur.
For instance, using the estimates for the dispersed component, we obtain 
\[
\|\Omega^J\psi\|_{L^2(\Sigma^{ex}_{t})}\lesssim\epsilon^{-1},\qquad
|\Omega^J\psi|\lesssim\sum_{|J_1|=|J|-1}|r\partial\Omega^{J_1}\psi|.
\]
Unlike \textit{(i)} where we  used $r \sim 1$, we inevitably incur a loss here -- either a loss in  $\epsilon$ or a loss in the decay rate.
Therefore,  we must carefully balance these two types of losses to close the exterior bootstrap.

    Finally, as input for the interior analysis, we  derive additional estimates including conformal energies and boundary contributions arising from the interface $\mathcal{B}_{4\epsilon}$ with the exterior region, together with refined $L^\infty$ estimates.
\medskip

\textit{(iii) Global existence in the interior region.}
Finally we prove global existence  in the interior region $\mathcal{D}^{in}=\{(t,x): t\geq t_1,\ 0\leq r\leq t-4\epsilon\}$ which is significantly more challenging.  We employ the hyperboloidal vector field approach using a hierarchy of energy estimates along hyperboloidal hypersurfaces $\mathcal{H}_\tau$. In particular, we make use of both the standard hyperboloidal energy $\mathcal{E}^{in}(\tau,\varphi)$ and a conformal energy $\mathcal{E}_{\text{con}}^{in}(\tau,\varphi)$.
A key difficulty lies in obtaining suitable pointwise control without losing too many derivatives. In  the region $r \geq t/8$, the classical Klainerman--Sobolev inequality reads (see the precise estimate \eqref{est:KS_in3})
$$|\phi| \lesssim  \tau^{-1} r^{-\frac12} \| (\tau/t) \Omega^{\leq 2} \phi\|^{1/2}_{L^2(\mathcal{H}_\tau)} \| (\tau/t) \Omega^{\leq 2} Z \phi\|^{1/2}_{L^2(\mathcal{H}_\tau)} + \text{endpoint data}.$$
Note the last factor requires three derivatives on $\phi$. 
However, to control $\| (\tau/t)\Omega^2 Z \phi\|$
one is forced to commute the equation twice, which produces nonlinear terms of the form $|\phi|^{p-2}\Omega\phi\,\Omega\phi$. Controlling such terms requires $L^\infty$ bounds for $\Omega\phi$, and hence higher-order Sobolev norms such as $\|(\tau/t)Z\Omega^{3}\phi\|$. This in turn necessitates bootstrap assumptions at the level of three derivatives. As a consequence, certain energy terms cannot be closed; for instance, we cannot control a boundary term of the form
\[
\int_{B_{4\epsilon}} |L_a \Omega^3 \phi|\,dx,\qquad\text{where}\ \ |L_a \Omega^3 \phi| \lesssim r^{-1} \epsilon^{\frac{1}{2}} \|\partial_r\Omega^{\le 2}L_a \Omega^3 \phi\|_{L^2(\Sigma_{t}^{ex})},
\]
without invoking estimates for up to six derivatives, which is not compatible with the range $p>5$. 

To overcome this difficulty, we introduce a different strategy by combining $L^\infty$ bounds with $L^q$ estimates. Our crucial observation is that in the region $r \geq t/8$, we can prove an $L^q-L^2$ estimate that requires only two derivatives on the right hand side (see \eqref{est:KS_in2} for the precise statement), 
\begin{align*}
   \|\phi\|_{L^q(\mathcal{H}_{\tau}\cap\{r\geq t/8\})}
            		&\lesssim \tau^{-1}r^{-\frac{1}{2}+\frac{3}{q}}
            		\big\|(\tau/t) \Omega^{\leq 1}\phi\big\|_{L^2(\mathcal{H}_\tau)}^{\frac{1}{2}}
            		\big\|(\tau/t)Z\Omega^{\leq 1}\phi\big\|_{L^2(\mathcal{H}_\tau)}^{\frac{1}{2}}+\text{endpoint data}, \quad q>6.
\end{align*}
The condition $q>6$ is used to rewrite an integral involving  $L^q(\mathbb{S}^2)$ into one involving $L^2(\mathcal{H}_\tau)$, which is required after we invoke the Sobolev embedding $H^1(\mathbb{S}^2)\hookrightarrow L^q(\mathbb{S}^2)$. In turn, the condition $q>6$ is used in our proof to get enough decay when applying the above estimate.

Our $L^q$ estimate allows us to lower the regularity level in the bootstrap assumptions: instead of controlling three derivatives, it suffices to work at essentially one derivative above the energy level. By contrast other works for instance \cite{Ding3, MPY} require far more derivatives than we do here.  
In particular, we impose bootstrap bounds of the form
\begin{equation*}
\begin{aligned}
\mathcal{E}^{in}(\tau, \Omega^{\leq 1} \phi)^{\frac{1}{2}}+ \mathcal{E}^{in}_{con}(\tau, \Omega^{\leq 1}\phi)^{\frac{1}{2}} &\lesssim \epsilon^{\frac{1}{2}}, \quad
\mathcal{E}^{in}(\tau, \partial \phi)^{\frac{1}{2}}+\mathcal{E}^{in}(\tau, L_a \phi)^{\frac{1}{2}} + \mathcal{E}^{in}_{con}(\tau, L_a \phi)^{\frac{1}{2}} \lesssim 1. 
\end{aligned}
\end{equation*}
In this bootstrap it is also crucial that we can reduce the number of derivatives in order to preserve smallness in $\epsilon$. Indeed as already seen in \textit{(i, ii)} above, higher-order derivatives tend to lose any smallness coming from the initial data. For example, refined pointwise estimates at time $t_1$ do not yield any $\epsilon$-gain at the level of three derivatives, e.g.
\[
\|\partial \Omega^3 \phi\|_{L^2(\Sigma_{t_1}\cap\mathcal{D}^{in})} \lesssim O(1),
\]
which shows that no smallness is available at this level. In contrast, the use of $L^q$ estimates allows us to avoid such losses.
With these ingredients, we are able to close the bootstrap argument in the interior region. The proof relies on a delicate combination of the hyperboloidal energy estimates, the mixed $L^\infty$ and $L^q$ control, and the previously established results in the local and exterior regions, which together yield global existence in $\mathcal{D}^{in}$.



%

\subsection{Organization} In Section \ref{preliminaries} we present some preliminary notation and standard results from functional analysis. In Section \ref{dispersed} we prove global existence for the decoupled problem \eqref{eq:intro-dispeq} with dispersed data. In Section \ref{sp:local} we turn to the second decoupled problem \eqref{eq:intro-shortpeq} and prove local in time existence, followed by global existence in the exterior and interior regions in Sections \ref{sp:exterior} and \ref{sec:interior} respectively. Finally in Section \ref{sec:classical} we upgrade our global solution to a global classical one. 

\subsection*{Acknowledgement}
The author Shijie Dong would like to acknowledge the support from the National Natural Science Foundation of China (Grant Nos. 12401280 and 12431007) and Guangdong Basic and Applied Basic Research Foundation (Grant Nos. 2025A1515012652 and 2023A1515110944).
The author Jingya Zhao is grateful to Murray Edwards College, University of Cambridge, for hospitality as work was being conducted on this project. 

\section{Preliminaries}\label{preliminaries}

\subsection{Notation}
Our problem is posed on the Minkowski spacetime $\mathbb{R}^{1+3}$ equipped with the metric $\eta = \text{diag}(-1, 1, 1, 1)$. We denote by $\Box=-\partial_t\partial_t+\Delta$ the d'Alembert operator. We use Roman letters to denote spatial indices $a, b, \cdots \in \{1, 2, 3 \}$, and Greek letters to denote spacetime indices $\alpha, \beta, \cdots \in \{0, 1, 2, 3\}$. The indices are raised or lowered by the metric $\eta$, and the Einstein summation convention for repeated upper and lower indices is adopted.  For a spacetime point $(t, x^1, x^2, x^3)$,  $x^0$ denotes the time variable $t$, and $r$ denotes the spatial radius $\sqrt{(x^1)^2 + (x^2)^2 + (x^3)^2}$. 

We rely on the standard Minkowski vector fields: the translations $\{ \partial_\alpha = \partial_{x^\alpha}\}$, the rotations $\{\Omega_{ab} = x_a\partial_b - x_b\partial_a: a<b\} $, the Lorentz boosts $ \{L_a = t\partial_a + x_a\partial_t\}$ and the scaling vector field $L_0 = t\partial_t + x^a\partial_a$. We collect the vector fields as
\begin{equation*}
    \Gamma=(\Gamma_1,\cdots,\Gamma_{11})=(\partial,\Omega,L_1,L_2,L_3,L_0),
\end{equation*}
where we use the shorthand notation $\partial=(\partial_t,\nabla)=(\partial_t,\partial_1,\partial_2,\partial_3)$ and $\Omega=(\Omega_{12},\Omega_{13},\Omega_{23})$. We sometimes denote $Z=\{\partial_t, L_1, L_2,L_3\}.$

It is important to express these vector fields with respect to the natural null frame defined on Minkowski. 
First, we denote the null coordinates $v=t+r, u=t-r$, which leads to 
\begin{align*}
    L = \partial_t + \partial_r =2\partial_v,
    \qquad
    \underline{L} = \partial_t - \partial_r =2\partial_u,
\end{align*}
where $\partial_r=\omega^a\partial_a$ with $\omega_a=x_a/r$.
We also introduce a notion of `good' derivatives 
\begin{equation*}
    G_a=\omega_a\partial_t+\partial_a,\quad a=1,2,3,
\end{equation*}
and set $G=(G_1,G_2,G_3)$. 
These definitions lead to the following identities amongst the vector fields.
\begin{equation}\label{eq:VF-identities}
    \begin{aligned}
        &L=\omega^a G_a=\frac{L_0+\omega^aL_a}{t+r},\quad\partial_t=\frac{1}{2}(L+\underline{L}),\quad\partial_a=\frac{\omega_a}{2}(L-\underline{L})-\frac{\omega^b}{r}\Omega_{ab},\\
        &G_a=\frac{\omega_a}{t+r}L_0+\frac{1}{t}L_a-\frac{\omega_ax^b}{t(t+r)}L_b
        =\omega_aL-\frac{\omega^b}{r}\Omega_{ab},\\
        &L_a=\frac{\omega_a}{2}((t+r)L-(t-r)\underline{L})-\frac{t}{r}\omega^b\Omega_{ab}.
    \end{aligned}
\end{equation}

We use the standard notation
\begin{equation*}
    |\Gamma f| = \Big(\sum_{k=1}^{11} |\Gamma_k f|^2\Big)^{\frac{1}{2}},
\end{equation*}
and for any multi-index $I=(I_1,\cdots,I_{11})\in\mathbb{N}^{11}$, we write $\Gamma^I=\Gamma^{I_1}_1\cdots\Gamma^{I_{11}}_{11}$ and $|I|=\sum_{k=1}^{11}I_{k}$. We occasionally use the shorthand notation $\Gamma^{\leq k}$ to denote a sum of $\Gamma^I$ with $|I|\leq k$. 
Finally, we use the Japanese bracket $\langle b\rangle=(1+|b|^{2})^{\frac{1}{2}}$ for a scalar or vector $b$, and denote by $\|\cdot\|$ the $L^2$ norm. We write $A \lesssim B$ to indicate $A\leq C_0 B$ with $C_0$ a universal constant. We write $A\sim B$ if $A\lesssim B$ and $B\lesssim A$.

Finally, we will  need to write numerous summations involving terms such as $|I_1|+\ldots +|I_5|$. Since this becomes cumbersome to read, we introduce the following simplified (and slightly abused) notation
    \[ \sum_{\substack{|I_1|+\cdots+|I_k|=|I|,\\ \text{Another condition}}} = \sum_{I_{1\to k} = I} \sum_{\substack{\text{Another}\\\text{ condition}}}.\]
    Furthermore, we occasionally write $\partial^I \Omega^J$ as $\Gamma^K$ when it is not  necessary to emphasise the value of $I$, and associate $K=(I, J)$ with $|K|=|I|+|J|$.

\subsection{Standard results} We first collect together several classical results from functional analysis which we will make use of extensively in our proof. 
\begin{lemma}\label{lem:commutator1} 
    The following commutation identities hold
    \begin{align*}
        [-\Box,\Gamma_k]&=-c_k\Box,\quad c_k=\left\{\begin{aligned}
            &0,\quad k=1,\cdots,10,\\
            &2,\quad k=11,
        \end{aligned}\right. \\
    [L_0,\Omega_{ab}]&=0,\quad \mbox{for}\quad 1\leq a<b\leq3.
    \end{align*}
\end{lemma}

\begin{lemma}[\cite{Sogge}]\label{lem:commutator2}
        For any smooth function $\varphi=\varphi(t,x)$,  the following commutator terms can be estimated as
		\begin{equation}\label{est:commutators}
                \sum_{\alpha=0}^{3}\left|[\partial_\alpha,\Gamma^I]\varphi\right|\lesssim\sum_{|J|<|I|}\sum_{\beta=0}^{3}|\partial_\beta\Gamma^J\varphi|,
                \qquad
                \sum_{a=1}^{3}\left|[L_a,\Omega^J]\varphi\right|\lesssim\sum_{|J|<|I|}\sum_{b=1}^{3}|L_b\Omega^J\varphi|.
		\end{equation}
	\end{lemma}

\begin{lemma}[Klainerman-Sobolev inequality {\cite[Theorem 1.3]{Sogge}}] \label{lem:global_KS}
    Let $\varphi$ be a sufficiently smooth function that vanishes when $|x|$ is large. Then 
    \begin{equation}\label{est:global_KS}
        |\varphi(t,x)|\lesssim\langle t+r\rangle^{-1}\langle t-r\rangle^{-\frac{1}{2}}\sum_{|I|\leq2}\|\Gamma^I\varphi(t,x)\|.
    \end{equation}
\end{lemma}

\begin{lemma}[Sobolev embedding theorem on a ball \cite{LiZhou2017}]\label{lem:Sobo_embedding}
    Let $\rho>0$ and $B(\rho)\subset\mathbb{R}^3$ denote the open ball of radius $\rho$. There exists a constant $C>0$, independent of $\rho$, such that for all $f\in H^1(B(\rho))$,
    \begin{equation}\label{est:Sobo_embedding}
        \|f\|_{L^6(B(\rho))}\leq C\big(\rho^{-1}\|f\|_{L^2(B(\rho))}+\|\nabla f\|_{L^2(B(\rho))}\big).
    \end{equation}
 \end{lemma}

\begin{lemma}[Interpolation in $\mathbb{R}^3$]\label{lem:inter_inequ}
    If $2<p'<6$ and $\theta=\frac{3}{p'}-\frac{1}{2}$, then
    \begin{equation}\label{est:inter_inequ}
        \|f\|_{L^{p'}}\leq \|f\|^\theta\|f\|_{L^6}^{1-\theta}.
    \end{equation}
\end{lemma}


\begin{lemma}[Hardy estimate]\label{lem:hardy}
    Suppose that $f$ is a sufficiently smooth function and $f\in H^1(\mathbb{R}^3)$, then 
    \begin{equation}\label{est:hardy}
        \|r^{-1} f\|\lesssim\|\partial_r f\|.
    \end{equation}
\end{lemma}

\begin{lemma}[Gr\"onwall inequality]\label{lem:Gronwall}
    Let $f(s)$ be a nonnegative function that satisfies the integral inequality
    \begin{equation*}
        f(t)\leq c+\int_{t_0}^{t}(a(s)f(s)+b(s)f^{\frac{1}{2}}(s))\, ds,
    \end{equation*}
    where $a(t)$ and $b(t)$ are continuous nonnegative functions and $c\geq0$. Then we have
    \begin{equation}\label{est:Gronwall}
        f(t)\leq\Big(c^{\frac{1}{2}}\exp\Big(\frac{1}{2}\int_{t_0}^{t}a(s)\, ds\Big)+\frac{1}{2}\int_{t_0}^{t}b(s)\exp\Big(\frac{1}{2}\int_{s}^{t}a(\tilde{s})\, d\tilde{s}\Big)\, ds\Big)^{2}.
    \end{equation}
\end{lemma}

\begin{definition}
    The standard natural and conformal energy functionals on constant $t$ slices read
        \[
\mathcal{E}(t, \varphi) = \frac12 \big(\| \partial_t \varphi\|^2+ \| \partial_{\vec{x}} \varphi\|^2\big), \quad 
\mathcal{E}_{con}(t, \varphi)= \frac12 \big( \| L_0 \varphi\|^2 + \sum_{a}\|L_a \varphi \|^2 + \| \Omega \varphi\|^2 + \| \varphi\|^2 \big)\,.
\]
\end{definition}

\begin{lemma}[Energy estimates for the wave equation \cite{Sogge, AlinhacBook}]\label{lem:energy_general}
    Let $\varphi=\varphi(t,x)$ be the solution to the Cauchy problem
    \begin{equation*}
        \left\{\begin{aligned}
            -\Box \varphi(t,x)&=G(t,x),\\
            (\varphi,\partial_t\varphi)|_{t=t_0}&=(\varphi_0,\varphi_1),
        \end{aligned}\right.
    \end{equation*}
    where $G(t,x)$ is a sufficiently smooth function. 
Then the following natural and conformal energy estimates hold:
    \begin{align}\label{est:energy_natural}
            \mathcal{E}^{\frac{1}{2}}(t,\varphi)&\lesssim\mathcal{E}^{\frac{1}{2}}(t_0,\varphi)+\int_{t_0}^{t}\|G(t,x)\|\, ds\,,
    \\ \label{est:energy_conformal}
            \mathcal{E}_{con}^{\frac{1}{2}}(t,\varphi)&\lesssim\mathcal{E}_{con}^{\frac{1}{2}}(t_0,\varphi)+\int_{t_0}^{t}\|\langle s+r\rangle G(t,x)\|\, ds.
        \end{align}
\end{lemma}

\section{Global existence for dispersed data}\label{dispersed}

\subsection{Global existence and decay}
We first consider the decoupled problem
\begin{equation}\label{eq:model-disp}
\begin{split}
    -\Box \psi =\pm |\psi|^{p-1} \psi,
    \\
    (\psi, \partial_t \psi)(t_0) = (u_0, u_1),
    \end{split}
\end{equation}
where the initial data will be \textit{dispersed} in the sense that the $L^2$ norm is large, i.e. $\|u_0\| \sim \epsilon^{-1}$, while higher derivatives of the data lead to increasing smallness, e.g. $\|\partial u_0\| \sim 1,\, \| \partial^2 u_0\|\sim \epsilon$.
Our goal in this section is to prove the following global existence result. 

\begin{theorem}\label{thm:main-disp}
    The Cauchy problem \eqref{eq:model-disp} with  initial data satisfying
        \begin{align}\label{eq:ID-disp-w}
        \sum_{|I|\leq 6-|J|, \, |J|= k}&\| \langle x\rangle^{|I|} \nabla^{I}\nabla^{J} u_0 \|
        \\& +\sum_{|I|\leq 5-\max\{|J|-1,0\}, \, |J|= k}\| \langle x\rangle^{|I|+\max\{1-|J|,0\}} \nabla^{I}\nabla^{\max\{|J|-1,0\}} u_1 \|
        \leq
        \epsilon^{k-1}, \quad k=0,1,2 \,,
    \end{align}
    (i.e. dispersed in the sense of \eqref{eq:ID-disp})
    admits a global solution.
\end{theorem}

Based on the boundedness and smallness assumptions in \eqref{eq:ID-disp-w}, we make the following bootstrap assumptions on $[t_0, T)$:
\begin{subequations}\label{eq:BA-disp}
\begin{align}
    \mathcal{E}_{con}^{1\over 2} (t, \Gamma^I \psi) 
    &\leq C_1 \epsilon^{-1},
    \qquad
    &|I|\leq 5, \label{eq:BA-disp01}
    \\
    \mathcal{E}^{1\over 2} (t, \Gamma^I \psi)
    &\leq C_1,
    \qquad
    &|I|\leq 5, \label{eq:BA-disp02}
    \\
    \mathcal{E}^{1\over 2} (t, \partial \Gamma^I \psi)
    &\leq C_1 \epsilon,
    \qquad
    &|I|\leq 4. \label{eq:BA-disp03}
\end{align}
\end{subequations}

In the above, $C_1 \gg 1$ is to be determined, $\epsilon\ll 1$ is the size of the data such that $C_1 < \epsilon^{-1}$, and 
$$
T:= \sup \{ s\in (t_0, +\infty): \eqref{eq:BA-disp} \text{ hold}  \}.
$$

In our next proposition we prove pointwise estimates using the bootstrap assumptions. We point out the key estimate for $|\Gamma^{\leq 2} \psi|$ in \eqref{eq:disp-psi01}: a naive application of the Klainerman-Sobolev inequality \eqref{est:global_KS} combined with \eqref{eq:BA-disp01} would suggest $|\Gamma^{\leq 4} \psi| \sim \epsilon^{-1}$. However, we gain a smallness coefficient of $\epsilon^{\frac12-}$ by working at lower order and applying an interpolation argument. 

\begin{lemma}[Preliminary estimates] \label{lem:disp-psi}
    Under the bootstrap assumptions \eqref{eq:BA-disp}, the following estimates are valid for $t\in [t_0, T)$.
    \begin{itemize}
        \item Estimates for $\psi$.
        \begin{equation}
    \begin{aligned}\label{eq:disp-psi01}
        |\Gamma^I \psi|
        &\lesssim
        C_1 \epsilon^{{1\over 2} - \delta},
        \qquad
        &|I|\leq 2,
        \\
        |\Gamma^I \psi|
        &\lesssim
        C_1 \langle t+r\rangle^{-{1\over 2} + \delta},
        \qquad
        &|I|\leq 3.
    \end{aligned}
    \end{equation}

 \item Estimates for $\partial\psi$.
        \begin{equation}\label{eq:disp-d-psi01}
    \begin{aligned}
        |\partial\Gamma^I \psi|
        &\lesssim
        C_1 \epsilon \langle t+r\rangle^{-{1\over 2} + \delta},
        \qquad
        &|I|\leq 2,
        \\
        |\partial\Gamma^I \psi|
        &\lesssim
        C_1 \langle t+r\rangle^{-1} \langle t-r\rangle^{-{1\over 2}},
        \qquad
        &|I|\leq 3.
    \end{aligned}
    \end{equation}

\item Estimates for $\partial\partial\psi$.
\begin{equation}
    \begin{aligned}\label{eq:disp-dd-psi01}
        |\partial \partial \Gamma^I \psi|
        &\lesssim
        C_1 \epsilon\langle t+r\rangle^{-1} \langle t-r\rangle^{-{1\over 2}},
        \qquad
        &|I|\leq 2.
    \end{aligned}
\end{equation}
    \end{itemize}
\end{lemma}

\begin{proof}
By the energy bounds listed in \eqref{eq:BA-disp}, the Klainerman-Sobolev inequality \eqref{est:global_KS}, and the commutator estimates in \eqref{est:commutators}, we get
\begin{align}
    |\Gamma^I \psi|
    &\lesssim
    C_1 \epsilon^{-1} \langle t+r\rangle^{-1} \langle t-r\rangle^{-{1\over 2}},
    \qquad
    |I|\leq 4, \label{eq:disp-0001}
    \\
    |\partial\Gamma^I \psi|
    &\lesssim
    C_1 \langle t+r\rangle^{-1} \langle t-r\rangle^{-{1\over 2}},
    \qquad
    |I|\leq 3, \label{eq:disp-0002}
    \\
    |\partial\partial\Gamma^I \psi|
    &\lesssim
    C_1 \epsilon \langle t+r\rangle^{-1} \langle t-r\rangle^{-{1\over 2}},
    \qquad
    |I|\leq 2. \label{eq:disp-0003}
\end{align}
These prove the second equation in \eqref{eq:disp-d-psi01} and the bound \eqref{eq:disp-dd-psi01}.

Next, by the fundamental theorem of calculus and the bounds in \eqref{eq:disp-0002}, we derive for $|J|\leq 3$ that
\begin{align*}
    |\Gamma^J \psi|(t, x)
    \lesssim
    \int_{|x|}^{+\infty} |\partial_r \Gamma^J \psi | \, \d \rho
    \lesssim
    C_1 \langle t+|x|\rangle^{-{1\over 2} + {\delta\over 2}},
\end{align*}
which proves the second inequality stated in \eqref{eq:disp-psi01}.
Similarly, by \eqref{eq:disp-0003}, we infer for $|J|\leq 2$ that
\begin{align*}
    |\partial\Gamma^J \psi|(t, x)
    \lesssim
    \int_{|x|}^{+\infty} |\partial_r \partial \Gamma^J \psi | \, \d \rho
    \lesssim
    C_1 \epsilon \langle t+|x|\rangle^{-{1\over 2} + {\delta\over 2}},
\end{align*}
which verifies the first inequality stated in \eqref{eq:disp-d-psi01}.
We then interpolate \eqref{eq:disp-d-psi01} and \eqref{eq:disp-0002} in order to gain a coefficient involving $\epsilon^{1/2-}$. For $|K|\leq 2$, we find that
\begin{align*}
    |\partial \Gamma^K \psi|
    \lesssim
    C_1\epsilon^{{1\over 2}-\delta} \langle t+r\rangle^{(-{1\over 2}+ {\delta })({1\over 2}-\delta)-({1\over 2}+\delta)} \langle t-r\rangle^{-{1\over 2}({1\over 2}+\delta)},
\end{align*}
which further leads to
\begin{align*}
    |\Gamma^K \psi|(t, x)
    \lesssim
    \int_{|x|}^{+\infty} |\partial_r \Gamma^K \psi | \, \d \rho
    \lesssim
    C_1 \epsilon^{{1\over 2}-\delta}.
\end{align*}
This confirms the first bound in \eqref{eq:disp-psi01}.
\end{proof}

We next combine the estimates from Lemma \ref{lem:disp-psi} with the Hardy inequality from Lemma \ref{lem:hardy} to deduce bounds on the nonlinearities of the wave equation.

\begin{lemma}[Bounds on the nonlinearities]\label{lem:nonlinearities-psi}
    Under the assumptions \eqref{eq:BA-disp}, we have 
    \begin{align}
        \big\|\langle t+r\rangle\Gamma^I (|\psi|^{p-1} \psi) \big\|
        &\lesssim
        C_1^p \epsilon^{-1+\delta} \langle t\rangle^{-1-\delta},
        \qquad
        &|I|\leq 5, \label{eq:disp-002}
        \\
        \big\|\Gamma^I (|\psi|^{p-1} \psi) \big\|
        &\lesssim
        C_1^p \epsilon^\delta \langle t\rangle^{-1-\delta} ,
        \qquad
        &|I|\leq 5, \label{eq:disp-001}
        \\
        \big\|\partial\Gamma^I (|\psi|^{p-1} \psi) \big\|
        &\lesssim
        C_1^p \epsilon^{1+\delta} \langle t\rangle^{-1-\delta},
        \qquad
        &|I|\leq 4. \label{eq:disp-003}
    \end{align}
\end{lemma}
\begin{proof}
    We first note the simple but useful relations
    \begin{align}\label{eq:psi-Leib}
        \big|\Gamma^I \big( |\psi|^{p-1} \psi \big) \big|
        &\lesssim
        \sum_{I_{1\to 5} \leq |I|} |\psi|^{p-5} |\Gamma^{I_1}\psi| |\Gamma^{I_2}\psi| \cdots |\Gamma^{I_5}\psi|,
    \\ \label{eq:d-psi-Leib}
        \big|\partial\Gamma^I \big( |\psi|^{p-1} \psi \big) \big|
       & \lesssim
        \sum_{I_{1\to 5}\leq |I|} |\psi|^{p-5} |\partial\Gamma^{I_1}\psi| |\Gamma^{I_2}\psi| \cdots |\Gamma^{I_5}\psi|.
    \end{align}

$\bullet$ We first prove \eqref{eq:disp-001}.
Our starting point is taking \eqref{eq:psi-Leib} and inserting a factor of $r^{-1}$ onto the term with the most derivatives. 
\begin{align*}
    \|\Gamma^I \big( |\psi|^{p-1} \psi \big)\|
    \lesssim
    \sum_{\substack{I_{1\to 5}\leq |I|\\ |I_1|\geq \max\{|I_2|, \ldots, |I_5|\}}} \|\psi\|_{L^\infty}^{p-5} \|r^{-1}\Gamma^{I_1}\psi\| \|r^{1\over 4} \Gamma^{I_2}\psi\|_{L^\infty}  \cdots \|r^{1\over 4}\Gamma^{I_5}\psi\|_{L^\infty}.
\end{align*}
Applying the Hardy inequality \eqref{lem:hardy}, the energy bounds  \eqref{eq:BA-disp02} and the pointwise bounds \eqref{eq:disp-psi01} gives
\begin{align*}
    \|\Gamma^I \big( |\psi|^{p-1} \psi \big)\|
    &\lesssim
    \sum_{\substack{I_{1\to 5}\leq |I|\\ |I_1|\geq \max\{|I_2|, \ldots, |I_5|\}}} \|\psi\|_{L^\infty}^{p-5} \|\partial\Gamma^{I_1}\psi\| \|r^{1\over 4} \Gamma^{I_2}\psi\|_{L^\infty} \cdots \|r^{1\over 4}\Gamma^{I_5}\psi\|_{L^\infty}
    \\
    &\lesssim
     C_1^p \epsilon^{({1\over 2}-\delta){p-5\over 2}} \langle t\rangle^{-{p-1\over 4} + {p+3 \over 2}\delta}.
\end{align*}
Therefore, by assuming $\delta\ll {p-5}$, we obtain \eqref{eq:disp-001}.

$\bullet$ Next, we prove \eqref{eq:disp-002}. We distribute the spacetime weight across four of the factors appearing in \eqref{eq:psi-Leib}. By employing the energy bounds from  \eqref{eq:BA-disp01} and the pointwise bounds in \eqref{eq:disp-psi01}, we infer 
\begin{align*}
    \|\langle t+r\rangle\Gamma^I \big( |\psi|^{p-1} \psi \big)\|
    &\lesssim
    \sum_{\substack{I_{1\to 5}\leq |I|\\ |I_1|\geq \max\{|I_2|, \ldots, |I_5|\}}} \|\psi\|_{L^\infty}^{p-5} \|\Gamma^{I_1}\psi\| \|\langle t+r\rangle^{1\over 4} \Gamma^{I_2}\psi\|_{L^\infty} \cdots \|\langle t+r\rangle^{1\over 4}\Gamma^{I_5}\psi\|_{L^\infty}
    \\
    &\lesssim
     C_1^p \epsilon^{-1+({1\over 2}-\delta){p-5\over 2}} \langle t\rangle^{1+ (-{1\over 2}+\delta){p+3 \over 2}},
\end{align*}
which again, for small $\delta$, leads us to \eqref{eq:disp-002}.

$\bullet$ To conclude we prove \eqref{eq:disp-003}. Using now \eqref{eq:d-psi-Leib}, we distribute factors of $r$ depending on the number of vector fields falling on the term with the additional $\partial$ derivative:
\begin{align*}
    \|\partial\Gamma^I \big( |\psi|^{p-1} \psi \big)\|
    &\lesssim
    \sum_{I_{1\to 5} \leq |I|} \Big( \sum_{|I_1|\geq 3} \|\psi\|_{L^\infty}^{p-5} \|r^{-1}\partial\Gamma^{I_1}\psi\| \|r^{1\over 4} \Gamma^{I_2}\psi\|_{L^\infty}  \cdots \|r^{1\over 4}\Gamma^{I_5}\psi\|_{L^\infty}
    \\
    &\quad +
    \sum_{|I_1|\leq 2} \|\psi\|_{L^\infty}^{p-5} \|\partial\Gamma^{I_1}\psi\|_{L^\infty} \|r^{-1} \Gamma^{I_2}\psi\| \|r^{1\over 3}\Gamma^{I_3}\psi\|_{L^\infty} \cdots \|r^{1\over 3}\Gamma^{I_5}\psi\|_{L^\infty}\Big) 
    =: \mathcal{G}_1 + \mathcal{G}_2.
\end{align*}
We use the Hardy inequality \eqref{lem:hardy}, energy bounds \eqref{eq:BA-disp03}, and pointwise bounds \eqref{eq:disp-psi01} to show
\begin{align*}
    \mathcal{G}_1
    &\lesssim
    \sum_{\substack{I_{1\to 5}\leq |I|\\ |I_1|\geq 3}} \|\psi\|_{L^\infty}^{p-5} \|\partial\partial\Gamma^{I_1}\psi\| \|r^{1\over 4} \Gamma^{I_2}\psi\|_{L^\infty}  \cdots \|r^{1\over 4}\Gamma^{I_5}\psi\|_{L^\infty}
    \\
    &\lesssim
    C_1^p \epsilon^{1+({1\over 2}-\delta){p-5\over 2}} \langle t\rangle^{-1-({1\over 2}-\delta){p-5\over 2}+4\delta}.
\end{align*}
Similarly, using \eqref{lem:hardy}, \eqref{eq:BA-disp02}, and \eqref{eq:disp-psi01}--\eqref{eq:disp-d-psi01}, we derive
\begin{align*}
    \mathcal{G}_2
    &\lesssim
    \sum_{\substack{I_{1\to 5}\leq |I|\\ |I_1|\leq 2}} \|\psi\|_{L^\infty}^{p-5} \|\partial\Gamma^{I_1}\psi\|_{L^\infty} \|\partial \Gamma^{I_2}\psi\| \|r^{1\over 3}\Gamma^{I_3}\psi\|_{L^\infty} \cdots \|r^{1\over 3}\Gamma^{I_5}\psi\|_{L^\infty}
    \\
    &\lesssim
    C_1^p \epsilon^{1+({1\over 2}-\delta){p-5\over 2}} \langle t\rangle^{-1-({1\over 2}-\delta){p-5\over 2}+4\delta}.
\end{align*}

Therefore, all together we have
\begin{align*}
    \|\partial\Gamma^I \big( |\psi|^{p-1} \psi \big)\|
    \lesssim
    C_1^p \epsilon^{1+({1\over 2}-\delta){p-5\over 2}} \langle t\rangle^{-1-({1\over 2}-\delta){p-5\over 2}+4\delta},
\end{align*}
and the proof is complete by choosing $\delta$ sufficiently small relative to $p-5$.
\end{proof}

\begin{proposition}[Improved  energy bounds] 
    Under the assumptions \eqref{eq:BA-disp}, we have 
    \begin{align}
    \mathcal{E}_{con}^{1\over 2} (t, \Gamma^I \psi) 
    & \lesssim \epsilon^{-1} + C_1^p \epsilon^{-1+\delta},
    \qquad
    &|I|\leq 5, \label{eq:refin-disp01} 
    \\
       \mathcal{E}^{1\over 2} (t, \Gamma^I \psi)
    &\lesssim 1 + C_1^p \epsilon^\delta,
    \qquad
    &|I|\leq 5, \label{eq:refine-disp02}
    \\
    \mathcal{E}^{1\over 2} (t, \partial \Gamma^I \psi)
    &\lesssim  \epsilon + C_1^p \epsilon^{1+\delta},
    \qquad
    &|I|\leq 4. \label{eq:refine-disp03}
    \end{align}
\end{proposition}
\begin{proof}

$\bullet$ We first apply the conformal energy estimate \eqref{est:energy_conformal} to the commuted equation \eqref{eq:high-order}. By injecting the estimate \eqref{eq:disp-002} on the weighted nonlinearity we obtain
\begin{align*}
    \mathcal{E}_{con}^{1\over 2}(t, \Gamma^I \psi)
    &\lesssim
    \mathcal{E}_{con}^{1\over 2}(t_0, \Gamma^I \psi)
    +
    \sum_{|J|\leq|I|}\int_{t_0}^t \big\|\langle s+r\rangle \Gamma^J \big( |\psi|^{p-1} \psi \big) \big\| \, \d s
    \\&
    \lesssim
    \epsilon^{-1} + C_1^p \epsilon^{-1+\delta} \int_{t_0}^t s^{-1-\delta} \, \d s
    \lesssim
    \epsilon^{-1} + C_1^p \epsilon^{-1+\delta}.
\end{align*}
$\bullet$ We next prove \eqref{eq:refine-disp02}. This simply comes from taking the commuted equation
    \begin{align}\label{eq:high-order}
        \Box \Gamma^I \psi = \sum_{|J|\leq |I|}c_{J} \Gamma^J \big( |\psi|^{p-1} \psi \big)
    \end{align}
with $c_{J}$ some constants, and applying  the canonical energy estimate \eqref{est:energy_natural}. Injecting also our nonlinear bounds \eqref{eq:disp-001} then yields
\begin{align*}
    \mathcal{E}^{1\over 2}(t, \Gamma^I \psi)
    \lesssim
    \mathcal{E}^{1\over 2}(t_0, \Gamma^I \psi)
    +
    \sum_{|J|\leq |I|}\int_{t_0}^t \big\| \Gamma^J \big( |\psi|^{p-1} \psi \big) \big\| \, \d s
    \lesssim
    1 + C_1^p \epsilon^\delta \int_{t_0}^t s^{-1-\delta} \, \d s
    \lesssim
    1 + C_1^p \epsilon^\delta.
\end{align*}

$\bullet$ We finally prove \eqref{eq:refine-disp03}. Again by considering the commuted equation
    \begin{align*}
        \Box \partial\Gamma^I \psi = \sum_{|J|\leq |I|}d_{J} \partial\Gamma^J \big( |\psi|^{p-1} \psi \big)
    \end{align*}
with $d_{J}$ some constants, we use the energy estimate \eqref{est:energy_natural} combined with the nonlinear bound \eqref{eq:disp-003} to find that
\begin{align*}
    \mathcal{E}^{1\over 2}(t, \partial\Gamma^I \psi)
    \lesssim
    \mathcal{E}^{1\over 2}(t_0, \partial\Gamma^I \psi)
    +
    \sum_{|J|\leq|I|}\int_{t_0}^t \big\| \partial\Gamma^J \big( |\psi|^{p-1} \psi \big) \big\| \, \d s
    \lesssim
    \epsilon + C_1^p \epsilon^{1+\delta} \int_{t_0}^t s^{-1-\delta} \, \d s
    \lesssim
    \epsilon + C_1^p \epsilon^{1+\delta}.
\end{align*}

\end{proof}


\section{Local existence for short-pulse data}\label{sp:local}

In this section we begin the construction of global solutions to the following PDE  
\begin{equation}\label{eq:model_short-pulse}
\begin{split}
    -\Box \phi &= \pm\Big(|\psi+\phi|^{p-1} \phi + (|\psi+\phi|^{p-1} - |\psi|^{p-1})\psi \Big) =: \pm \,\mathcal{N},
    \\
  (\phi, \partial_t \phi)(t_0) &= (w_0, w_1),
    \end{split}
\end{equation}
which is coupled to our previous solution $\psi$ and 
where the unknown $\phi$ will initially have short-pulse data:
\begin{align}\label{eq:ID-shortp}
(w_0(r, \omega), w_1(r, \omega)) = \big(\epsilon^{1\over 2} h_0(\tfrac{1-r}{\epsilon}, \omega), \epsilon^{-{1\over 2}} h_1(\tfrac{1-r}{\epsilon}, \omega)\big),
\qquad 
 |L\phi(t_0)| = |w_1+\partial_r w_0| \leq \epsilon^{1\over 2},
\end{align}
 where $h_0(x,y,z) = h_0(r, \omega), h_1(x,y,z)=h_1(r, \omega)$ are arbitrary functions supported in $\{ 1\leq r \leq 2 \}$ satisfying 
    \begin{align*}
        \sum_{|I|\leq 6}|\nabla^I h_0(x)| 
        + \sum_{|J|\leq 5}|\nabla^J h_1(x)|
        \leq 1.
    \end{align*}
Our argument is organized by a decomposition of spacetime
into a local-in-time slab and two global-in-time regions separated by a null hypersurface, as shown in Figure \ref{fig:spacetime}.

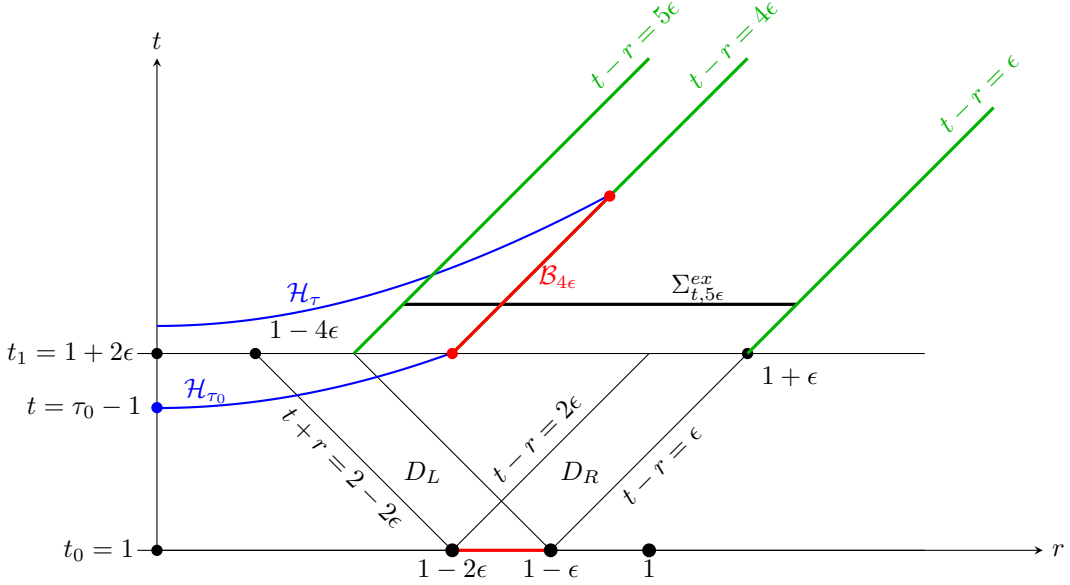
\begin{figure}[!htb]
\centering
\begin{tikzpicture}[scale=1.3,>=stealth]
\draw[->] (0,0) -- (9,0) node[right] {$r$};
\draw[->] (0,0) -- (0,5) node[above] {$t$};

\draw (-0.2,0) -- (7.8,0);
\draw (-0.2,2) -- (7.8,2);

\node[circle, fill, inner sep=1.5pt, label=left:{$t_1=1+2\epsilon\;$}] at (0,2) {};
\node[circle, fill, inner sep=1.5pt, label=left:{$t_0=1\;$}] at (0,0) {};

\node[below] at (3,0) {$1-2\epsilon$};
\node[below] at (4,0) {$1-\epsilon$};
\node[below] at (5,0) {$1$};

\node[below] at (4.3,1) {$D_R$};
\node[below] at (2.7,1) {$D_L$};

\node[circle, fill, inner sep=1.5pt, label=below right:$1+\epsilon$] at (6,2) {};
\node[circle, fill, inner sep=1.5pt, label=above right:$1-4\epsilon$] at (1,2) {};

\draw (3,0) -- (5,2) node[midway, above, sloped] {$t-r=2\epsilon$};
\draw (4,0) -- (6,2) node[midway, below, sloped] {$t-r=\epsilon$};
\draw (1,2) -- (3,0) node[midway, below, sloped] {$\;t+r=2-2\epsilon$};
\draw (2,2) -- (4,0);

\draw[black, very thick] (2.5,2.5) -- (6.5,2.5) node[above] at (5.5, 2.4) {$\Sigma_{t,5\epsilon}^{ex}$};

\def\C{53}
\draw[thick, blue, samples=200, domain=0:4.6]
  plot ({\x},{-5+ sqrt(\x*\x + \C)});
  \node[above, blue] at (1.5,2.4) {$\mathcal{H}_{\tau}$};

\def\A{60}
\draw[thick, blue, samples=200, domain=0:3]
  plot ({\x},{-6.3+ sqrt(\x*\x + \A)});
\node[above, blue] at (0.5,1.4) {$\mathcal{H}_{\tau_0}$};

\draw[green!70!black, very thick] (2,2) -- (5,5)  node[at end, above, sloped] {$t-r=5 \epsilon$};
\draw[green!70!black, very thick] (3,2) -- (6,5)  node[at end, above, sloped] {$t-r=4 \epsilon$};
\draw[green!70!black, very thick] (6,2) -- (8.5,4.5)  node[at end, above, sloped] {$t-r=\epsilon$};

\draw[red, very thick] (3,2) -- (4.6,3.6)  node[below] at (4, 3) {$\; \; \mathcal{B}_{4\epsilon}$};
\node[circle, red, fill, inner sep=1.5pt] at (3,2) {};
\node[circle, red, fill, inner sep=1.5pt] at (4.6,3.6) {};

\draw[red, very thick] (3,0) -- (4,0);

\fill (5,0) circle (2pt);
\fill (3,0) circle (2pt);
\fill (4,0) circle (2pt);

\node[circle, fill, blue, inner sep=1.5pt, label=left:{$t=\tau_0-1$}] at (0,1.45) {};

\end{tikzpicture}
\caption{\label{fig:spacetime} Local existence time interval $t \in [1, 1+2\epsilon]$ where the solution is supported in $ r \in [1-4\epsilon, 1+\epsilon] $. For the global existence, the exterior region $\mathcal{D}^{ex}$ is foliated by constant-time hypersurfaces $\Sigma_{t,5\epsilon}^{ex}$. The interior region $\mathcal{D}^{in}$ is foliated by constant-hyperbolic time hypersurfaces $\mathcal{H}_{\tau}$. The null boundary between these two regions is denoted $\mathcal{B}_{4\epsilon}$.}
\end{figure}

In the following Section \ref{subsec:local-existence}, we  first establish local existence on the short time interval $[t_0=1,t_1=1+2\epsilon]$  which provides controlled data on the slice $\{t=t_1\}$.
Then, in Section \ref{sp:exterior}, we propagate the solution globally in the exterior region
\[
\mathcal D^{ex}=\{(t,x): t\ge t_1,\ t-5\epsilon\le r\le t-\epsilon\}.
\]
In Section \ref{subsec:flux} we derive quantitative decay and flux bounds on the incoming boundary
$\mathcal B_{4\epsilon}=\{t-r=4\epsilon\}$.
These boundary estimates will serve as the interface input for the interior global
existence argument, given in Section \ref{sec:interior}.

\subsection{Local in time existence}\label{subsec:local-existence}

Our goal now is to prove the following:

    \begin{proposition}\label{prop:local_existence} There exists $\epsilon>0$ sufficiently small such that \eqref{eq:model_short-pulse} subject to the data \eqref{eq:ID-shortp} admits a local solution on $[t_0, t_1]$, where $t_1=1+2\epsilon$.
    \end{proposition}

Due to the properties of the support of the short pulse data \eqref{eq:ID-shortp}, we have that $\|\phi(t_0)\| \sim \epsilon$ and $\| \partial_t \phi(t_0)\| \sim 1$. Motivated by this, we make the following bootstrap assumptions on $[t_0,t_1]$:
\begin{subequations}\label{est:boot_local}
		\begin{align}
			\mathcal{E} (t,\partial^I\Omega^J\phi)^{\frac{1}{2}}&\leq C_2 \epsilon^{-|I|},
			\quad\quad |I|+|J|\leq 5, \label{est:boot_local1}\\
                \mathcal{E} _{con}(t,\partial^I\Omega^J\phi)^{\frac{1}{2}}&\leq C_2 \epsilon^{1-|I|},
			\quad\quad |I|+|J|\leq 5. \label{est:boot_local2}
		\end{align}
\end{subequations}
In the above, $C_2>C_1\gg1$ is to be determined, and $\epsilon$ is the size of the initial data.

We first prove a Sobolev inequality which, thanks to the smallness in the support of the function, yields a key smallness factor of $\epsilon^{\frac12}$. 

\begin{lemma}[Local Sobolev inequality] \label{lem:point_local}
Suppose that, for any fixed $t\in[t_0,t_1]$, $f(t,x)$ is supported in $2-2\epsilon-t \leq r \leq t-\epsilon$. Then we have
    \begin{equation}\label{est:point_local}
        |f(t,x)|\lesssim\epsilon^{\frac{1}{2}}\sum_{|J|\le2}\|\partial_r\Omega^Jf\|.
    \end{equation}
\end{lemma}

\begin{proof}
    First, recall the Sobolev embedding theorem on a unit sphere
		\begin{equation}\label{est:Sobolev_unit}
			|f(t,x)| \lesssim \sum_{|J|\le2}\|\Omega^{J} f(t,r\omega)\|_{L^2(\mathbb{S}^2)},
		\end{equation}
	    where $(t,x)=(t,r\omega)$ with $r=|x|$ and $\omega\in\mathbb{S}^2$. From the fundamental theorem of calculus, integrating $\partial_r f$ from $r$ to $t-\epsilon$ with $\omega$ fixed, we obtain
		\begin{equation*}
			\begin{aligned}
				|f(t,rw)| &\lesssim \int_r^{t-\epsilon} |\partial_\rho f(t,\rho\omega)|\, \d\rho 
				= \int_r^{t-\epsilon} |\partial_\rho f| \rho^{-1} \rho\, \d\rho 
				\lesssim r^{-1} \int_r^{t-\epsilon} |\partial_\rho f|\, \rho\, \d\rho\\
				&\lesssim r^{-1} \left( \int_r^{t-\epsilon} 1\, \d\rho \right)^{1/2}
				\left( \int_r^{t-\epsilon} |\partial_\rho f|^2 \rho^2\, \d\rho \right)^{1/2}
                \lesssim r^{-1} \epsilon^{\frac{1}{2}} \left( \int_r^{t-\epsilon}  \sum_{|J|\leq2}\|\Omega^{J} \partial_\rho f\|_{L^2(\mathbb{S}^2)}^2\rho^2\, d\rho \right)^{1/2}\\
				&\lesssim r^{-1} \epsilon^{\frac{1}{2}} \sum_{|J|\leq2}\|\partial_r\Omega^{J}f\|,
			\end{aligned}
		\end{equation*}
        where in the penultimate step we used~\eqref{est:Sobolev_unit} as well as $r\sim1$.
\end{proof}

\begin{proposition}[Energy and pointwise estimates] \label{prop:estimates_local1}
    Under the assumptions \eqref{est:boot_local}, the following estimates hold  for $t\in[t_0,t_1]$: the energy bounds
        \begin{align}
            \|\partial\partial^I\Omega^J\phi\|&\lesssim C_2\epsilon^{-|I|},\label{est:natural_local} & |I|+|J|\leq 5,\\
            \|L_0\partial^I\Omega^J\phi\|+\sum_{a=1}^{3}\|L_a\partial^I\Omega^J\phi\|
            +\|\Omega\partial^I\Omega^J\phi\|+\|\partial^I\Omega^J\phi\|&\lesssim C_2\epsilon^{1-|I|}, & |I|+|J|\leq 5,\label{est:conformal_local}
        \end{align}
as well as the pointwise bounds
        \begin{align}
		|\partial^I\Omega^J \phi| &\lesssim C_2\epsilon^{\frac{1}{2}-|I|}, &|I|+|J|\leq 3,\label{est:pointwise_local1}\\
		|L_0 \partial^I \Omega^J \phi|+\sum_{a=1}^{3}|L_a \partial^I \Omega^J \phi| &\lesssim  C_2\epsilon^{\frac{1}{2}-|I|}, & |I|+|J|\leq 2.\label{est:pointwise_local2}
	\end{align}
\end{proposition}

\begin{proof}
The $L^2$ estimates \eqref{est:natural_local} and \eqref{est:conformal_local} easily follow from the bootstrap assumptions~\eqref{est:boot_local}.
To prove \eqref{est:pointwise_local1}, we combine the above Sobolev inequality \eqref{est:point_local}, the commutator estimates~\eqref{est:commutators}, and the natural energy bound ~\eqref{est:natural_local}, to obtain
    \begin{equation*}
        |\partial^I\Omega^J\phi|\lesssim\epsilon^{\frac{1}{2}}\sum_{|J_1|\leq2}\|\partial_r\Omega^{J_1}\partial^I\Omega^J\phi\|\lesssim\epsilon^{\frac{1}{2}}\sum_{|J_1|\leq2}\|\partial_r\partial^I\Omega^{J_1}\Omega^J\phi\|\lesssim C_2\epsilon^{\frac{1}{2}-|I|}, \qquad |I|+|J|\leq3 \,.
    \end{equation*}
    Similarly, combining Lemma~\ref{lem:point_local} with the $L^2$ estimates~\eqref{est:conformal_local}, we get
    \begin{align*}
        |L_\alpha\partial^I\Omega^J\phi|\lesssim C_2\epsilon^{\frac{1}{2}-|I|},\quad \alpha=0,1,2,3,\quad|I|+|J|\leq2.
    \end{align*}
\end{proof}

\begin{remark}[Differentiated nonlinearities and notation]\label{rem:diff_nonlinearity}
    In the next lemma and following subsections, we need to distribute vector fields $\Gamma \in \{\partial, \Omega\}$ across our nonlinear term
    \[\mathcal{N}(\psi,\phi) = |\psi+\phi|^{p-1}(\psi+\phi)-|\psi|^{p-1}\psi\,.
    \]
    When there are no vector fields acting, we can simply use the elementary inequality
    \begin{equation}
     |\mathcal{N}(\psi,\phi)| \lesssim |\phi|^p + |\psi|^{p-1}|\phi|   \,.\label{eq:elementary}
    \end{equation}
    
    For the differentiated nonlinearity, numerous terms arise from the Leibniz rule. For example
    \begin{align*} |\Gamma \mathcal{N}(\psi,\phi)| &\lesssim |\phi|^{p-1}|\Gamma \phi| + |\phi|^{p-1}|\Gamma \psi| + |\psi|^{p-1}|\Gamma \phi| + |\phi||\psi|^{p-2}|\Gamma \psi|\,,\\
    |\Gamma^2 \mathcal{N}(\psi,\phi)| &\lesssim |\phi|^{p-1}|\Gamma^2 \phi| + |\phi|^{p-1} |\Gamma^2\psi| + |\psi|^{p-1}|\Gamma^2\phi| + |\phi||\psi|^{p-2}|\Gamma^2 \psi| + |\phi|^{p-2}|\Gamma \phi|^2 \\&+ |\phi|^{p-2} |\Gamma \phi||\Gamma \psi|  + |\phi|^{p-2} |\Gamma \psi|^2 + |\psi|^{p-2} |\Gamma \phi|^2 + |\psi|^{p-2} |\Gamma \phi|\Gamma \psi| + |\phi||\psi|^{p-3}|\Gamma \psi|^2\,.
    \end{align*}
        
    We will typically extract from each nonlinear term a factor of $|\star|^{p-5}$, for $\star \in \{\phi, \psi\}$. So, looking at two example terms in $\Gamma^2 \mathcal{N}$ above, we may write
    \[ |\phi|^{p-1}|\Gamma^2 \phi| =|\phi|^{p-5} \cdot|\Gamma^2 \phi||\phi|^4 , \quad |\phi||\psi|^{p-3}|\Gamma \psi|^2 = |\psi|^{p-5}\cdot |\phi||\Gamma \psi|^2 |\psi|^2.
    \]
    More generally, there can then be at most 5 factors that remain to be controlled (i.e. the terms that appear after the $\cdot$ above). We can then estimate such terms by the following products 
    \begin{align*}
                \sum_{|I|+|J|=2}|\partial^I \Omega^J \mathcal{N}(\psi,\phi)|
                &\lesssim\sum_{\substack{I_{1\to 5}=|I|\\{J_{1\to 5}}=|J|}} \Big( \sum_{0\leq m\leq5} |\phi|^{p-5}\prod_{k=1}^{m}|\partial^{I_k}\Omega^{J_k}\phi|\prod_{l=m+1}^{5}|\partial^{I_l}\Omega^{J_l}\psi| \\& \qquad \qquad \qquad 
                +\sum_{1\leq m\leq5} |\psi|^{p-5}\prod_{k=1}^{m}|\partial^{I_k}\Omega^{J_k}\phi|\prod_{l=m+1}^{5}|\partial^{I_l}\Omega^{J_l}\psi|\Big).
     \end{align*}
     Note that the product $\prod$ is defined to be 1 if the top index is strictly smaller than the bottom index.
    These products introduce many new terms that do not actually appear in our nonlinearities. However this does not pose a problem, as key to our analysis is keeping track of the number of factors of $\phi, \psi$ and keeping track of where the derivatives fall, e.g. terms like $\partial \phi$ lead to factors of $\epsilon^{-1}$ while $\Omega \phi$ do not.  
    \end{remark}

\begin{lemma}[Local-in-time bounds on the nonlinearities]\label{lem:nonlinear_local}
 Under the assumptions \eqref{est:boot_local}, the following estimate holds for $t\in[t_0,t_1]$: 
        \begin{align*}
            \|\partial^I\Omega^J \mathcal{N}(\psi,\phi)\| + \|\langle t+r\rangle\partial^I\Omega^J \mathcal{N}(\psi,\phi)\| \lesssim C_2^p\epsilon^{\frac{p-1}{2}-(p-1)\delta-|I|}, & \quad|I|+|J|\leq5\,.
        \end{align*}
      
\end{lemma}

\begin{proof}
     Since on the small interval on $[t_0,t_1]$ the weight $\langle t+r\rangle$ is comparable to a positive constant, we only need to study 
$$ \partial^I\Omega^J \mathcal{N}(\psi,\phi)= \partial^I\Omega^J\big(|\psi+\phi|^{p-1}(\psi+\phi)-|\psi|^{p-1}\psi\big) \,.$$
We apply Leibniz rule
and separately study two cases depending on the total number of derivatives $|K|=|I|+|J|$. The first case, when $|I|+|J|=0$, can be directly estimated by taking $\psi$ only in $L^\infty$. Using Proposition~\ref{prop:estimates_local1} and the pointwise estimates on $\psi$ given in \eqref{eq:disp-psi01}, we obtain
    \begin{equation*}
        \begin{aligned}
            \|\mathcal{N}(\psi,\phi) \|\lesssim\||\phi|^p\|+\||\psi|^{p-1}\phi\|
            \lesssim\|\phi\|\|\phi\|^{p-1}_{L^\infty}+\|\phi\|\|\psi\|^{p-1}_{L^\infty}
            \lesssim C_2^p\epsilon^{\frac{p+1}{2}-(p-1)\delta}.
        \end{aligned}
    \end{equation*}
    
The second case, where $1\leq|I|+|J|\leq5$, is more involved.  Following the discussion in Remark \ref{rem:diff_nonlinearity}, we express the nonlinear terms into two pieces distinguished by an overall factor of $|\phi|^{p-5}$ or $|\psi|^{p-5}$:
    \begin{equation*}
        \begin{aligned}
            \|\partial^I\Omega^J\mathcal{N}(\psi,\phi)\|
            &\lesssim\sum_{\substack{I_{1\to 5}=I\\J_{1\to5}=J}} \Big( \sum_{0\leq m\leq5}\Big\||\phi|^{p-5}\prod_{k=1}^{m}|\partial^{I_k}\Omega^{J_k}\phi|\prod_{l=m+1}^{5}|\partial^{I_l}\Omega^{J_l}\psi|\Big\|\\
            &+\sum_{1\leq m\leq5}\Big\||\psi|^{p-5}\prod_{k=1}^{m}|\partial^{I_k}\Omega^{J_k}\phi|\prod_{l=m+1}^{5}|\partial^{I_l}\Omega^{J_l}\psi|\Big\|\Big) =:\mathcal{R}_1+\mathcal{R}_2.
        \end{aligned}
    \end{equation*}
\emph{Estimate for $\mathcal{R}_1$.} We first break $\mathcal{R}_1$ into 3 pieces based on terms with just products of $\Gamma\phi$, those terms with mixed products of $\Gamma \phi$ and $\Gamma \psi$ (which are the `worst' terms in the sense of $\epsilon$ growth), and those terms just with products of $\Gamma \psi$. Each term will be estimated by placing the higher-order derivative in $L^2$ and controlling all remaining factors in $L^\infty$. 
    \begin{equation*}
        \begin{aligned}
            \mathcal{R}_1&\lesssim\sum_{\substack{I_{1\to 5}=I\\J_{1\to5}=J}}\Big( \Big\||\phi|^{p-5}\prod_{k=1}^{5}|\partial^{I_k}\Omega^{J_k}\phi|\Big\|
            +\sum_{1\leq m\leq4}\Big\||\phi|^{p-5}\prod_{k=1}^{m}|\partial^{I_k}\Omega^{J_k}\phi|\prod_{l=m+1}^{5}|\partial^{I_l}\Omega^{J_l}\psi|\Big\|\\
            &\qquad \qquad \qquad +\Big\||\phi|^{p-5}\prod_{l=1}^{5}|\partial^{I_l}\Omega^{J_l}\psi|\Big\|\Big) =:\mathcal{R}_1^1+\mathcal{R}_1^2+\mathcal{R}_1^3.
        \end{aligned}
    \end{equation*}
For $\mathcal{R}_1^1$, which involves only $\phi$ factors, we insert the estimates from Proposition~\ref{prop:estimates_local1}  to obtain
    \begin{equation*}
        \begin{aligned}
            \mathcal{R}_{1}^1&\lesssim\sum_{\substack{I_{1\to 5}=I\\J_{1\to5}=J}}\sum_{|K_1|\ge|K_2|, \ldots, |K_5|}\|\partial^{I_1}\Omega^{J_1}\phi\|\|\phi\|_{L^\infty}^{p-5}\prod_{k=2}^{5}\|\partial^{I_k}\Omega^{J_k}\phi\|_{L^\infty}
            \lesssim C_2^p\epsilon^{\frac{p+1}{2}-|I|}.
        \end{aligned}
    \end{equation*}
For the term $\mathcal{R}_{1}^2$, we also need to use the $L^\infty$ estimates on $\psi$ from \eqref{eq:disp-psi01}. We find
    \begin{equation*}
        \begin{aligned}
            \mathcal{R}_{1}^2&\lesssim\sum_{\substack{I_{1\to 5}=I\\J_{1\to5}=J}} \Big(\sum_{\substack{|K_2|, \ldots, |K_5|\leq2\\1\leq m\leq4}}\|\partial^{I_1}\Omega^{J_1}\phi\|\|\phi\|_{L^\infty}^{p-5}\prod_{k=2}^{m}\|\partial^{I_k}\Omega^{J_k}\phi\|_{L^\infty}\prod_{l=m+1}^{5}\|\partial^{I_l}\Omega^{J_l}\psi\|_{L^\infty}\\
            &\qquad \qquad \qquad +\sum_{\substack{|K_1|, \ldots, |K_4|\leq2\\|K_5'|=|K_5|-1\geq2\\1\leq m\leq4}}\|\partial\partial^{I'_5}\Omega^{J'_5}\psi\|\|\phi\|_{L^\infty}^{p-5}\prod_{k=1}^m\|\partial^{I_k}\Omega^{J_k}\phi\|_{L^\infty}\prod_{k=m+1}^{4}\|\partial^{I_k}\Omega^{J_k}\psi\|_{L^\infty}\Big)\\
            &\lesssim C_2^p\epsilon^{\frac{p-1}{2}-3\delta-|I|},
        \end{aligned}
    \end{equation*}
    where we used the facts that $\langle r\rangle\sim1$ and 
    \begin{equation}\label{est:Omegad}
        |\Omega^J\psi|\lesssim\sum_{|J'|=|J|-1}|\langle r\rangle\partial\Omega^{J'}\psi|.
    \end{equation}

    For $\mathcal{R}_1^3$, all derivatives act on $\psi$ and none on $\phi$. Consequently, using~\eqref{est:Omegad}, \eqref{eq:disp-psi01} and  Proposition~\ref{prop:estimates_local1}  we derive
    \begin{equation*}
        \begin{aligned}
            \mathcal{R}_{1}^3&\lesssim \sum_{\substack{I_{1\to 5}=I\\J_{1\to5}=J}} \sum_{\substack{|K_1'|=|K_1|-1\\|K_1|\ge|K_2|, \ldots, |K_5|}}\|\partial\partial^{I'_1}\Omega^{J'_1}\psi\|\|\phi\|_{L^\infty}^{p-5}\prod_{k=2}^{5}\|\partial^{I_k}\Omega^{J_k}\psi\|_{L^\infty}
            \lesssim C_2^p\epsilon^{\frac{p-1}{2}-4\delta}.
        \end{aligned}
    \end{equation*}
\emph{Estimate for $\mathcal{R}_2$.} We similarly  break $\mathcal{R}_2$ into two pieces -- one which just involves products of $\Gamma\phi$, and the other with mixed products of $\Gamma \phi$ and $\Gamma \psi$ (these latter terms being the worse to estimate).
    \begin{equation*}
        \begin{aligned}
            \mathcal{R}_2\lesssim\sum_{\substack{I_{1\to 5}=I\\J_{1\to5}=J}}\Big( \Big\||\psi|^{p-5}\prod_{k=1}^{5}|\partial^{I_k}\Omega^{J_k}\phi|\Big\|
            +\sum_{1\leq m\leq4}\Big\||\psi|^{p-5}\prod_{k=1}^{m}|\partial^{I_k}\Omega^{J_k}\phi|\prod_{l=m+1}^{5}|\partial^{I_l}\Omega^{J_l}\psi|\Big\|\Big) =:\mathcal{R}_2^1+\mathcal{R}_2^2.
        \end{aligned}
    \end{equation*}
    For $\mathcal{R}_2^1$, where no derivatives appear on the $\psi$ factor, we obtain
    \begin{equation*}
        \begin{aligned}
            \mathcal{R}_2^1&\lesssim\sum_{\substack{I_{1\to 5}=I\\J_{1\to5}=J}} \sum_{|K_1|\ge|K_2|, \ldots, |K_5|}\|\partial^{I_1}\Omega^{J_1}\phi\|\|\psi\|_{L^\infty}^{p-5}\prod_{k=2}^{5}\|\partial^{I_k}\Omega^{J_k}\phi\|_{L^\infty}
            \lesssim C_2^p\epsilon^{\frac{p+1}{2}-(p-5)\delta-|I|}.
        \end{aligned}
    \end{equation*}
For the term $\mathcal{R}_{2}^2$, we have mixed products of $\Gamma \phi$ and $\Gamma \psi$. We obtain 
    \begin{equation*}
        \begin{aligned}
            &\mathcal{R}_{2}^2\lesssim\sum_{\substack{I_{1\to 5}=I\\J_{1\to5}=J}} \Big( \sum_{\substack{|K_2|, \ldots, |K_5|\le2 \\1\leq m\leq 4}}\|\partial^{I_1}\Omega^{J_1}\phi\|\|\psi\|_{L^\infty}^{p-5}\prod_{k=2}^{m}\|\partial^{I_k}\Omega^{J_k}\phi\|_{L^\infty}\prod_{l=m+1}^{5}\|\partial^{I_l}\Omega^{J_l}\psi\|_{L^\infty}\\
            &+\sum_{\substack{|K_1|, \ldots, |K_4|\le2\\|K_5'|=|K_5|-1\geq2\\1\leq m\leq4}}\|\partial\partial^{I'_5}\Omega^{J'_5}\psi\|\|\psi\|_{L^\infty}^{p-5}\prod_{k=1}^m\|\partial^{I_k}\Omega^{J_k}\phi\|_{L^\infty}\prod_{l=m+1}^{4}\|\partial^{I_l}\Omega^{J_l}\psi\|_{L^\infty}\Big) \\
            &\lesssim C_2^p\epsilon^{\frac{p+1}{2}-(p-1)\delta-|I|}+C_2^p\epsilon^{\frac{p-1}{2}-(p-2)\delta-|I|}\lesssim C_2^p\epsilon^{\frac{p-1}{2}-(p-2)\delta-|I|}.
        \end{aligned}
    \end{equation*}
Combining the above estimates, we conclude that for $|I|+|J|\leq5$,
    \begin{equation*}
        \|\partial^I\Omega^J\mathcal{N}(\psi,\phi)\|\lesssim C_2^p\epsilon^{\frac{p-1}{2}-(p-1)\delta-|I|}.
    \end{equation*}
\end{proof}

    We now conclude the proof of Proposition \ref{prop:local_existence} by closing   the bootstrap assumptions on $[t_0,t_1]$. 
    
    \begin{proof}[Proof of Proposition \ref{prop:local_existence}]
    By considering the commuted equation
    \begin{equation*}
        -\Box\partial^I\Omega^J\phi=\partial^I\Omega^J\mathcal{N}(\psi,\phi)  = \partial^I\Omega^J\big(|\psi+\phi|^{p-1}(\psi+\phi)-|\psi|^{p-1}\psi\big),
    \end{equation*}
    we apply the energy estimate~\eqref{est:energy_natural}  and the nonlinear bounds from Lemma \ref{lem:nonlinear_local} to obtain
	\begin{equation*}
	    \begin{aligned}
	        \mathcal{E}^{\frac{1}{2}}(t,\partial^I\Omega^J\phi)&\lesssim\mathcal{E}^{\frac{1}{2}}(t_0,\partial^I\Omega^J\phi)+\int_{t_0}^t\|\partial^I\Omega^J\mathcal{N}(\psi,\phi)\|\, ds\\
            &\lesssim\epsilon^{-|I|}+\int_{t_0}^t C_2^p\epsilon^{\frac{p-1}{2}-(p-1)\delta-|I|}\, ds
            \lesssim\epsilon^{-|I|}+C_2^p\epsilon^{\frac{p+1}{2}-(p-1)\delta-|I|}.
	    \end{aligned}
	\end{equation*}

    Similarly, employing the conformal energy~\eqref{est:energy_conformal} and the nonlinear bounds from Lemma \ref{lem:nonlinear_local}, we have

    \begin{equation*}
	    \begin{aligned}
	        \mathcal{E}_{con}^{\frac{1}{2}}(t,\partial^I\Omega^J\phi)&\lesssim\mathcal{E}_{con}^{\frac{1}{2}}(t_0,\partial^I\Omega^J\phi)+\int_{t_0}^t\|\langle
             s+r\rangle\partial^I\Omega^J\mathcal{N}(\psi,\phi)\|\, ds\\
            &\lesssim\epsilon^{1-|I|}+\int_{t_0}^t C_2^p\epsilon^{\frac{p-1}{2}-(p-1)\delta-|I|}\, ds
            \lesssim\epsilon^{1-|I|}+C_2^p\epsilon^{\frac{p+1}{2}-(p-1)\delta-|I|}.
	    \end{aligned}
	\end{equation*}
	These estimates strictly improve the bootstrap assumptions~\eqref{est:boot_local} by taking $C_2^{p-1}\epsilon^{\frac{p-1}{2}-(p-1)\delta}\ll1$. Therefore, the bootstrap argument closes, and the solution $\phi$ exists on $[t_0,t_1].$
     \end{proof}

    Before proceeding to the analysis of the exterior and interior regions, we record the following refined $L^\infty$ estimates at time $t_1$. The proof involves writing the wave equation with respect to the null frame and integrating in $u$ or $v$ directions (for \eqref{eq:local_Lest_t1} and \eqref{eq:local_Lbarest_t1} respectively) using the fundamental theorem of calculus and the fact that the solution is supported on an interval of size $\epsilon$. We remark also a key smallness gain of $|L^2\phi|\sim \epsilon^{1/2}$ in \eqref{eq:local_Lest_t1} (compared to $\epsilon^{-1/2}$ from \eqref{est:pointwise_local2}) which will be used later on in the boundary flux estimates. 
    
	\begin{lemma}[Refined pointwise  estimates at $t_1$] \label{lem:local-estimates} The solution $\phi$ from Proposition \ref{prop:local_existence} obeys the following bounds for $L^{ \leq 2}\partial^I\Omega^J\phi$:
    \begin{equation}\label{eq:local_Lest_t1}\begin{split}
				|\partial^I\Omega^J\phi(t_1,x)|&\lesssim C_2\epsilon^{\frac{1}{2}-|I|},\quad r\in[1-4\epsilon,1+\epsilon],\quad |I|+|J|\leq 3,\\
				|L\partial^I\Omega^J\phi(t_1,x)|&\lesssim C_2\epsilon^{\frac{1}{2}-|I|},\quad r\in[1-4\epsilon,1+\epsilon],\quad |I|+|J|\leq 2,\\
				|L^2\phi(t_1,x)|&\lesssim C_2\epsilon^{\frac{1}{2}},\quad\quad\  r\in[1-3\epsilon,1+\epsilon],
			\end{split}\end{equation}
	the following estimates for $\underline{L}\underline{L}^{\leq 1}\partial^I\Omega^J\phi$:
			\begin{equation}\label{eq:local_Lbarest_t1}\begin{split}
            |\underline{L}\partial^I\Omega^J\phi(t_1,x)|&\lesssim C_2\epsilon^{\frac{3}{2}-|I|},\quad r\in[1-4\epsilon,1],\quad |I|+|J|\leq 1,\\
				|\underline{L}^2\phi(t_1,x)|&\lesssim C_2\epsilon^{\frac{1}{2}},\quad\quad\  r\in[1-4\epsilon,1],
			\end{split}\end{equation}
	and the following estimates for $\partial^I\Omega^J\phi$:
            \begin{equation}\label{eq:local_Omegaest_t1}\begin{split}
            |\Omega^J\phi(t_1,x)|&\lesssim C_2\epsilon^{\frac{1}{2}},\quad \qquad\ \ r\in[1-4\epsilon,1+\epsilon],\quad |J|\leq3,\\
            |\partial^I\Omega^J\phi(t_1,x)|&\lesssim C_2\epsilon^{\frac{3}{2}-|I|},\quad\quad\ r\in[1-4\epsilon,1], \quad |I|\geq1,\quad|I|+|J|\leq2.
			\end{split}\end{equation}
	
	\end{lemma}
	
	\begin{proof}
     For this region, we  introduce the notation $D_R, D_L$ shown in Figure \ref{fig:spacetime} and defined by
    $$D_R=\{(t,x): t_0\leq t\leq t_1,\; 2-\epsilon-t\leq r\leq t-\epsilon\}, \quad D_L=\{(t,x): t_0\leq t\leq t_1,\; 2-2\epsilon-t\leq r\leq t-2\epsilon\}\,.$$
    
		$\bullet$ We start by proving \eqref{eq:local_Lest_t1}. Notice that $L=(t+r)^{-1}(L_0+\omega^aL_a)$. Therefore, the first two estimates follow directly from the $L^\infty$ bounds in Proposition~\ref{prop:estimates_local1}. Moreover, when $t_1=1+2\epsilon$, the support of $\phi$ is contained in the region $r\in[1-4\epsilon,\,1+\epsilon]$, so that $r\sim1$ on this time slice. 
        We next turn to the pointwise bound for $L^2\phi$ in \eqref{eq:local_Lest_t1}. First, we rewrite the equation for $\phi$ in the form
        \begin{equation}\label{equ:waveLL_}
			\underline{L}L\phi=\frac{1}{r}L\phi-\frac{1}{r}\underline{L}\phi+\frac{1}{r^2}\slashed{\Delta}_{\mathbb{S}^2}\phi\pm\mathcal{N}(\psi,\phi),
		\end{equation}
        where $\slashed{\Delta}_{\mathbb{S}^2}$ denotes the Laplace--Beltrami operator on the unit sphere $\mathbb{S}^2$.
		Acting the vector field $L$ on both sides of~\eqref{equ:waveLL_}, we have
		\begin{equation*}
			\begin{aligned}
				\underline{L}L^2\phi&=L\left(\frac{1}{r}L\phi-\frac{1}{r}\underline{L}\phi+\frac{1}{r^2}\slashed{\Delta}_{\mathbb{S}^2}\phi\pm\mathcal{N}(\psi,\phi)\right)\\
				&=\frac{1}{r}L^2\phi-\frac{1}{r}L\underline{L}\phi+\frac{1}{r^2}L\slashed{\Delta}_{\mathbb{S}^2}\phi
				-\frac{1}{r^2}L\phi+\frac{1}{r^2}\underline{L}\phi-\frac{2}{r^3}\slashed{\Delta}_{\mathbb{S}^2}\phi\pm L\mathcal{N}(\psi,\phi).
			\end{aligned}
		\end{equation*}
	    Using the pointwise estimates on $\psi$ in~\eqref{eq:disp-psi01}--\eqref{eq:disp-d-psi01}, and the pointwise bounds on $\phi$ from \eqref{est:pointwise_local1}--\eqref{est:pointwise_local2}, together with the fact that $r\sim1$ in the region under consideration, we deduce that
	    \begin{equation*}
	    	|\underline{L}L^2\phi|\lesssim|L^2\phi|+C_2\epsilon^{-\frac{1}{2}}.
	    \end{equation*}
	    Applying the fundamental theorem of calculus along the $u$-direction, we obtain
	    \begin{equation*}
	    	|L^2\phi(t,x)|\leq\int_{\epsilon}^{5\epsilon}|\partial_u L^2\phi|\, du\leq\frac{1}{2}\int_{\epsilon}^{5\epsilon}|\underline{L}L^2\phi|\, du\lesssim C_2\epsilon^{\frac{1}{2}}+\int_{\epsilon}^{5\epsilon}|L^2\phi|\, du,
	    \end{equation*}
	    for any $(t,x)\in D_R$. 
	    By Gronwall inequality, it follows that
	    \begin{equation*}
	    	|L^2 \phi(t,x)|\lesssim C_2\epsilon^{\frac{1}{2}},\quad \mbox{for}\quad (t,x)\in D_R.
	    \end{equation*}

        $\bullet$ We now prove \eqref{eq:local_Lbarest_t1} by studying the commuted equation
        \begin{equation*}
        	-\Box \partial^I\Omega^J\phi=\pm\partial^I\Omega^J\mathcal{N}(\psi,\phi),\quad |I|+|J|\leq 1.
        \end{equation*}
        Rewriting the wave operator in terms of the null frame, we obtain
        \begin{equation*}
        \begin{aligned}
            L\underline{L}\partial^I\Omega^J\phi&=\frac{1}{r}L\partial^I\Omega^J\phi-\frac{1}{r}\underline{L}\partial^I\Omega^J\phi+\frac{1}{r^2}\slashed{\Delta}_{\mathbb{S}^2}\partial^I\Omega^J\phi\pm\partial^I\Omega^J\mathcal{N}(\psi,\phi).
        \end{aligned}
        \end{equation*}
        Using the estimates on $\psi$ established in~\eqref{eq:disp-psi01}, and on $\phi$ in \eqref{est:pointwise_local1} and \eqref{est:pointwise_local2}, we deduce
        \begin{equation*}
        	|L\underline{L}\partial^I\Omega^J\phi|\lesssim C_2\epsilon^{\frac{1}{2}-|I|}+|\underline{L}\partial^I\Omega^J\phi|.
        \end{equation*}
        Applying the fundamental theorem of calculus along the $v$-direction, we obtain
        \begin{equation*}
        	\begin{aligned}
        	    |\underline{L}\partial^I\Omega^J\phi(t,x)|&\leq\int_{2-2\epsilon}^{2+2\epsilon}|\partial_v\underline{L}\partial^I\Omega^J\phi|\, dv\leq\frac{1}{2}\int_{2-2\epsilon}^{2+2\epsilon}|L\underline{L}\partial^I\Omega^J\phi|\, dv
                \lesssim C_2\epsilon^{\frac{3}{2}-|I|}+\int_{2-2\epsilon}^{2+2\epsilon}|\underline{L}\partial^I\Omega^J\phi|\, dv,
        	\end{aligned}
        \end{equation*}
        for any $(t,x)\in D_L$.
        Combining the above estimate with  Gronwall inequality, we have
        \begin{equation*}
        	|\underline{L}\partial^I\Omega^J\phi(t,x)|\lesssim C_2\epsilon^{\frac{3}{2}-|I|},\quad\mbox{for}\quad (t,x)\in D_L.
        \end{equation*}
        
        Next, acting $\underline{L}$ on both sides of the equation~\eqref{equ:waveLL_}, we have
        \begin{equation*}
        	\begin{aligned}
        		L\underline{L}^2\phi&=\underline{L}\left(\frac{1}{r}L\phi-\frac{1}{r}\underline{L}\phi+\frac{1}{r^2}\slashed{\Delta}_{\mathbb{S}^2}\phi\pm\mathcal{N}(\psi,\phi)\right)\\
        		&=\frac{1}{r}\underline{L}L\phi-\frac{1}{r}\underline{L}^2\phi+\frac{1}{r^2}\underline{L}\slashed{\Delta}_{\mathbb{S}^2}\phi
        		+\frac{1}{r^2}L\phi-\frac{1}{r^2}\underline{L}\phi+\frac{2}{r^3}\slashed{\Delta}_{\mathbb{S}^2}\phi\pm\underline{L}\mathcal{N}(\psi,\phi). 
                \end{aligned}
        \end{equation*}
        Using again the previously established bounds, we find
        \begin{equation*}
        	|L\underline{L}^2\phi|\lesssim C_2\epsilon^{-\frac{1}{2}}+|\underline{L}^2\phi|.
        \end{equation*}
        Applying the fundamental theorem of calculus along the $v$-direction, we apply Gronwall inequality to conclude that
        \begin{equation*}
            |\underline{L}^2\phi(t,x)|
            \lesssim C_2\epsilon^{\frac12},
            \qquad \mbox{for}\quad (t,x)\in D_L.
        \end{equation*}
        
        
        
        $\bullet$ We now prove \eqref{eq:local_Omegaest_t1}. The first inequality, where $|I|=0$, simply follows from \eqref{eq:local_Lest_t1}. For the second inequality, when $|I|\geq1$ and $|I|+|J|\leq2$, we write
        \begin{equation*}
        	\begin{aligned}
        		&|\partial^I\Omega^J\phi|\lesssim\sum_{|I_1|=|I|-1}|\partial\partial^{I_1}\Omega^J\phi|
                \lesssim\sum_{|I_1|=|I|-1}\left(|L\partial^{I_1}\Omega^J\phi|+|\underline{L}\partial^{I_1}\Omega^J\phi|+\left|\frac{1}{r}\Omega\partial^{I_1}\Omega^J\phi\right|\right).
        	\end{aligned}
        \end{equation*}
        Using the bounds established in \eqref{eq:local_Lest_t1} and \eqref{eq:local_Lbarest_t1}, we deduce
        \begin{equation*}
        	|\partial^I\Omega^J\phi(t_1,x)|\lesssim C_2\epsilon^{\frac{3}{2}-|I|},\quad\mbox{for}\quad r\in[1-4\epsilon,1].
        \end{equation*}
	\end{proof}

\section{Global  existence for short-pulse data in the exterior}\label{sp:exterior}

In this section, we establish the following global existence of the solution in the exterior region:

    \begin{proposition}\label{prop:global_ext_existence} There exists $\epsilon>0$ sufficiently small such that \eqref{eq:model_short-pulse} admits a global in time solution in $\mathcal{D}^{ex}=\{(t,x): t\geq t_1,t-5\epsilon\leq r\leq t-\epsilon\}.$ 
    \end{proposition}

\subsection{Notation and auxiliary estimates}\label{se:notation_ex}
We begin by introducing some notation and auxiliary estimates. 
In the exterior region, for any $\epsilon\leq \tilde{u}\leq5\epsilon$, we define 
	\begin{align*}
		\Sigma_{t,\tilde{u}}^{ex}&=\{(s,x)\in\mathcal{D}^{ex}: s=t,s-\tilde{u}\leq|x|\leq s-\epsilon\},\\
        \mathcal{D}^{ex}_{t,\tilde{u}}&=\{(s,x):t_1\leq s\leq t, s-\tilde{u}\leq|x|\leq s-\epsilon\},\\
		\mathcal{B}_{\tilde{u}}&=\{(s,x)\in\mathcal{D}^{ex}: s-|x|=\tilde{u}\}.
	\end{align*}
	The natural energy and the conformal energy are defined as follows
	\begin{align*}
		\mathcal{E}^{ex}(\tilde{u},t,\varphi)&=\int_{\Sigma_{t,\tilde{u}}^{ex}}\big(|\partial_t\varphi|^2+\sum_a|\partial_a\varphi|^2\big)\, dx,\\
		\mathcal{E}^{ex}_{con}(\tilde{u},t,\varphi)&=\int_{\Sigma_{t,\tilde{u}}^{ex}}\big(|L_0\varphi+2\varphi|^2+|\Omega\varphi|^2+\sum_a|L_a\varphi|^2\big)\, dx.
	\end{align*}
    It is also useful to define the following energies with the various boundary terms that will arise in the energy estimates:
    \begin{equation*}
        \begin{aligned}
            E^{ex}(\tilde{u},t,\varphi)&=\mathcal{E}^{ex}(\tilde{u},t,\varphi)+\int_{\mathcal{B}_{\tilde{u}}\cap \mathcal{D}^{ex}_{t,\tilde{u}}}|G\varphi|^2\, dx,\\
            E_{con}^{ex}(\tilde{u},t,\varphi)&=\mathcal{E}^{ex}_{con}(\tilde{u},t,\varphi)+\int_{\mathcal{B}_{\tilde{u}}\cap \mathcal{D}^{ex}_{t,\tilde{u}}}|(t+r) L\varphi+2\varphi|^2\, dx+\int_{\mathcal{B}_{\tilde{u}}\cap \mathcal{D}^{ex}_{t,\tilde{u}}}\tilde{u}^2\big|\frac{\Omega\varphi}{r}\big|^2\, dx.
        \end{aligned}
    \end{equation*}

\begin{lemma}[Exterior energy estimates]\label{lem:energy_ex}
		Let $\varphi$ be a solution to the wave equation vanishing on $\mathcal{B}_\epsilon$. Then, the following natural energy estimate holds
			\begin{equation}\label{est:natural_ex}
				\mathcal{E}^{ex}(\tilde{u},t,\varphi)^{1/2}+\left(\int_{\mathcal{B}_{\tilde{u}}\cap \mathcal{D}^{ex}_{t,\tilde{u}}}|G\varphi|^2\, dx\right)^{1/2}\lesssim\mathcal{E}^{ex}(\tilde{u},t_1,\varphi)^{1/2}+\int_{t_1}^{t}\|\Box\varphi\|_{L^2(\Sigma^{ex}_{s,\tilde{u}})}\, ds,
			\end{equation}
		as well as the conformal energy estimate \begin{equation}\label{est:conformal_ex}
					\begin{aligned}
						\mathcal{E}^{ex}_{con}(\tilde{u},t,\varphi)^{1/2}&+\left(\int_{\mathcal{B}_{\tilde{u}}\cap \mathcal{D}^{ex}_{t,\tilde{u}}}|(t+r) L\varphi+2\varphi|^2\, dx\right)^{1/2}+\left(\int_{\mathcal{B}_{\tilde{u}}\cap \mathcal{D}^{ex}_{t,\tilde{u}}}\tilde{u}^2\big|\frac{\Omega\varphi}{r}\big|^2\, dx\right)^{1/2}\\
						&\lesssim\mathcal{E}^{ex}_{con}(\tilde{u},t_1,\varphi)^{1/2}+\int_{t_1}^{t}\|(s+r)\Box\varphi\|_{L^2(\Sigma^{ex}_{s,\tilde{u}})}\, ds.
					\end{aligned}
			\end{equation}
	\end{lemma}
	\begin{proof}
		$\bullet$ We prove \eqref{est:natural_ex} by multiplying the wave equation by $\partial_t\varphi$ and integrating over the region $\mathcal{D}^{ex}_{t,\tilde{u}}$. This yields
		\begin{equation*}
			\begin{aligned}
				\int_{\Sigma_{t,\tilde{u}}^{ex}}|\partial\varphi|^2\, dx+\sum_{a}\int_{\mathcal{B}_{\tilde{u}}\cap \mathcal{D}^{ex}_{t,\tilde{u}}}|G_a\varphi|^2\, dx
				=\int_{\Sigma_{t_1,\tilde{u}}^{ex}}|\partial\varphi|^2\, dx+\int_{\mathcal{D}^{ex}_{t,\tilde{u}}}2(-\Box \varphi)\partial_t\varphi\, dxds,
			\end{aligned}
		\end{equation*}
        and the conclusion follows by applying the Cauchy–Schwarz inequality to the spacetime integral and then Gronwall inequality.
        
		$\bullet$ We next prove \eqref{est:conformal_ex}. Multiplying the wave operator by $(r^2+t^2)\partial_t\varphi+2rt\partial_r\varphi+2t\varphi$, we get
		\begin{equation*}
			\begin{aligned}
				-\Box\varphi\,\cdot&\left((r^2+t^2)\partial_t\varphi+2rt\partial_r\varphi+2t\varphi\right)
				=\frac{1}{2}\partial_t\left((r^2+t^2)|\partial\varphi|^2+4tr\partial_r\varphi\partial_t\varphi+4t\varphi\partial_t\varphi+4\varphi^2+4r\varphi\partial_r\varphi\right) \\&
                -\partial^a\left((r^2+t^2)\partial_a\varphi\partial_t\varphi+tx_a(|\partial_t\varphi|^2-\partial^b\varphi\partial_b\varphi)+2tr\partial_r\varphi\partial_a\varphi+2t\varphi\partial_a\varphi+2x_a\varphi\partial_t\varphi\right).
			\end{aligned}
		\end{equation*} 
		Integrating the above identity over the region $\mathcal{D}^{ex}_{t,\tilde{u}}$, we obtain
		\begin{equation*}
			\begin{aligned}
				&\int_{\Sigma_{t,\tilde{u}}^{ex}}\big(|L_0\varphi+2\varphi|^2+|\Omega\varphi|^2+\sum_a|L_a\varphi|^2\big)\, dx
				+\int_{\mathcal{B}_{\tilde{u}}\cap \mathcal{D}^{ex}_{t,\tilde{u}}}|(t+r) L\varphi+2\varphi|^2\, dx+\int_{\mathcal{B}_{\tilde{u}}\cap \mathcal{D}^{ex}_{t,\tilde{u}}}\tilde{u}^2\big|\frac{\Omega\varphi}{r}\big|^2\, dx\\
				=&\int_{\Sigma_{t_1,\tilde{u}}^{ex}}\big(|L_0\varphi+2\varphi|^2+|\Omega\varphi|^2+\sum_a|L_a\varphi|^2\big)\, dx
				+\int_{\mathcal{D}^{ex}_{t,\tilde{u}}}2(-\Box\varphi )[s(L_0\varphi+2\varphi)+x^aL_a\varphi]\, dxds.
			\end{aligned}
		\end{equation*}
        The conclusion follows by again applying Cauchy–Schwarz to the spacetime integral and then Gronwall inequality.
	\end{proof}

The next result, due to Lindblad and Rodnianski, is a type of spacetime Hardy inequality where the weight depends on the distance from the light cone $t=r$. 
    \begin{lemma}[Hardy inequality \cite{L-R-cmp, L-R-annals}]\label{lem:hardy_ex}
		Assume that $\varphi$ vanishes on $\mathcal{B}_\epsilon$. Then in the exterior region $\mathcal{D}^{ex}$, for any $\tilde{u}\in[\epsilon,5\epsilon]$, we have the following Hardy inequality
		\begin{equation}\label{est:hardy_ex}
			\left\|(t-r)^{-1}\varphi\right\|_{L^2(\Sigma_{t,\tilde{u}}^{ex})}\lesssim\|\partial_r\varphi\|_{L^2(\Sigma_{t,\tilde{u}}^{ex})}.
		\end{equation}
	\end{lemma}
    \begin{proof}
    	Let $t$ be fixed and $\omega\in\mathbb{S}^2$. Define $\chi(r)=r\varphi(t,r\omega)$ so that $\partial_r\chi=\varphi+r\partial_r\varphi$. Since $\varphi$ vanishes on $\mathcal{B}_\epsilon$, we have $\chi(t-\epsilon)=0$. We compute
    	\begin{equation*}
    		\begin{aligned}
    			\int_{t - \tilde{u}}^{t - \epsilon} \left|\frac{\chi}{t - r}\right|^2 \, dr
    			&= \int_{t -  \tilde{u}}^{t - \epsilon} \chi^2 \, d\left(\frac{1}{t - r}\right)
    			= \frac{\chi^2}{t - r}\Big|_{t -  \tilde{u}}^{t - \epsilon}
    			- 2\int_{t -  \tilde{u}}^{t - \epsilon} \frac{\chi}{t - r}\partial_r\chi \, dr\\
    			&\leq-\frac{\chi^2(t - \tilde{u})}{\tilde{u}} +\frac{1}{2}\int_{t - \tilde{u}}^{t - \epsilon} \left|\frac{\chi }{t - r}\right|^2 \, dr+2\int_{t - \tilde{u}}^{t - \epsilon} |\partial_r \chi|^2\, dr,
    		\end{aligned}
    	\end{equation*}
        where Young's inequality was used in the last step.	
        Since $\tilde{u}\in[\epsilon,5\epsilon]$, we obtain
        \begin{equation*}
        	\int_{t - \tilde{u}}^{t - \epsilon} \left|\frac{\chi}{t - r}\right|^2 \, dr\leq4\int_{t - \tilde{u}}^{t - \epsilon} |\partial_r \chi|^2\, dr.
        \end{equation*}
    	On the other hand, we observe that
    	\begin{equation*}
    		|\partial_r\varphi|^2r^2=\Big|\partial_r\chi-\frac{\chi}{r}\Big|^2=|\partial_r\chi|^2-\partial_r\Big(\frac{\chi^2}{r}\Big),
    	\end{equation*}
    	which implies
    	\begin{equation*}
    			\int_{t - \tilde{u}}^{t - \epsilon} |\partial_r \chi|^2\, dr=\frac{\chi^2}{r}\Big|_{t-\tilde{u}}^{t-\epsilon}+\int_{t - \tilde{u}}^{t - \epsilon} |\partial_r\varphi|^2r^2\, dr
    			\leq\int_{t - \tilde{u}}^{t - \epsilon} |\partial_r\varphi|^2r^2\, dr.
    	\end{equation*}
    	Combining the above inequalities, we conclude that
    	\begin{equation*}
    		\int_{t - \tilde{u}}^{t - \epsilon} \Big|\frac{\varphi}{t - r}\Big|^2r^2 \, dr\leq4\int_{t - \tilde{u}}^{t - \epsilon} |\partial_r \varphi|^2r^2\, dr.
    	\end{equation*}
    	Integrating over the unit sphere $\mathbb{S}^2$ yields the desired estimate~\eqref{est:hardy_ex}.
    \end{proof}

    We next record a Sobolev estimate in $\mathcal{D}^{ex}$, which can be shown adapting the proof of Lemma~\ref{lem:point_local} to the exterior region. Note the key factor of the smallness parameter $\epsilon^{1/2}$ that appears. Additionally, since we are in the exterior region, we also now gain decay in $r$ compared to Lemma \ref{lem:point_local}. 
    \begin{lemma}\label{lem:pointwise_ex}
		For any point $(t,x)\in\mathcal{D}^{ex}$, the following pointwise decay estimate holds:
		\begin{equation}\label{est:pointwise_ex}
			|f(t,x)| \lesssim r^{-1} \epsilon^{\frac{1}{2}} \|\partial_r\Omega^{\le 2}f\|_{L^2(\Sigma_{t,5\epsilon}^{ex})}.
		\end{equation}
	\end{lemma}
	

\subsection{Bootstrap and global existence argument}
    
	We now initiate the bootstrap argument by making the following assumptions in the exterior region:
	\begin{equation}\label{est:boots_ex}
		E^{ex}(5\epsilon,t, \partial^I\Omega^J \phi)^{\frac{1}{2}} \le C_3\epsilon^{-|I|}, \quad |I|+|J| \le 5,\quad\mbox{for}\ \ t\in[t_1,T^{*}).
	\end{equation}

	Here, $C_3>C_2\gg1$ is a bootstrap constant to be chosen later, and $\epsilon$ denotes the size of the initial data. We define
    \begin{equation*}
        T^{*}:=\sup\{s\in(t_1,+\infty):~\eqref{est:boots_ex}\  \mbox{holds on}\ [t_1,s]\}.
    \end{equation*}
    
	Before improving these assumptions, we prove several immediate consequences.

	\begin{proposition}\label{prop:estimates_ex1}
		Under the bootstrap assumptions~\eqref{est:boots_ex}, the following hold on $[t_1,T^{*})$:
			\begin{align}
				\|\partial\partial^I \Omega^J \phi\|_{L^2(\Sigma_{t,5\epsilon}^{ex})}+\left(\int_{\mathcal{B}_{5\epsilon}}|G\partial^I \Omega^J\phi|^2\, dx\right)^{1/2} &\lesssim C_3\epsilon^{-|I|}, \quad &|I|+|J| \le 5,\label{est:L2_ex1}\\
                    \|\partial^I \Omega^J \phi\|_{L^2(\Sigma_{t,5\epsilon}^{ex})} &\lesssim C_3\epsilon^{1-|I|}, \ \quad &|I|+|J| \le 5,\label{est:L2_ex2}\\
				|\partial^I\Omega^J \phi|&\lesssim C_3 \epsilon^{\frac{1}{2}-|I|} r^{-1}, \quad\ \quad\quad &|I|+|J| \le 3.\label{est:L2_ex3}
			\end{align}
	\end{proposition}
    \begin{proof}
        The estimate~\eqref{est:L2_ex1} follows directly from the bootstrap bounds~\eqref{est:boots_ex}.
        For~\eqref{est:L2_ex2}, we apply the Hardy inequality in Lemma~\ref{lem:hardy_ex} to $\partial^I\Omega^J\phi$ and obtain
        \begin{equation*}
            \|(t-r)^{-1}\partial^I\Omega^J\phi\|_{L^2(\Sigma_{t,5\epsilon}^{ex})}\lesssim\|\partial_r\partial^I\Omega^J\phi\|_{L^2(\Sigma_{t,5\epsilon}^{ex})}.
        \end{equation*}
        Since $\epsilon\le t-r\le 5\epsilon$ on the slice $\Sigma_{t,5\epsilon}^{ex}$, we have
        \begin{equation*}
            \|\partial^I\Omega^J\phi\|_{L^2(\Sigma_{t,5\epsilon}^{ex})}\lesssim\epsilon\| (t-r)^{-1}\partial^I\Omega^J\phi\|_{L^2(\Sigma_{t,5\epsilon}^{ex})}
            \lesssim C_3\epsilon^{1-|I|},
        \end{equation*}
        where we used~\eqref{est:L2_ex1} in the last step.
        Finally, the pointwise estimate~\eqref{est:L2_ex3} follows from combining~\eqref{est:L2_ex1}  with the Sobolev estimate \eqref{est:pointwise_ex}.
    \end{proof}

        \begin{lemma}[Exterior bounds on the nonlinearities] \label{lem:nonlinear_ex1}
        Under the bootstrap assumptions~\eqref{est:boots_ex},
		\begin{equation}\label{est:nonlinear_ex1}
            \begin{aligned}
                \|\partial^I\Omega^J\mathcal{N}(\psi,\phi)\|_{{L^2(\Sigma_{t,5\epsilon}^{ex})}}
                \lesssim C_3^p\epsilon^{\frac{p-5}{2}-|I|}t^{-p+4+4\delta}+C_3^p\epsilon^{1-\delta-|I|}t^{-\frac{p}{2}+\frac{3}{2}+(p-3)\delta}, \quad |I|+|J|\leq5.
            \end{aligned}
		\end{equation}
	\end{lemma}

        \begin{proof} $\bullet$ We first prove the bound when $|I|+|J|=0$. From the elementary inequality \eqref{eq:elementary}, we take $\psi$ in $L^\infty$ and apply Proposition~\ref{prop:estimates_ex1} and both the pointwise estimates in \eqref{eq:disp-psi01} to find 
        \begin{equation*}
            \begin{aligned}
                \|\mathcal{N}(\psi,\phi) \|_{L^2(\Sigma_{t,5\epsilon}^{ex})} \lesssim\|\phi\|_{L^2(\Sigma_{t,5\epsilon}^{ex})}\|\phi\|_{L^\infty(\Sigma_{t,5\epsilon}^{ex})}^{p-1}+\|\phi\|_{L^2(\Sigma_{t,5\epsilon}^{ex})}\|\psi\|_{L^\infty(\Sigma_{t,5\epsilon}^{ex})}^{p-1}
                \lesssim C_3^{p}\epsilon^{\frac{3}{2}-\delta}t^{-\frac{p}{2}+1+(p-2)\delta}.
            \end{aligned}
        \end{equation*}

    $\bullet$ We next consider the case when $1\leq|I|+|J|\leq5$. As in the local analysis and following Remark \ref{rem:diff_nonlinearity}, applying Leibniz rule yields
        \begin{equation*}
            \begin{aligned}
                \|\partial^I \Omega^J \mathcal{N}(\psi,\phi)\|_{L^2(\Sigma_{t,5\epsilon}^{ex})}
                &\lesssim\sum_{\substack{I_{1\to 5}=I\\J_{1\to5}=J}}\Big( \sum_{0\leq m\leq5}\Big\||\phi|^{p-5}\prod_{k=1}^{m}|\partial^{I_k}\Omega^{J_k}\phi|\prod_{l=m+1}^{5}|\partial^{I_l}\Omega^{J_l}\psi|\Big\|_{L^2(\Sigma_{t,5\epsilon}^{ex})}\\
                &\qquad \qquad \quad +\sum_{1\leq m\leq5}\Big\||\psi|^{p-5}\prod_{k=1}^{m}|\partial^{I_k}\Omega^{J_k}\phi|\prod_{l=m+1}^{5}|\partial^{I_l}\Omega^{J_l}\psi|\Big\|_{L^2(\Sigma_{t,5\epsilon}^{ex})}\Big).
            \end{aligned}
        \end{equation*}

        Each term is estimated analogously to the local case by placing the highest-order derivative in $L^2$ and controlling all remaining factors in $L^\infty$, using the estimates on $\phi$ in Proposition~\ref{prop:estimates_ex1}, the energy bound on $\psi$ from~\eqref{eq:BA-disp02}, and the pointwise bounds on $\psi$ from~\eqref{eq:disp-psi01}. In the exterior region, the pointwise estimates of $\psi$ and $\phi$ yield the additional time decay.
        Consequently, for $1\leq|I|+|J|\leq5$, we obtain
        \begin{equation*}
            \begin{aligned}
            \|\partial^I\Omega^J\mathcal{N}(\psi,\phi)\|_{L^2(\Sigma_{t,5\epsilon}^{ex})}
            \lesssim C_3^p\epsilon^{\frac{p-5}{2}-|I|}t^{-p+4+4\delta}+C_3^p\epsilon^{1-\delta-|I|}t^{-\frac{p}{2}+\frac{3}{2}+(p-3)\delta}.
            \end{aligned}
        \end{equation*}
	\end{proof}

	\begin{proof}[Proof of Proposition \ref{prop:global_ext_existence}]
	We are now in a position to close the bootstrap assumptions~\eqref{est:boots_ex}. First, by applying the $L^2$ estimates in Proposition~\ref{prop:estimates_local1}, we have
	\begin{equation*}
		\mathcal{E}^{ex}(5\epsilon,t_1, \partial^I\Omega^J \phi)^{\frac{1}{2}}
		\lesssim\|\partial\partial^I\Omega^J \phi\|_{{L^2(\Sigma_{t_1,5\epsilon}^{ex})}}\lesssim C_2\epsilon^{-|I|}.
	\end{equation*}
    Then by the energy estimate~\eqref{est:natural_ex} and Lemma~\ref{lem:nonlinear_ex1}, we deduce that
	\begin{equation}\label{eq:E-ex-partialOmega}
		\begin{aligned}
			E^{ex}(5\epsilon,t, \partial^I\Omega^J \phi)^{\frac{1}{2}} 
			&\lesssim 
			\mathcal{E}^{ex}(5\epsilon,t_1, \partial^I\Omega^J \phi)^{\frac{1}{2}}
            + \int_{t_1}^t \|\partial^I\Omega^J\mathcal{N}(\psi,\phi)\|_{{L^2(\Sigma_{s,5\epsilon}^{ex})}}\, ds\\
            &\lesssim C_2\epsilon^{-|I|}+\int_{t_1}^{t}C_3^p\epsilon^{\frac{p-5}{2}-|I|}s^{-p+4+4\delta}\, ds+\int_{t_1}^{t}C_3^p\epsilon^{1-\delta-|I|}s^{-\frac{p}{2}+\frac{3}{2}+(p-3)\delta}\, ds\\
            &\lesssim C_2\epsilon^{-|I|}+C_3^p\epsilon^{\frac{p-5}{2}-|I|}+C_3^p\epsilon^{1-\delta-|I|}.
		\end{aligned}
	\end{equation}
	This estimate strictly improves the bootstrap bounds in \eqref{est:boots_ex} provided that $C_2\ll C_3,\ C_3^{p-1}\epsilon^{\frac{p-5}{2}}\ll1,\ C_3^{p-1}\epsilon^{1-\delta}\ll1$. Consequently, the bootstrap argument closes, and the solution exists globally in the exterior region.
    Similar to~\eqref{eq:E-ex-partialOmega}, we also have the following energy bound 
    \begin{equation}\label{eq:E-ex-partialOmega2}
		\begin{aligned}
			E^{ex}(4\epsilon,t, \partial^I\Omega^J \phi)^{\frac{1}{2}} 
            \lesssim C_3\epsilon^{-|I|},
		\end{aligned}
	\end{equation}
    which we will use in the next section to control certain boundary flux term.
    \end{proof}
	
\subsection{Boundary flux estimates}\label{subsec:flux}	
	In order to prove global existence in the interior region, we need to control certain  flux terms along the null boundary $\mathcal{B}_{4\epsilon}$, shown in Figure \ref{fig:spacetime}, that we will obtain from the exterior estimates. Since in the interior region we need to use both the natural and conformal energy in our bootstrap argument (unlike the exterior region where we only required the natural energy), these boundary flux terms will involve both natural and conformal energy components. For this reason, in the following proposition we derive several conformal energy estimates in the exterior region.     
    The last estimate in \eqref{eq:E-con-Lpphi}, involving the natural energy, will actually be used later in Section \ref{sec:classical} when we establish regularity for our global solution.

	\begin{proposition}\label{prop:energy_ex}
    Suppose that 
    \begin{equation} \label{eq:C3peps_assumption}
        C_3^{p-1}\epsilon^{(\frac{1}{2}-\delta)\frac{p-5}{2}}<1,\quad C_3^{p-1}\epsilon^{1-\delta}<1, \quad 0<\delta<\tfrac{p-5}{2(p+3)}.
    \end{equation}
	Then in the exterior region $\mathcal{D}^{ex}$, for any $\tilde{u}\in[2\epsilon,5\epsilon],$ we have the following conformal energy estimates
				\begin{equation}\label{eq:E-con-pOphi}\begin{split}
					E_{con}^{ex}(\tilde{u}, t, \partial^I \Omega^J \phi)^{\frac{1}{2}}
					&\lesssim C_3\epsilon^{1-|I|}, \qquad |I|+|J|\le 3,\\
                    E_{con}^{ex}(\tilde{u}, t, \partial^I \Omega^J \phi)^{\frac{1}{2}}
					&\lesssim C_3\epsilon^{1-|I|}, \qquad |I|+|J|=4,\quad|I|\ge1,\\
					E_{con}^{ex}(\tilde{u}, t, \partial^I \Omega^J \phi)^{\frac{1}{2}}
					&\lesssim C_3\epsilon^{1-|I|}, \qquad |I|+|J|=5,\quad|I|\ge2,
				\end{split}\end{equation}
            as well as the following natural and conformal energy estimates for $L_\alpha\phi$
                \begin{equation}\label{eq:E-con-Lpphi}\begin{split}
                        E_{con}^{ex}(\tilde{u}, t, L_\alpha \phi)^{\frac{1}{2}}
					    &\lesssim C_3\epsilon,  \qquad\alpha=0,\cdots,3,\\
					    E_{con}^{ex}(\tilde{u}, t, L_a\partial^I\phi)^{\frac{1}{2}}
					    &\lesssim C_2\epsilon^{-|I|}+C_3, \qquad   a=1,2,3,\ |I|=1,2,\\
                        E^{ex}(\tilde{u}, t, L_a\partial^I\phi)^{\frac{1}{2}}
					      &\lesssim C_2\epsilon^{-|I|}+C_3,     \qquad   \ a=1,2,3,\ |I|=1,2.
                \end{split}\end{equation}
	\end{proposition}

        \begin{proof}
		$\bullet$ We first prove \eqref{eq:E-con-pOphi} by considering the conformal energy estimates on $\partial^I\Omega^J\phi$ and dividing into cases depending on the number of vector fields. 
        
        \textbf{Case 1. $|I|+|J|\leq3$.} Note that in this case we can always estimate $\psi$ in $L^\infty$ using \eqref{eq:disp-psi01} and we can arrange for the term which remains in $L^2$ to be $\phi$. We break up the nonlinear terms as follows
        \begin{equation*}
            \begin{aligned}
                \|(t+r)&\partial^I\Omega^J\mathcal{N}(\psi,\phi)\|_{L^2(\Sigma_{t,\tilde{u}}^{ex})} \\
                &\lesssim \sum_{\substack{I_{1\to 3}=I\\J_{1\to3}=J}}\Big( \sum_{0\leq m\leq3}\Big\|t|\phi|^{p-3}\prod_{k=1}^{m}|\partial^{I_k}\Omega^{J_k}\phi|\prod_{l=m+1}^{3}|\partial^{I_l}\Omega^{J_l}\psi|\Big\|_{L^2(\Sigma_{t,\tilde{u}}^{ex})}\\
                &+\sum_{1\leq m\leq3}\Big\|t|\psi|^{p-3}\prod_{k=1}^{m}|\partial^{I_k}\Omega^{J_k}\phi|\prod_{l=m+1}^{3}|\partial^{I_l}\Omega^{J_l}\psi|\Big\|_{L^2(\Sigma_{t,\tilde{u}}^{ex})}+\Big\|t\phi|\psi|^{p-4}\prod_{l=1}^{3}|\partial^{I_l}\Omega^{J_l}\psi|\Big\|_{L^2(\Sigma_{t,\tilde{u}}^{ex})}\Big)
                \\&                =:\mathcal{A}_1+\mathcal{A}_2+\mathcal{A}_3.
            \end{aligned}
        \end{equation*}
		\emph{Estimate for $\mathcal{A}_1$.} The terms in $\mathcal{A}_1$ involve at least a factor of $|\phi|^{p-3}$. We break these up depending on whether they contain at least one $\Gamma \phi$ factor, or solely involve $\Gamma \psi$:
        \begin{equation*}
            \begin{aligned}
                \mathcal{A}_1&\lesssim\sum_{\substack{I_{1\to 3}=I\\J_{1\to3}=J}} \Big( \sum_{1\leq m\leq3}\Big\|t|\phi|^{p-3}\prod_{k=1}^{m}|\partial^{I_k}\Omega^{J_k}\phi|\prod_{l=m+1}^{3}|\partial^{I_l}\Omega^{J_l}\psi|\Big\|_{L^2(\Sigma_{t,\tilde{u}}^{ex})}
                +\Big\|t|\phi|^{p-3}\prod_{l=1}^{3}|\partial^{I_l}\Omega^{J_l}\psi|\Big\|_{L^2(\Sigma_{t,\tilde{u}}^{ex})}\Big)
                \\&
                =:\mathcal{A}_1^1+\mathcal{A}_1^2.
            \end{aligned}
        \end{equation*}
Using Proposition~\ref{prop:estimates_ex1} and~\eqref{eq:disp-psi01} on the terms in $\mathcal{A}_{1}$, we have
		\begin{equation*}
			\begin{aligned}
				\mathcal{A}_{1}^1&\lesssim\sum_{\substack{I_{1\to 3}=I\\J_{1\to3}=J}} \sum_{1\leq m\leq3}\|\partial^{I_1}\Omega^{J_1}\phi\|_{L^2(\Sigma_{t,\tilde{u}}^{ex})}\|t|\phi|^{p-3}\|_{L^\infty(\Sigma_{t,\tilde{u}}^{ex})}\cdot\prod_{k=2}^m\|\partial^{I_k}\Omega^{J_k}\phi\|_{L^\infty(\Sigma_{t,\tilde{u}}^{ex})}\prod_{l=m+1}^3\|\partial^{I_l}\Omega^{J_l}\psi\|_{L^\infty(\Sigma_{t,\tilde{u}}^{ex})}\\
                &\lesssim C_3^p\epsilon^{\frac{p}{2}-\delta-|I|}t^{-p+\frac{7}{2}+\delta}.
			\end{aligned}
		\end{equation*}
Similarly, we obtain
		\begin{equation*}
			\begin{aligned}
				\mathcal{A}_1^{2}&\lesssim\sum_{\substack{I_{1\to 3}=I\\J_{1\to3}=J}} \|\phi\|_{L^2(\Sigma_{t,\tilde{u}}^{ex})}\|t|\phi|^{p-4}\|_{L^\infty(\Sigma_{t,\tilde{u}}^{ex})}\prod_{l=1}^3\|\partial^{I_l}\Omega^{J_l}\psi\|_{L^\infty(\Sigma_{t,\tilde{u}}^{ex})}
                \lesssim C_3^p\epsilon^{\frac{p-1}{2}-\delta}t^{-p+4+2\delta}.
			\end{aligned}
		\end{equation*}
\emph{Estimate for $\mathcal{A}_2$.} The terms in $\mathcal{A}_2$ involve at least a factor of $|\psi|^{p-3}$. We again apply Proposition~\ref{prop:estimates_ex1} and~\eqref{eq:disp-psi01} to find
		\begin{equation*}
			\begin{aligned}
				\mathcal{A}_2
                &\lesssim\sum_{\substack{I_{1\to 3}=I\\J_{1\to3}=J}}  \sum_{1\leq m\leq 3}\|\partial^{I_1}\Omega^{J_1}\phi\|_{L^2(\Sigma_{t,\tilde{u}}^{ex})}\|t|\psi|^{p-3}\|_{L^\infty(\Sigma_{t,\tilde{u}}^{ex})}\cdot\prod_{k=2}^m\|\partial^{I_k}\Omega^{J_k}\phi\|_{L^\infty(\Sigma_{t,\tilde{u}}^{ex})}\prod_{l=m+1}^3\|\partial^{I_l}\Omega^{J_l}\psi\|_{L^\infty(\Sigma_{t,\tilde{u}}^{ex})}\\
                &\lesssim C_3^p\epsilon^{1+(\frac{1}{2}-\delta)\frac{p-5}{2}-|I|}t^{-\frac{p-1}{4}+\frac{p+3}{2}\delta}.
			\end{aligned}
		\end{equation*}
\emph{Estimate for $\mathcal{A}_3$.} Lastly the terms in $\mathcal{A}_3$ involve at least a factor of $|\psi|^{p-4}$ and no derivatives on $\phi$. 
        \begin{equation*}
            \begin{aligned}
                \mathcal{A}_3&\lesssim\sum_{\substack{I_{1\to 3}=I\\J_{1\to3}=J}} \|\phi\|_{L^2(\Sigma_{t,\tilde{u}}^{ex})}\|\psi\|_{L^\infty(\Sigma_{t,\tilde{u}}^{ex})}^{p-4}\prod_{l=1}^3\|t^{\frac{1}{3}}\partial^{I_l}\Omega^{J_l}\psi\|_{L^\infty(\Sigma_{t,\tilde{u}}^{ex})}
                \lesssim C_3^p\epsilon^{1+(\frac{1}{2}-\delta)\frac{p-5}{2}}t^{-\frac{p-1}{4}+\frac{p+3}{2}\delta}.
            \end{aligned}
        \end{equation*}
Thus, for $|I|+|J|\leq3$, we have
        \begin{equation*}
            \begin{aligned}
                \|(t+r)\partial^I\Omega^J\mathcal{N}(\psi,\phi)\|_{L^2(\Sigma_{t,\tilde{u}}^{ex})}\lesssim C_3^p\epsilon^{1+(\frac{1}{2}-\delta)\frac{p-5}{2}-|I|}t^{-\frac{p-1}{4}+\frac{p+3}{2}\delta},
            \end{aligned}
        \end{equation*}
        where $0<\delta<\frac{p-5}{2(p+3)}.$

        Applying the conformal energy estimate~\eqref{est:conformal_ex}, we conclude that for $|I|+|J|\leq3$
    	\begin{equation*}
    		\begin{aligned}
    			E^{ex}_{con}(\tilde{u},t, \partial^I \Omega^J \phi)^{\frac{1}{2}}&\lesssim\mathcal{E}^{ex}_{con}(\tilde{u},t_1, \partial^I \Omega^J \phi)^{\frac{1}{2}}+\int_{t_1}^{t}\|(s+r)\partial^I\Omega^J\mathcal{N}(\psi,\phi)\|_{L^2(\Sigma_{s,\tilde{u}}^{ex})}\d s\\
    			&\lesssim C_2\epsilon^{1-|I|}+C_3^p\epsilon^{1+(\frac{1}{2}-\delta)\frac{p-5}{2}-|I|},
    		\end{aligned}
    	\end{equation*}
        where we used Lemma~\ref{lem:local-estimates}.
		
		\textbf{Case 2. $|I|+|J|=4$ with $|I|\ge1$.} Note that, unlike the previous case, we can no longer always estimate $\psi$ in $L^\infty$. However, since $\|\Gamma^{\leq 5} \psi \|\sim \epsilon^{-1}$ and $\| \partial \Gamma^{\leq 5} \psi\| \sim 1$, we  need to take at least one derivative in $\partial$ as otherwise we  would lose the necessary powers of $\epsilon$. We first divide up our terms depending on the coefficients $|\phi|^{p-4}$ or $|\psi|^{p-4}$ and the mixture of products $\Gamma \phi, \Gamma \psi$ that appear. 
        \begin{equation*}
            \begin{aligned}
                \|&(t+r)\partial^I\Omega^J\mathcal{N}(\psi,\phi)\|_{L^2(\Sigma_{t,\tilde{u}}^{ex})}
                \\&\lesssim \sum_{\substack{I_{1\to 4}=I\\J_{1\to4}=J}}  \Big(\sum_{|K_1|, \ldots, |K_4|=1}\Big\|t|\phi|^{p-4}\prod_{l=1}^{4}|\partial^{I_l}\Omega^{J_l}\psi|\Big\|_{L^2(\Sigma_{t,\tilde{u}}^{ex})}
                +\sum_{1\leq m\leq4}\Big\|t|\phi|^{p-4}\prod_{k=1}^{m}|\partial^{I_k}\Omega^{J_k}\phi|\prod_{l=m+1}^{4}|\partial^{I_l}\Omega^{J_l}\psi|\Big\|_{L^2(\Sigma_{t,\tilde{u}}^{ex})}
                \\& +\sum_{1\leq m\leq 4}\Big\|t|\psi|^{p-4}\prod_{k=1}^{m}|\partial^{I_k}\Omega^{J_k}\phi|\prod_{l=m+1}^{4}|\partial^{I_l}\Omega^{J_l}\psi|\Big\|_{L^2(\Sigma_{t,\tilde{u}}^{ex})}
                +\Big\|t\phi|\psi|^{p-5}\prod_{l=1}^{4}|\partial^{I_l}\Omega^{J_l}\psi|\Big\|_{L^2(\Sigma_{t,\tilde{u}}^{ex})}\Big)
                \\&=:\mathcal{A}_1+\mathcal{A}_2+\mathcal{A}_3+\mathcal{A}_4.
            \end{aligned}
        \end{equation*}
        \emph{Estimate for $\mathcal{A}_1$.} Applying Proposition~\ref{prop:estimates_ex1} and~\eqref{eq:disp-psi01}, we deduce 
		\begin{equation*}
		    \begin{aligned}
		        \mathcal{A}_1&\lesssim\sum_{\substack{I_{2\to 4}=|I|-1\\J_{2\to 4}=J} }\sum_{|K_2|,\ldots, |K_4|=1}\|\partial\psi\|_{L^2(\Sigma_{t,\tilde{u}}^{ex})}\|t|\phi|^{p-4}\|_{L^\infty(\Sigma_{t,\tilde{u}}^{ex})}\prod_{l=2}^{4}\|\partial^{I_l}\Omega^{J_l}\psi\|_{L^\infty(\Sigma_{t,\tilde{u}}^{ex})}\\
                &\lesssim C_3^p\epsilon^{\frac{p-3}{2}-\delta}t^{-p+4+2\delta}.
		    \end{aligned}
		\end{equation*}
        \emph{Estimate for $\mathcal{A}_2$.} We can break up the estimate for $\mathcal{A}_2$ as follows:
        \begin{equation*}
            \begin{aligned}
                \mathcal{A}_2&\lesssim\sum_{\substack{I_{1\to 4}=I\\J_{1\to4}=J}} \Big( \Big\|t|\phi|^{p-4}\prod_{k=1}^{4}|\partial^{I_k}\Omega^{J_k}\phi|\Big\|_{L^2(\Sigma_{t,\tilde{u}}^{ex})}
                +\sum_{1\leq m\leq3}\Big\|t|\phi|^{p-4}\prod_{k=1}^{m}|\partial^{I_k}\Omega^{J_k}\phi|\prod_{l=m+1}^{4}|\partial^{I_l}\Omega^{J_l}\psi|\Big\|_{L^2(\Sigma_{t,\tilde{u}}^{ex})}\Big) \\&
                =:\mathcal{A}_2^1+\mathcal{A}_2^2.
            \end{aligned}
        \end{equation*}
Using Proposition~\ref{prop:estimates_ex1}, we have 
		\begin{equation*}
			\begin{aligned}
				\mathcal{A}_{2}^1&\lesssim
                \sum_{\substack{I_{1\to 4}=I\\J_{1\to4}=J}}\sum_{|K_1|\geq |K_2|, \ldots, |K_4|} \|\partial^{I_1}\Omega^{J_1}\phi\|_{L^2(\Sigma_{t,\tilde{u}}^{ex})}\|t|\phi|^{p-4}\|_{L^\infty(\Sigma_{t,\tilde{u}}^{ex})}
                \prod_{k=2}^4\|\partial^{I_k}\Omega^{J_k}\phi\|_{L^\infty(\Sigma_{t,\tilde{u}}^{ex})}
                \lesssim C_3^p\epsilon^{\frac{p+1}{2}-|I|}t^{-p+2}.
			\end{aligned}
		\end{equation*}
For the $\mathcal{A}_2^{2}$ term, we need to separately treat the nonlinearity involving $\partial \Gamma^3 \psi$, since this needs to be controlled in $L^2$. We have
		\begin{equation*}
			\begin{aligned}
				\mathcal{A}_2^{2}&\lesssim\sum_{\substack{I_{1\to 4}=I\\J_{1\to4}=J}} \sum_{\substack{|K_2|,\ldots, |K_4|\leq3\\1\leq m\leq 3}}\|\partial^{I_1}\Omega^{J_1}\phi\|_{L^2(\Sigma_{t,\tilde{u}}^{ex})}\|t|\phi|^{p-4}\|_{L^\infty(\Sigma_{t,\tilde{u}}^{ex})}\\&\qquad\qquad \cdot\prod_{k=2}^m\|\partial^{I_k}\Omega^{J_k}\phi\|_{L^\infty(\Sigma_{t,\tilde{u}}^{ex})}\prod_{l=m+1}^4\|\partial^{I_l}\Omega^{J_l}\psi\|_{L^\infty(\Sigma_{t,\tilde{u}}^{ex})}\\
                &+\sum_{\substack{|I|\geq1,\\ 1\leq m\leq3}}\|\partial^I\Omega^J\psi\|_{L^2(\Sigma_{t,\tilde{u}}^{ex})}\|t|\phi|^{p-4+m}\|_{L^\infty(\Sigma_{t,\tilde{u}}^{ex})}\|\psi\|^{3-m}_{L^\infty(\Sigma_{t,\tilde{u}}^{ex})}\\
                &\lesssim C_3^p\epsilon^{\frac{p-1}{2}-2\delta-|I|}t^{-p+4+2\delta}.
			\end{aligned}
		\end{equation*}
\emph{Estimate for $\mathcal{A}_3$.} The same argument applies to $\mathcal{A}_3$, yielding the following estimate
        \begin{equation*}
            \begin{aligned}
                \mathcal{A}_3&\lesssim\sum_{\substack{I_{1\to 4}=I\\J_{1\to4}=J}} \Big(\Big\|t|\psi|^{p-4}\prod_{k=1}^{4}|\partial^{I_k}\Omega^{J_k}\phi|\Big\|_{L^2(\Sigma_{t,\tilde{u}}^{ex})}
                +\sum_{1\leq m\leq 3}\Big\|t|\psi|^{p-4}\prod_{k=1}^{m}|\partial^{I_k}\Omega^{J_k}\phi|\prod_{l=m+1}^{4}|\partial^{I_l}\Omega^{J_l}\psi|\Big\|_{L^2(\Sigma_{t,\tilde{u}}^{ex})}\Big)\\
                &\lesssim C_3^p\epsilon^{1+(\frac{1}{2}-\delta)\frac{p-5}{2}-|I|}t^{-\frac{p-1}{4}+\frac{p+3}{2}\delta}+C_3^p\epsilon^{2-\delta-|I|}t^{-\frac{p}{2}+\frac{3}{2}+(p-3)\delta}.
            \end{aligned}
        \end{equation*}
\emph{Estimate for $\mathcal{A}_4$.} Applying Proposition~\ref{prop:estimates_ex1},~\eqref{eq:BA-disp02}, and~\eqref{eq:disp-psi01}, we obtain
        \begin{equation*}
            \begin{aligned}
                \mathcal{A}_4&\lesssim\|\partial^I\Omega^J\psi\|_{L^2(\Sigma_{t,\tilde{u}}^{ex})}\|t\phi\|_{L^\infty(\Sigma_{t,\tilde{u}}^{ex})}\|\psi\|_{L^\infty(\Sigma_{t,\tilde{u}}^{ex})}^{p-2}
                \\& +\sum_{\substack{I_{1\to 4}=I\\J_{1\to4}=J}} \sum_{|K_1|,\ldots, |K_4|\leq3}\|\phi\|_{L^2(\Sigma_{t,\tilde{u}}^{ex})}\|\psi\|_{L^\infty(\Sigma_{t,\tilde{u}}^{ex})}^{p-5}\prod_{l=1}^4\|t^{\frac{1}{4}}\partial^{I_l}\Omega^{J_l}\psi\|_{L^{\infty}(\Sigma_{t,\tilde{u}}^{ex})}\\
                &\lesssim C_3^p\epsilon^{1-\delta}t^{-\frac{p}{2}+\frac{3}{2}+(p-1)\delta}.
            \end{aligned}
        \end{equation*}

        Then for $|I|+|J|=4$ with $|I|\geq1$, applying the conformal energy estimate in~\eqref{est:conformal_ex}, we have
    	\begin{equation*}
    		\begin{aligned}
    			E^{ex}_{con}(\tilde{u},t, \partial^I \Omega^J \phi)^{\frac{1}{2}}&\lesssim\mathcal{E}^{ex}_{con}(\tilde{u},t_1, \partial^I \Omega^J \phi)^{\frac{1}{2}}+\int_{t_1}^{t}\|(s+r)\partial^I\Omega^J\mathcal{N}(\psi,\phi)\|_{L^2(\Sigma_{s,\tilde{u}}^{ex})}\d s\\
    			&\lesssim C_2\epsilon^{1-|I|}+C_3^p\epsilon^{2-\delta-|I|}+C_3^p\epsilon^{1+(\frac{1}{2}-\delta)\frac{p-5}{2}-|I|}.
    		\end{aligned}
    	\end{equation*}

        \textbf{Case 3. $|I|+|J|=5$ with $|I|\ge2$.}
        \begin{equation*}
            \begin{aligned}
                &\|(t+r)\partial^I\Omega^J\mathcal{N}(\psi,\phi)\|_{L^2(\Sigma_{t,\tilde{u}}^{ex})}\\
                &\lesssim \sum_{\substack{I_{1\to 5}=I\\J_{1\to5}=J}} \Big( \sum_{1\leq m\leq4}\Big\|t|\psi|^{p-5}\prod_{k=1}^{m}|\partial^{I_k}\Omega^{J_k}\phi|\prod_{l=m+1}^{5}|\partial^{I_l}\Omega^{J_l}\psi|\Big\|_{L^2(\Sigma_{t,\tilde{u}}^{ex})}\\
                &+\sum_{\substack{|I_l|+|J_l|=1\\ 1\leq l\leq5}}\left(\Big\|t|\psi|^{p-5}\prod_{l=1}^{5}|\partial^{I_l}\Omega^{J_l}\phi|\Big\|_{L^2(\Sigma_{t,\tilde{u}}^{ex})}+\Big\|t|\phi|^{p-5}\prod_{l=1}^{5}|\partial^{I_l}\Omega^{J_l}\psi|\Big\|_{L^2(\Sigma_{t,\tilde{u}}^{ex})}\right)\\
                &+\sum_{1\leq m\leq5}\Big\|t|\phi|^{p-5}\prod_{k=1}^{m}|\partial^{I_k}\Omega^{J_k}\phi|\prod_{l=m+1}^{5}|\partial^{I_l}\Omega^{J_l}\psi|\Big\|_{L^2(\Sigma_{t,\tilde{u}}^{ex})}\Big) =:\mathcal{A}_1+\mathcal{A}_2+\mathcal{A}_3+\mathcal{A}_4.
            \end{aligned}
        \end{equation*}
 \emph{Estimate for $\mathcal{A}_1$.} Applying Proposition~\ref{prop:estimates_ex1} and~\eqref{eq:disp-psi01}, we obtain
        \begin{equation*}
            \begin{aligned}
                \mathcal{A}_1&\lesssim\sum_{\substack{I_{1\to 5}=I\\J_{1\to5}=J}}  \Big( \sum_{\substack{|K_2|, \ldots, |K_5|\leq3\\ 1\leq m\leq4}}\|\partial^{I_1}\Omega^{J_1}\phi\|_{L^2(\Sigma_{t,\tilde{u}}^{ex})}\|t|\psi|^{p-5}\|_{L^\infty(\Sigma_{t,\tilde{u}}^{ex})}
                \cdot \\ & \qquad \qquad \cdot \prod_{k=2}^m\|\partial^{I_k}\Omega^{J_k}\phi\|_{L^\infty(\Sigma_{t,\tilde{u}}^{ex})}\prod_{l=m+1}^5\|\partial^{I_l}\Omega^{J_l}\psi\|_{L^\infty(\Sigma_{t,\tilde{u}}^{ex})}\\
                &+\sum_{\substack{|K_5|\geq|K|-1,\ \\1\leq m\leq4}}\|\partial^{I_5}\Omega^{J_5}\psi\|_{L^2(\Sigma_{t,\tilde{u}}^{ex})}\prod_{k=1}^m\|t^{\frac{1}{m}}\partial^{I_k}\Omega^{J_k}\phi\|_{L^\infty(\Sigma_{t,\tilde{u}}^{ex})}  \cdot\|\psi\|_{L^\infty(\Sigma_{t,\tilde{u}}^{ex})}^{p-5}\prod_{l=m+1}^4\|\partial^{I_l}\Omega^{J_l}\psi\|_{L^\infty(\Sigma_{t,\tilde{u}}^{ex})} \Big)\\
                &\lesssim C_3^p\epsilon^{1+(\frac{1}{2}-\delta)\frac{p-5}{2}-|I|}t^{-\frac{p-1}{4}+\frac{p+3}{2}\delta}+C_3^p\epsilon^{2-\delta-|I|}t^{-\frac{p}{2}+\frac{3}{2}+(p-3)\delta}.
            \end{aligned}
        \end{equation*}\
        \emph{Estimate for $\mathcal{A}_2$.}
        \begin{equation*}
            \begin{aligned}
                \mathcal{A}_2&\lesssim \sum_{\substack{I_{1\to 5}=I\\J_{1\to5}=J}} \sum_{|K_2|, \ldots, |K_5|=1}\|\partial^{I_1}\Omega^{J_1}\phi\|_{L^2(\Sigma_{t,\tilde{u}}^{ex})}\|\psi\|^{p-5}_{L^\infty(\Sigma_{t,\tilde{u}}^{ex})}
                \prod_{k=2}^5\|t^{\frac{1}{4}}\partial^{I_k}\Omega^{J_k}\phi\|_{L^\infty(\Sigma_{t,\tilde{u}}^{ex})}\\
                &\lesssim C_3^p\epsilon^{3-|I|}t^{-3}.
            \end{aligned}
        \end{equation*}
        \emph{Estimate for $\mathcal{A}_3$.} Applying Proposition~\ref{prop:estimates_ex1} and~\eqref{eq:disp-psi01}, we obtain
        \begin{equation*}
            \begin{aligned}
                \mathcal{A}_3& \lesssim\sum_{\substack{I_{2\to 5}=|I|-1 \geq 1 \\J_{2\to5}=J}} \sum_{|K_2|, \ldots, |K_5|=1}\|\partial\psi\|_{L^2(\Sigma_{t,\tilde{u}}^{ex})}\|\phi\|^{p-5}_{L^\infty(\Sigma_{t,\tilde{u}}^{ex})}\prod_{l=2}^5\|t^{\frac{1}{4}}\partial^{I_l}\Omega^{J_l}\psi\|_{L^\infty(\Sigma_{t,\tilde{u}}^{ex})}\\
                &\lesssim C_3^p\epsilon^{\frac{p-3}{2}}t^{-p+4+4\delta}.
            \end{aligned}
        \end{equation*}
        \emph{Estimate for $\mathcal{A}_4$.} The same argument for $\mathcal{A}_3$ applies here, yielding the following result
        \begin{equation*}
            \begin{aligned}
                \mathcal{A}_4
                &\lesssim C_3^p\epsilon^{\frac{p-3}{2}-|I|}t^{-p+4+4\delta}.
            \end{aligned}
        \end{equation*}
Then for $|I|+|J|=5$ with $|I|\geq2$, applying the conformal energy estimate in~\eqref{est:conformal_ex}, we have
    	\begin{equation*}
    		\begin{aligned}
    			E^{ex}_{con}(\tilde{u},t, \partial^I \Omega^J \phi)^{\frac{1}{2}}&\lesssim\mathcal{E}^{ex}_{con}(\tilde{u},t_1, \partial^I \Omega^J \phi)^{\frac{1}{2}}+\int_{t_1}^{t}\|(s+r)\partial^I\Omega^J\mathcal{N}(\psi,\phi)\|_{L^2(\Sigma_{s,\tilde{u}}^{ex})}\d s\\
    			&\lesssim C_2\epsilon^{1-|I|}+C_3^p\epsilon^{1+(\frac{1}{2}-\delta)\frac{p-5}{2}-|I|}+C_3^p\epsilon^{2-\delta-|I|}.
    		\end{aligned}
    	\end{equation*}

	    $\bullet$ We next derive \eqref{eq:E-con-Lpphi} from the conformal energy estimates on $L_\alpha\phi$, where $\alpha=0,1,2,3$. 
        From \eqref{eq:E-con-pOphi} and~\eqref{est:L2_ex2}, we obtain
        \begin{align*}
            \|L_0\phi\|_{L^2(\Sigma_{t,\tilde{u}}^{ex})}&\lesssim\mathcal{E}_{con}^{ex}(\tilde{u}, t,\phi)^{\frac{1}{2}}+\|\phi\|_{L^2(\Sigma_{t,5\epsilon}^{ex})}\lesssim C_3\epsilon,\\
            \|L_a\phi\|_{L^2(\Sigma_{t,\tilde{u}}^{ex})}&\lesssim\mathcal{E}_{con}^{ex}(\tilde{u}, t,\phi)^{\frac{1}{2}}\lesssim C_3\epsilon,\quad a=1,2,3.
        \end{align*}
        Applying the Leibniz rule, we deduce
        \begin{align*}
            |L_\alpha\mathcal{N}(\psi,\phi)|
            &\lesssim|\phi|^{p-1}|L_\alpha\phi|+|\phi|^{p-1}|L_\alpha\psi|+|\psi|^{p-1}|L_\alpha\phi|+|\phi||\psi|^{p-2}|L_\alpha\psi|.
        \end{align*}
        Employing the above results, it follows that
	    \begin{equation*}
	    	\begin{aligned}
	    		\|(t+r)L_\alpha\mathcal{N}(\psi,\phi)\|_{L^2(\Sigma_{t,\tilde{u}}^{ex})}&\lesssim\|L_\alpha\phi\|_{L^2(\Sigma_{t,\tilde{u}}^{ex})}\|t|\phi|^{p-1}\|_{L^\infty(\Sigma_{t,\tilde{u}}^{ex})}+\|\phi\|_{L^2(\Sigma_{t,\tilde{u}}^{ex})}\|t|\phi|^{p-2}L_\alpha\psi\|_{L^\infty(\Sigma_{t,\tilde{u}}^{ex})}\\
	    		&+\|L_\alpha\phi\|_{L^2(\Sigma_{t,\tilde{u}}^{ex})}\|t|\psi|^{p-1}\|_{L^\infty(\Sigma_{t,\tilde{u}}^{ex})}+\|\phi\|_{L^2(\Sigma_{t,\tilde{u}}^{ex})}\|t|\psi|^{p-2}L_\alpha\psi\|_{L^\infty(\Sigma_{t,\tilde{u}}^{ex})}\\
	    		&\lesssim C_3\epsilon t^{-\frac{p-1}{4}+\frac{p+3}{2}\delta},
	    	\end{aligned}
	    \end{equation*}
        where we used the first two conditions in \eqref{eq:C3peps_assumption}. 
        Then from the conformal energy  estimate~\eqref{est:conformal_ex} and the pointwise bounds at $t_1$ given in Lemma~\ref{lem:local-estimates}, we have
    	\begin{equation*}
    		\begin{aligned}
    			E^{ex}_{con}(\tilde{u},t, L_\alpha \phi)^{\frac{1}{2}}&\lesssim\mathcal{E}^{ex}_{con}(\tilde{u},t_1, L_\alpha \phi)^{\frac{1}{2}}+\int_{t_1}^{t}\|(s+r)L_\alpha\mathcal{N}(\psi,\phi)\|_{L^2(\Sigma_{s,\tilde{u}}^{ex})}\d s\lesssim C_2\epsilon+C_3\epsilon\lesssim C_3\epsilon.
    		\end{aligned}
    	\end{equation*}
        
        Next, we consider the conformal energy on $L_a\partial^I\phi$ with $|I|=1,2$. The  energy estimate~\eqref{est:conformal_ex} implies
        \begin{equation*}
    		\begin{aligned}
    			&E^{ex}_{con}(\tilde{u},t, L_a\partial^I \phi)^{\frac{1}{2}}\lesssim\mathcal{E}^{ex}_{con}(\tilde{u},t_1, L_a\partial^I \phi)^{\frac{1}{2}}+\int_{t_1}^{t}\|(s+r)L_a\partial^I\mathcal{N}(\psi,\phi)\|_{L^2(\Sigma_{s,\tilde{u}}^{ex})}\d s.
    		\end{aligned}
    	\end{equation*}
        The initial term is estimated from Proposition~\ref{prop:estimates_local1}, and so 
        \begin{equation*}
            \mathcal{E}^{ex}_{con}(\tilde{u},t_1,L_a\partial^I\phi)^{\frac{1}{2}}\lesssim C_2\epsilon^{-|I|}.
        \end{equation*}
         The spacetime integral can be treated with the same argument as in the proof of \eqref{eq:E-con-pOphi} , which implies
         \begin{equation*}
             E^{ex}_{con}(\tilde{u},t, L_a\partial^I \phi)^{\frac{1}{2}}\lesssim C_2\epsilon^{-|I|}+C_3.
         \end{equation*}
        The same argument applies to $E^{ex}(\tilde{u},t,L_a\partial^I\phi)$ and yields the same estimate. 
	\end{proof}

    Thanks to our previous conformal energy estimates, we obtain in the following proposition various  pointwise decay estimates for the Lorentz and scaling vector fields applied to $\phi$. These will be used in the analysis of the interior region.
	\begin{proposition}\label{prop:decay_ex}
		In the exterior region $\mathcal{D}^{ex}$, for $\alpha=0,1,2,3$, we have the following estimates:
		\begin{align*}
			|L_\alpha \partial^I \Omega^J \phi|
			&\lesssim C_3\epsilon^{\frac{1}{2}-|I|} r^{-1},
			\quad |I|+|J|\le 2,\ |J|\leq1,\\
            |L \partial^I \Omega^J \phi|
			&\lesssim C_3\epsilon^{\frac{1}{2}-|I|} r^{-2},
			\quad |I|+|J|\le 2,\ |J|\leq1.
		\end{align*}
	\end{proposition}
        \begin{proof}
    	      Applying Lemma~\ref{lem:pointwise_ex} to $L_\alpha\partial^I\Omega^J\phi$, together with Propositions~\ref{prop:estimates_ex1} and~\ref{prop:energy_ex}, we obtain the first two estimates. For $L\partial^I\Omega^J\phi$, if follows from 
              \begin{equation*}
                  |L \partial^I \Omega^J \phi|\lesssim\langle t+r\rangle^{-1}|L_\alpha\partial^I \Omega^J \phi|.
              \end{equation*}
        \end{proof}

In order to establish global existence in the interior region, we need some refined pointwise estimates. In particular in \eqref{est:refined_ex3} we gain smallness and decay in $|L^2\phi|$ which we use later when proving the boundary flux estimates in Lemma \ref{lem:boundary}. 
	
	\begin{lemma}\label{lem:refined_ex}
		Under condition \eqref{eq:C3peps_assumption},  the following pointwise estimates hold on $\mathcal{B}_{\tilde u}$, with $2\epsilon\le\tilde u\le5\epsilon$,
            \begin{align}
                |\underline{L}\Omega^J\phi|\lesssim C_3\epsilon^{\frac{1}{2}}\langle t+r\rangle^{-1}, \quad |J|\leq1, \label{est:refined_ex1}\\
                |L\underline{L}\Omega^J\phi|\lesssim C_3\epsilon^{\frac{1}{2}}\langle t+r\rangle^{-2}, \quad |J|\leq1 .\label{est:refined_ex2}
            \end{align}
         The following pointwise estimate for $L^2\phi$ hold in $\mathcal{D}^{ex}$
            \begin{equation}\label{est:refined_ex3}
                |L^2\phi|\lesssim C_3\epsilon^{\frac{1}{2}}\langle t+r\rangle^{-3}.
            \end{equation}
	
	\end{lemma}
    \begin{proof}
    	$\bullet$ We first prove \eqref{est:refined_ex1}. First, acting $\Omega^J$ on both sides of the equations for $\phi$, and using the standard null frame identity, the PDE reads
    	\begin{equation}\label{equ:waveLL}
    		L\underline{L}\Omega^J\phi=\frac{1}{r}L\Omega^J\phi-\frac{1}{r}\underline{L}\Omega^J\phi+\frac{1}{r^2}\slashed{\Delta}_{\mathbb{S}^2}\Omega^J\phi\pm\Omega^J\mathcal{N}(\psi,\phi).
    	\end{equation}
    A direct computation yields
        \begin{equation*}
        \begin{aligned}
            L(r \underline{L} \Omega^J\phi)& = r L \underline{L} \Omega^J\phi + \underline{L} \Omega^J\phi
            = L \Omega^J\phi + \frac{1}{r}  \slashed{\Delta}_{\mathbb{S}^2}\Omega^J\phi \pm r \Omega^J\mathcal{N}(\psi,\phi).
        \end{aligned}
        \end{equation*}
    It follows from Proposition~\ref{prop:decay_ex} that
        \begin{equation*}
        	|L(r \underline{L} \Omega^J\phi)| \lesssim C_3 \epsilon^{\frac{1}{2}} (t+r)^{-2} + C_3^p \epsilon^{\frac{1}{2}+(\frac{1}{2}-\delta)\frac{p-5}{2}} (t+r)^{-\frac{p+3}{4}+\frac{p+3}{2}\delta}.
        \end{equation*}
        Applying the fundamental theorem of calculus along the $v$-direction, we deduce
        \begin{equation*}
            \begin{aligned}
                |r \underline{L} \Omega^J\phi| &\le |r \underline{L} \Omega^J\phi|_{\{t=t_1,\ t-r=\tilde{u}\}} + \int_{2+4\epsilon-\tilde{u}}^{\infty} |\partial_v(r \underline{L} \Omega^J\phi)|\, dv\lesssim C_2\epsilon^{\frac{1}{2}}+\int_{2+4\epsilon-\tilde{u}}^{\infty}|L(r\underline{L}\Omega^J)\phi|\, dv
        	\lesssim C_3 \epsilon^{\frac{1}{2}},
            \end{aligned}
        \end{equation*}
        where $C_3^{p-1}\epsilon^{(\frac{1}{2}-\delta)\frac{p-5}{2}}<1.$
       Combining the above inequality with Lemma~\ref{lem:local-estimates}, we have
       \begin{equation*}
       	|\underline{L} \Omega^J\phi| \lesssim C_3 \epsilon^{\frac{1}{2}} r^{-1}, 
       	\quad\mbox{on}\ \mathcal{B}_{\tilde{u}}, \quad 2\epsilon \le \tilde{u} \le 5\epsilon, \quad |J| \leq 1.
       \end{equation*}

      $\bullet$ We next prove \eqref{est:refined_ex2}. Injecting~\eqref{est:refined_ex1} in the equation~\eqref{equ:waveLL}, we obtain
       \begin{equation*}
           |L\underline{L}\Omega^J\phi|\lesssim C_3\epsilon^{\frac{1}{2}}\langle t+r\rangle^{-2},\quad |J|\leq1,
       \end{equation*}
       provided that we use the first and third conditions of \eqref{eq:C3peps_assumption}.
       
       $\bullet$ We conclude by proving \eqref{est:refined_ex3}. Let $|J|=0$ in~\eqref{equ:waveLL}. Then acting $L$ on both sides of~\eqref{equ:waveLL}, we obtain
       \begin{equation*}
       	\begin{aligned}
       		\underline{L} L^2 \phi &=L\left(\frac{1}{r}L\phi-\frac{1}{r}\underline{L}\phi+\frac{1}{r^2}\slashed{\Delta}_{\mathbb{S}^2}\phi\pm\mathcal{N}(\psi,\phi)\right)\\
       		&=\frac{1}{r}L^2\phi-\frac{1}{r}L\underline{L}\phi+\frac{1}{r^2}L\slashed{\Delta}_{\mathbb{S}^2}\phi-\frac{1}{r^2}L\phi+\frac{1}{r^2}\underline{L}\phi-\frac{2}{r^3}\slashed{\Delta}_{\mathbb{S}^2}\phi
            \pm L(\mathcal{N}(\psi,\phi)).
       	\end{aligned}
       \end{equation*}
       We compute, for $(t,x)\in\mathcal{D}^{ex}$,
       \begin{equation*}
       	|\underline{L} (rL^2 \phi)|\lesssim C_3\epsilon^{-\frac{1}{2}}(t+r)^{-2}.
       \end{equation*}
       Employing the fundamental theorem of calculus along the $u$-direction, we obtain
       \begin{equation*}
       \begin{aligned}
           |rL^2\phi(t,x)|&\leq\left\{\begin{aligned}
       	    &|rL^2\phi|_{\{t=t_1,\ 1-3\epsilon\leq r\leq 1+\epsilon\}}\\
            &|rL^2\phi|_{\{t-r=\epsilon\}}
       	\end{aligned}\right.+\int_\epsilon^{5\epsilon} |\partial_u(r L^2\phi)|\, du\\
        &\lesssim C_2\epsilon^{\frac{1}{2}}r^{-2}+\int_\epsilon^{5\epsilon}|\underline{L}(rL^2\phi)|\, du
        \lesssim C_3\epsilon^{\frac{1}{2}}(t+r)^{-2},
       \end{aligned}
       \end{equation*}
       where we used Lemma~\ref{lem:local-estimates} and the fact that $L^2\phi$ vanishes on $t-r=\epsilon$. Finally, in the exterior region since $t-r\in[\epsilon,5\epsilon]$ then $r\sim t\sim \langle t+r\rangle$, and thus
       \begin{equation*}
         	|L^2 \phi| \lesssim r^{-1}\,|rL^2\phi|\lesssim C_3 \epsilon^{\frac{1}{2}} \langle t+r\rangle^{-3}.
       \end{equation*}
    \end{proof}

\section{Interior global existence for short-pulse data}\label{sec:interior}
In this section, we establish the following:
        \begin{proposition}[Global in time existence in the interior region]\label{prop:global_in_existence} There exists $\epsilon>0$ sufficiently small such that \eqref{eq:model_short-pulse} admits a global in time solution in the interior region $$\mathcal{D}^{in}=\{(t,x): t\geq t_1=1+2\epsilon,\ 0\leq r\leq t-4\epsilon\}.$$ 
    \end{proposition}
    
A key point is that the boundary flux control on $\mathcal B_{4\epsilon}$ obtained in the exterior analysis (see Proposition~\ref{prop:energy_ex} and Lemma~\ref{lem:refined_ex}) provides the necessary boundary input for the interior energy inequalities. We use the hyperboloidal method and close the bootstrap by deriving refined bounds. Combining the interior and exterior arguments yields a global solution in $\mathcal D^{ex}\cup\mathcal D^{in}$ for all $t\ge t_1$.

Before presenting the proof, we introduce the geometric setup and several auxiliary definitions that are used throughout this section.

\subsection{Hyperboloidal setting and basic estimates}
In this subsection, we introduce the hyperboloidal foliation of the interior region
and collect several analytic tools that will be used in the proof of global existence. For $\tau\geq \tau_0:=\sqrt{3(1+4\epsilon)}$, we define (see also Figure \ref{fig:spacetime})
	\begin{equation*}
		\mathcal{H}_\tau:=\{(t,x)\in\mathcal{D}^{in}:(t+1)^2-r^2=\tau^2\}.
	\end{equation*}
We also decompose the hyperboloid into the two regions
  \begin{equation}\label{eq:Htau-split}
      \mathcal H_\tau=\big(\mathcal H_\tau\cap\{r\le t/4\}\big)\cup\big(\mathcal H_\tau\cap\{r\ge t/8\}\big) =:\Hcali\cup\Hcalf.
  \end{equation}
On each hyperboloid $\mathcal{H}_{\tau}$, the time coordinate satisfies $\tau\leq t+1\leq\tau^2$. 
Associated with this hyperboloidal foliation, we introduce the tangential vector fields
	\begin{align*}
		\slashed{\partial}_{\tau}=\frac{\tau}{t+1}\partial_t,\quad\slashed{\partial}_a=\frac{x_a}{t+1}\partial_t+\partial_a=-\frac{x_a}{t(t+1)}\partial_t+\frac{L_a}{t},
	\end{align*}
and define the radial tangential derivative by $\slashed{\partial}_r=\omega^a\slashed{\partial}_a$.

    For $\tau \ge \tau_0$, we define the $L^2$ norm of a function $f$ on the hyperboloid $\mathcal H_\tau$ by
    \begin{equation*}
        \|f\|_{L^2(\mathcal H_\tau)}^2
        :=\int_{\mathcal H_\tau} |f(t,x)|^2 \, dx .
    \end{equation*}
    Using the parametrization $(t,x)=\big(\sqrt{\tau^2+|y|^2}-1,\,y\big),$
    this norm can be equivalently written as
    \begin{equation*}
        \|f\|_{L^2(\mathcal H_\tau)}^2
        =\int_{r_*(\tau)\le |y|\le r^*(\tau)}
        |f(\sqrt{\tau^2+|y|^2}-1,y)|^2 \, dy, 
    \end{equation*}
    where the integration region is given by
\begin{equation}\label{eq:def_rstar}
        r_*(\tau)=
        \begin{cases}
            \sqrt{(2+2\epsilon)^2-\tau^2}, & \tau_0\le \tau \le 2+2\epsilon,\\[4pt]
            0, & \tau \ge 2+2\epsilon,
        \end{cases}
        \qquad
        r^*(\tau)=\dfrac{\tau^2-(1+4\epsilon)^2}{2(1+4\epsilon)},\quad \omega\in\mathbb{S}^2.
        \end{equation}

	\begin{lemma}[Hardy inequality]\label{lem:hardy_in}
		In the interior region $\mathcal{D}^{in}$, we have the following Hardy inequality
		\begin{equation}\label{est:hardy_in}
            \begin{aligned}
                \|r^{-1}f\|_{L^2(\mathcal{H}_\tau)}&\lesssim r_*^{\frac{1}{2}}\|f(\sqrt{\tau^2+r_*^2}-1,r_*\omega)\|_{L^2(\mathbb{S}^2)}+(r^*)^{\frac{1}{2}}\|f(r^*+4\epsilon,r^*\omega)\|_{L^2(\mathbb{S}^2)}+\|\slashed{\partial}_rf\|_{L^2(\mathcal{H}_\tau)},
            \end{aligned}
		\end{equation}
        where $r_*$ and $r^*$ are defined in \eqref{eq:def_rstar}.
	\end{lemma}
	\begin{proof}
		Let $(t,x)=(t,r\omega)\in\mathcal{H}_{\tau}$ with $\omega\in\mathbb{S}^2$,  $\tau\geq\tau_0$. Then $f(t,x)=f(\sqrt{\tau^2+|x|^2}-1,x)$. By the definition, we derive
		\begin{equation}\label{equ:L2norm}
			\|r^{-1}f\|_{L^2(\mathcal{H}_\tau)}^2=\int_{\mathcal{H}_\tau}\frac{f^2}{r^2}\, dx=\int_{\mathbb{S}^2}\int_{r_*}^{r^*}f^2(\sqrt{\tau^2+|x|^2}-1,x)\, drd\omega,
		\end{equation}
		where $r_*$ and $r^*$ are given in \eqref{eq:def_rstar}. 
		Applying the fundamental theorem of calculus and Cauchy-Schwarz inequality, we derive
		\begin{equation*}
			\begin{aligned}
				\int_{r_*}^{r^*}f^2(\sqrt{\tau^2+|x|^2}-1,x)\, dr
				&=\int_{r_*}^{r^*}\partial_r(rf^2)(\sqrt{\tau^2+|x|^2}-1,x)\, dr-2\int_{r_*}^{r^*}rf\partial_{r}f(\sqrt{\tau^2+|x|^2}-1,x)\, dr\\
				&\leq (rf^2)\Big|_{r_*}^{r^*}+\frac{1}{2}\int_{r_*}^{r^*}f^2\, dr+2\int_{r_*}^{r^*}|\slashed{\partial}_rf|^2r^2\, dr.
			\end{aligned}
		\end{equation*}
		Integrating over the unit sphere $\mathbb{S}^2$ and using~\eqref{equ:L2norm} we obtain
		\begin{equation*}
            \begin{aligned}
                \|r^{-1}f\|_{L^2(\mathcal{H}_\tau)}^2&\lesssim r_*\|f(\sqrt{\tau^2+r_*^2}-1,r_*\omega)\|^2_{L^2(\mathbb{S}^2)}+r^*\|f(r^*+4\epsilon,r^*\omega)\|_{L^2(\mathbb{S}^2)}^2+\|\slashed{\partial}_rf\|_{L^2(\mathcal{H}_\tau)}^2.
            \end{aligned}
		\end{equation*}
	
	\end{proof}
\begin{definition}
    The natural and conformal energy functionals on constant hyperboloidal time slices read
    \begin{align*}
        \mathcal{E}^{in}(\tau,\varphi) &=\frac12\|\slashed{\partial}_{\tau}\varphi\|_{L^2(\mathcal{H}_\tau)}^2+\frac12 \sum_{a}\|\slashed{\partial}_a\varphi\|_{L^2(\mathcal{H}_\tau)}^2,  \\
        \mathcal{E}_{\mathrm{con}}^{in} (\tau,\varphi)&=\frac{1}{2}\|\tau\slashed{\partial}_\tau\varphi+2x^a\slashed{\partial}_a\varphi+2\varphi\|_{L^2(\mathcal{H}_\tau)}^2+\frac{1}{2}\sum_{a}\|\tau\slashed{\partial}_a\varphi\|_{L^2(\mathcal{H}_\tau)}^2.
    \end{align*}

\end{definition}
        \begin{lemma}[Hyperboloidal energy inequality \cite{KlainermanHyp, HormanderHyp, MaHu}]\label{lem:energy_in}
		Suppose that $\varphi$ is the solution to the wave equation, then we have the natural energy estimate
			\begin{equation}\label{est:energy_in1}
				\mathcal{E}^{in}(\tau,\varphi)^{\frac{1}{2}}\lesssim\|\partial\varphi\|_{L^2(\Sigma_{t_1}\cap\mathcal{D}^{in})}+\Big(\int_{\mathcal{B}_{4\epsilon}\cap \mathcal{D}_\tau^{in}}|G\varphi|^2\, dx\Big)^{\frac{1}{2}}+\int_{\tau_0}^{\tau}\|\Box \varphi \|_{L^2(\mathcal{H}_{\tilde{\tau}})}\, d\tilde{\tau},
			\end{equation}
		and the conformal energy estimate
		    \begin{equation}\label{est:energy_in2}
		    	\begin{aligned}
		    		\mathcal{E}_{\mathrm{con}}^{in}(\tau,\varphi)^{\frac{1}{2}}
		    		&\lesssim
		    		\|\partial\varphi\|_{L^2(\Sigma_{t_1}\cap\mathcal{D}^{in})}
		    		+\|\varphi\|_{L^2(\Sigma_{t_1}\cap\mathcal{D}^{in})}\\
		    		&+\left(
		    		\int_{\mathcal{B}_{4\epsilon}\cap \mathcal{D}_\tau^{in}}
		    		\Big[(t+r+1)L\varphi + 2\varphi\Big]^2
		    		+\Big|\frac{1}{r}\Omega\varphi\Big|^{2}
		    		\, dx
		    		\right)^{\frac{1}{2}}
		    		+\int_{\tau_0}^{\tau}
		    		\|\tilde{\tau}\Box \varphi \|_{L^{2}(\mathcal{H}_{\tilde{\tau}})}
		    		\, d\tilde{\tau}.
		    	\end{aligned}
		    \end{equation}
	\end{lemma}
	\begin{proof}
		$\bullet$ We first prove \eqref{est:energy_in1}. 
        First, multiplying the equation by $\partial_t \varphi$, and integrating over the region $\mathcal{D}_\tau^{in}=\{(t,x)\in\mathcal{D}^{in}:(t+1)^2-|x|^2\leq \tau^2\}$, we deduce
		\begin{equation*}
			\begin{aligned}
				\int_{\mathcal{H}_\tau}\left(\frac{1}{2}|\partial\varphi|^2+\frac{x^a}{t+1}\partial_a\varphi\partial_t\varphi\right)\, dx=\frac{1}{2}\int_{\Sigma_{t_1}\cap \mathcal{D}^{in}}|\partial\varphi|^2\, dx
				+\frac{1}{2}\int_{\mathcal{B}_{4\epsilon}\cap \mathcal{D}_\tau^{in}}|G\varphi|^2\, dx+\int_{\mathcal{D}_\tau^{in}}\frac{\tilde{\tau}}{t+1}(-\Box \varphi)\partial_t\varphi\, dxd\tilde{\tau}.
			\end{aligned}
		\end{equation*}
	Note that
	\begin{equation*}
		\begin{aligned}
			\mathcal{E}^{in}(\tau,\varphi)&=\int_{\mathcal{H}_\tau}\big(\frac{1}{2}|\partial\varphi|^2+\frac{x^a}{t+1}\partial_a\varphi\partial_t\varphi\big)\, dx\\
		&=\frac{1}{2}\|\partial_t\varphi+\frac{x^a}   {t+1}\partial_a\varphi\|_{L^2(\mathcal{H}_\tau)}^2+\frac{1}{2}\sum_{a}\|\frac{\tau}{t+1}\partial_a\varphi\|_{L^2(\mathcal{H}_\tau)}^2+\frac{1}{2}\|\frac{\Omega\varphi}{t+1}\|_{L^2(\mathcal{H}_\tau)}^2\\
            &=\frac{1}{2}\|\slashed{\partial}_{\tau}\varphi\|_{L^2(\mathcal{H}_\tau)}^2+\frac{1}{2}\sum_{a}\|\slashed{\partial}_a\varphi\|_{L^2(\mathcal{H}_\tau)}^2.
		\end{aligned}
	\end{equation*}
        Then we deduce
    \begin{equation*}
    	\begin{aligned}
    		\mathcal{E}^{in}(\tau,\varphi)\lesssim\int_{\Sigma_{t_1}\cap \mathcal{D}^{in}}|\partial\varphi|^2\, dx
    		+\int_{\mathcal{B}_{4\epsilon}\cap \mathcal{D}_\tau^{in}}|G\varphi|^2\, dx+\int_{\tau_0}^{\tau}\|\Box\varphi\|_{L^2(\mathcal{H}_{\tilde{\tau}})}\mathcal{E}^{in}(\tilde{\tau},\varphi)^{\frac{1}{2}}\, d\tilde{\tau}.
    	\end{aligned}
    \end{equation*}
	An application of Gronwall inequality~\eqref{est:Gronwall} yields the desired result.
	
	$\bullet$ We next turn to \eqref{est:energy_in2}.  Define a multiplier as $\mathcal{M}=(r^2+(t+1)^2)\partial_t+2r(t+1)\partial_r$. 	
	Multiplying the equation by $\mathcal{M}\varphi+2(t+1)\varphi$, we obtain
	\begin{equation*}
		\begin{aligned}
			-\Box\varphi \big(\mathcal{M}\varphi+2(t+1)\varphi\big)&=\frac{1}{2}\partial_t\Big[(r^2+(t+1)^2)|\partial\varphi|^2+4r(t+1)\partial_r\varphi\partial_t\varphi+4(t+1)\varphi\partial_t\varphi+4\varphi^2+4r\varphi\partial_r\varphi\Big]\\
			&\quad -\partial^a\Big[(r^2+(t+1)^2)\partial_a\varphi\partial_t\varphi+(t+1)x_a(|\partial_t\varphi|^2-\sum_{b}|\partial_b\varphi|^2)\\
			&\quad\quad\ +2r(t+1)\partial_r\varphi\partial_a\varphi+2(t+1)\varphi\partial_a\varphi+2x_a\varphi\partial_t\varphi\Big].
		\end{aligned}
	\end{equation*}
	Integrating over the region $\mathcal{D}_\tau^{in}$, we have
	\begin{equation*}
		\begin{aligned}
			\mathcal{E}_{con}^{in}(\tau,\varphi)&\lesssim\int_{\Sigma_{t_1}\cap \mathcal{D}^{in}}\big(|\partial\varphi|^2+|\varphi|^2\big)\, dx
            +\int_{\mathcal{B}_{4\epsilon}}[(r+t+1)L\varphi+2\varphi]^2\, dx+\int_{\mathcal{B}_{4\epsilon}}\big|\frac{1}{r}\Omega\varphi|^2\, dx\\
			&+\int_{\mathcal{D}_\tau^{in}}\big|\frac{\tilde{\tau}}{t+1}(\Box \varphi)(M\varphi+2(t+1)\varphi)\big|\, dxd\tilde{\tau},
		\end{aligned}
	\end{equation*}
	where
	\begin{equation*}
		\begin{aligned}
			\mathcal{E}_{con}^{in}(\tau,\varphi)&=\frac{1}{2}\|\tau\slashed{\partial}_\tau\varphi+2x^a\slashed{\partial}_a\varphi+2\varphi\|_{L^2(\mathcal{H}_\tau)}^2+\frac{1}{2}\sum_{a}\|\tau\slashed{\partial}_a\varphi\|_{L^2(\mathcal{H}_\tau)}^2\\
			&=\frac{1}{2}\big\|\frac{\mathcal{M}\varphi}{t+1}+2\varphi\big\|_{L^2(\mathcal{H}_\tau)}^2+\frac{1}{2}\sum_{a}\|\tau\slashed{\partial}_a\varphi\|_{L^2(\mathcal{H}_\tau)}^2.
		\end{aligned}
	\end{equation*}
	Using Gronwall inequality we are done.
	\end{proof}
    
	Next, we introduce the Klainerman-Sobolev inequalities on truncated hyperboloids. For simplicity, we denote $Z=\{\partial_t, L_1, L_2,L_3\}.$ While the proof is straightforward, a key part of our analysis relies on the novel $L^q$ estimate, for $q>6$, given in \eqref{est:KS_in2}. Note that in this estimate we only need to take $\Omega^{\leq 1}$ as a vector field derivative, instead of $\Omega^{\leq 2}$, on the right hand side.

	\begin{lemma}[Klainerman-Sobolev inequalities on truncated hyperboloids]\label{lem:KS_in} \,
		\begin{enumerate}
			\item If $(t,x)\in\Hcali$, then
			\begin{equation}\label{est:KS_in1}
				|\varphi(t,x)|\lesssim\tau^{-1}t^{-\frac{1}{2}}\sum_{|I|\leq 2}\epsilon^{(|I|-\frac{3}{2})\lambda}\left\|\frac{\tau}{t+1}Z^I\varphi\right\|_{L^2(\mathcal{H}_\tau)},
			\end{equation}
            where $\lambda$ is an arbitrary nonnegative constant.
   
            \item If $(t,x)\in\Hcalf$, then
            \begin{equation}\label{est:KS_in3}
		    	\begin{aligned}
		    		|\varphi(t,x)|&\lesssim|\varphi(t^*,x^*)|
                    +\tau^{-1}|x|^{-\frac{1}{2}}\sum_{\substack{|J_1|\leq 2,\ |J_2|\leq2\\|I|=1}}\left\|\frac{\tau}{t}\Omega^{J_1}\varphi\right\|_{L^2(\mathcal{H}_\tau)}^{\frac{1}{2}}\left\|\frac{\tau}{t}\Omega^{J_2}Z^I\varphi\right\|_{L^2(\mathcal{H}_\tau)}^{\frac{1}{2}},
		    	\end{aligned}
		    \end{equation}
            where $(t^*,x^*)$ denotes the intersection point of $\mathcal{H}_\tau$ and $\mathcal{B}_{4\epsilon}$.
            
            \item Let $q>6$. If $(t_0,x_0)\in\Hcalf$, then
            \begin{equation}\label{est:KS_in2}
            	\begin{aligned}
            		\|\varphi\|_{L^q(\mathcal{H}_{\tau}\cap\{r_0\leq |x|\leq r^*\})}
            		&\lesssim 
            		(r^*)^{\frac{3}{q}}\sum_{|J|\leq1}\|\Omega^{J}\varphi(t^*,r^*\omega)\|_{L^2(\mathbb{S}^2)}\\
            		&+ \tau^{-1}r_0^{-\frac{1}{2}+\frac{3}{q}}
            		\sum_{\substack{|J|\leq 1,\ |I|=1}}\left\|\frac{\tau}{t}\Omega^{J}\varphi\right\|_{L^2(\mathcal{H}_\tau)}^{\frac{1}{2}}
            		\left\|\frac{\tau}{t}Z^I\Omega^{J}\varphi\right\|_{L^2(\mathcal{H}_\tau)}^{\frac{1}{2}},
            	\end{aligned}
            \end{equation}
            where $(t^*,x^*)$ denotes the intersection point of $\mathcal{H}_\tau$ and $\mathcal{B}_{4\epsilon}$.
		\end{enumerate}
	\end{lemma}
	\begin{proof}
		$\bullet$ We first prove \eqref{est:KS_in1}. Fix $(t_0,x_0)\in\mathcal{H}_\tau$ with $|x_0|\leq t_0/4$. For $|z|\leq1$, set
        \begin{equation*}
            x(z)=x_0+t_0\epsilon^{\lambda}z, \qquad t(z)=\sqrt{|x(z)|^2+\tau^2}-1,\qquad g(z)=\varphi(t(z),x(z)).
        \end{equation*}
        By Sobolev embedding on the unit ball $H^2(B_1)\hookrightarrow L^\infty(B_1)$,
        \begin{equation*}
            |g(0)|^2\lesssim\sum_{|I|\leq2}\int_{|z|\leq1}|\partial_z^I g(z)|^2\, dz.
        \end{equation*}
        Since $\frac{\partial t}{\partial x_a}=\frac{x_a}{t+1}$, we have the chain rule identity
        \begin{equation}\label{eq:chain1}
            \partial_{z_a}=t_0\epsilon^{\lambda}\Big(\partial_{x_a}+\frac{x_a}{t+1}\partial_t\Big)=t_0\epsilon^\lambda\slashed{\partial}_a.
        \end{equation}
		Then
        \begin{equation*}
            \begin{aligned}
                |g(0)|^2=|\varphi(t_0,x_0)|^2
                &\lesssim\sum_{|I|\leq2}\int_{|z|\leq1}|\partial_{z}^I\varphi(t(z),x(z))|^2\, dz\\
                &\lesssim\int_{|x-x_0|\leq t_0\epsilon^\lambda}|\varphi(t,x)|^2(t_0\epsilon^\lambda)^{-3}\, dx+\int_{|x-x_0|\leq t_0\epsilon^\lambda}\sum_a|\slashed{\partial}_a\varphi(t,x)|^2(t_0\epsilon^\lambda)^{-1}\, dx\\
                &\quad +\int_{|x-x_0|\leq t_0\epsilon^\lambda}\sum_{a,b}|\slashed{\partial}_a\slashed{\partial}_b\varphi(t,x)|^2(t_0\epsilon^\lambda)\, dx.
            \end{aligned}
        \end{equation*}
        Using~\eqref{eq:chain1}, we deduce
        \begin{equation*}
            \begin{aligned}
                |\slashed{\partial}_a\varphi|\lesssim\frac{1}{t}\sum_{|I|=1}|Z^I\varphi|,\qquad
                |\slashed{\partial}_a\slashed{\partial}_b\varphi|
                \lesssim\frac{1}{t^2}\sum_{1\leq|I|\leq2}|Z^I\varphi|,
            \end{aligned}
        \end{equation*}
        which implies
        \begin{equation*}
            \begin{aligned}
                |\varphi(t_0,x_0)|\lesssim\tau^{-1}t_0^{-\frac{1}{2}}\sum_{|I|\leq2}\epsilon^{(|I|-\frac{3}{2})\lambda}\Big\|\frac{\tau}{t+1}Z^I\varphi\Big\|_{L^2(\mathcal{H}_\tau)}.
            \end{aligned}
        \end{equation*}
        
		$\bullet$ We next prove \eqref{est:KS_in3}. 
        Let $(t,x)\in\mathcal{H}_\tau$ with $|x|\ge t/8$. By the fundamental theorem of calculus,
        \begin{equation}\label{eq:fundamental-varphi}
            \varphi^2(t,x)=\varphi^2(t^*,x^*)-\int_{r}^{r^*}\partial_\rho(\varphi^2(t,\rho\omega))\, d\rho,
        \end{equation}
        where $(t^*,x^*)$ denotes the intersection point of $\mathcal{H}_\tau$ and $\mathcal{B}_{4\epsilon}$, satisfying
        \begin{equation}\label{eq:intersection}
            t^*=|x^*|+4\epsilon,\qquad |x^*|=\frac{\tau^2-(1+4\epsilon)^2}{2(1+4\epsilon)}.
        \end{equation}
        Moreover, we have
        \begin{equation*}
            \begin{aligned}
                \partial_\rho(\varphi^2(t,\rho\omega))&=2\varphi\partial_\rho\varphi(t(\rho),\rho\omega),\qquad t(\rho)=\sqrt{\rho^2+\tau^2}-1,\\
                |\partial_\rho\varphi(t(\rho),\rho\omega)|&\lesssim\sum_a|\slashed{\partial}_a\varphi|\lesssim\frac{1}{t}\sum_{|I|=1}|Z^I\varphi|,\quad Z=\{\partial_t, L_1, L_2, L_3\}.
            \end{aligned}
        \end{equation*}
        Applying the above estimates to~\eqref{eq:fundamental-varphi}, we obtain
        \begin{equation*}
            \begin{aligned}
                |\varphi(t,x)|^2\leq|\varphi(t^*,x^*)|^2+2\sum_{|I|=1}\int_r^{r^*}\frac{1}{t}|\varphi||Z^I\varphi|\, d\rho,
            \end{aligned}
        \end{equation*}
        together with the Sobolev embedding on the unit sphere, $H^2(\mathbb{S}^2)\hookrightarrow L^\infty(\mathbb{S}^2)$,
        \begin{equation*}
            |\varphi(t,r\omega)|\lesssim \sum_{|J|\le 2}\|\Omega^J\varphi(t,r\omega)\|_{L^2(\mathbb{S}^2)},
        \end{equation*}
        giving
        \begin{equation*}
            \begin{aligned}
                |\varphi(t,x)|^2&\lesssim|\varphi(t^*,x^*)|^2+\sum_{\substack{|J_1|\leq2,\ |J_2|\leq2\\|I|=1}}\int_r^{r^*}\frac{1}{t}\|\Omega^{J_1}\varphi\|_{L^2(\mathbb{S}^2)}\|\Omega^{J_2}Z^I\varphi\|_{L^2(\mathbb{S}^2)}\, d\rho\\
                &\lesssim|\varphi(t^*,x^*)|^2+\tau^{-2}r^{-1}\sum_{\substack{|J_1|\leq2,\ |J_2|\leq2\\|I|=1}}\Big\|\frac{\tau}{t}\Omega^{J_1}\varphi\Big\|_{L^2(\mathcal{H}_\tau)}\Big\|\frac{\tau}{t}\Omega^{J_2}Z^I\varphi\Big\|_{L^2(\mathcal{H}_\tau)}.
            \end{aligned}
        \end{equation*}
        
        $\bullet$ Lastly we prove \eqref{est:KS_in2}. 
	    Let $q>6$ and assume $(t_0,x_0)\in\mathcal{H}_\tau$, $x_0=|x_0|\omega$ and $|x_0|\geq t_0/8,$ where $\omega\in\mathbb{S}^2$.  Fix $\omega$, $(t^*,x^*)$ is the intersection point of $\mathcal{H}_{\tau}$ and $\mathcal{B}_{4\epsilon}$, which satisfies~\eqref{eq:intersection}.
        For a fixed $\rho\in[r_0,r^*]$ we can apply the Sobolev embedding $H^1(\mathbb{S}^2)\hookrightarrow L^q(\mathbb{S}^2)$ to obtain
        \begin{equation*}
            \|\varphi(t(\rho),\rho\omega)\|_{L^q(\mathbb{S}^2)}\lesssim_q\sum_{|J|\leq1}\|\Omega^J\varphi(t(\rho),\rho\omega)\|_{L^2(\mathbb{S}^2)},\quad t(\rho)=\sqrt{\rho^2+\tau^2}-1.
        \end{equation*}
        Repeating the argument of \eqref{est:KS_in3} for $|J|\leq1$ yields
        \begin{equation*}
            \begin{aligned}
                \|\Omega^J\varphi(t(\rho),\rho\omega)\|^2_{L^2(\mathbb{S}^2)}
                &\lesssim \|\Omega^J\varphi(t^*,r^*\omega)\|^2_{L^2(\mathbb{S}^2)}
                +\tau^{-2}\rho^{-1}\sum_{|I|=1}\Big\|\frac{\tau}{t}\Omega^{J}\varphi\Big\|_{L^2(\mathcal{H}_\tau)}\Big\|\frac{\tau}{t}Z^I\Omega^{J}\varphi\Big\|_{L^2(\mathcal{H}_\tau)}.
            \end{aligned}
        \end{equation*}
        Summing over $|J|\leq1$ gives
        \begin{equation*}
            \begin{aligned}
                \|\varphi(t(\rho),\rho\omega)\|^2_{L^q(\mathbb{S}^2)}&\lesssim\sum_{|J|\leq1}\|\Omega^J\varphi(t^*,r^*\omega)\|^2_{L^2(\mathbb{S}^2)}
                +\tau^{-2}\rho^{-1}\sum_{|I|=1,\ |J|\leq1}\Big\|\frac{\tau}{t}\Omega^{J}\varphi\Big\|_{L^2(\mathcal{H}_\tau)}\Big\|\frac{\tau}{t}Z^I\Omega^{J}\varphi\Big\|_{L^2(\mathcal{H}_\tau)}.
            \end{aligned}
        \end{equation*}
        Integrating in $\rho$ we have
        \begin{equation*}
            \|\varphi\|_{L^q(\mathcal{H}_\tau\cap\{r_0\le r\le r^*\})}^q
            =\int_{r_0}^{r^*}\|\varphi(t(\rho),\rho\omega)\|_{L^q(\mathbb{S}^2)}^q\,\rho^2\,d\rho.
        \end{equation*}      
        Thus, combining with the above estimates, we obtain
        \begin{equation*}
        	\begin{aligned}
        		\|\varphi\|_{L^q(\mathcal{H}_{\tau}\cap\{r_0\le r\le r^*\})}^q
        		&\lesssim_q\sum_{|J|\leq1}\|\Omega^J\varphi(t^*,r^*\omega)\|_{L^2(\mathbb{S}^2)}^q\int_{r_0}^{r^*}\rho^2\, d\rho\\
        		&+\tau^{-q}\sum_{|I|=1,\ |J|\leq1}\Big\|\frac{\tau}{t}\Omega^{J}\varphi\Big\|^{\frac{q}{2}}_{L^2(\mathcal{H}_\tau)}\Big\|\frac{\tau}{t}Z^I\Omega^{J}\varphi\Big\|^{\frac{q}{2}}_{L^2(\mathcal{H}_\tau)}\int_{r_0}^{r^*}\rho^{-\frac{q}{2}+2}\, d\rho.
        	\end{aligned}
        \end{equation*}
        Since $q>6$, we have $-\frac{q}{2}+2<-1$, hence together with $r_0\leq r^*$,
        \begin{equation*}
            \int_{r_0}^{r^*}\rho^2\, d\rho\lesssim (r^*)^3,\qquad\int_{r_0}^{r^*}\rho^{-\frac{q}{2}+2}\, d\rho\lesssim r_0^{-\frac{q}{2}+3}.
        \end{equation*}
        Taking the $q$-th root gives~\eqref{est:KS_in2}.
	\end{proof}

    \subsection{Bootstrap assumptions}
In order to prove Proposition \ref{prop:global_in_existence}, we introduce the following bootstrap assumptions for $\phi$. For a sufficiently large constant $C_4>C_3\gg1$ to be determined later, we impose the following assumptions on $[\tau_0,\tau^*)$:
    \begin{equation}\label{est:Boot_in}
    	\begin{aligned}
    		\mathcal{E}^{in}(\tau, \phi)^{\frac{1}{2}}+\mathcal{E}^{in}(\tau, \Omega \phi)^{\frac{1}{2}}+ \mathcal{E}^{in}_{con}(\tau, \phi)^{\frac{1}{2}} +\mathcal{E}^{in}_{con}(\tau, \Omega \phi)^{\frac{1}{2}}&\le C_4 \epsilon^{\frac{1}{2}}, \\
    		\mathcal{E}^{in}(\tau, \partial \phi)^{\frac{1}{2}}+\sum_{a}\mathcal{E}^{in}(\tau, L_a \phi)^{\frac{1}{2}} + \sum_{a}\mathcal{E}^{in}_{con}(\tau, L_a \phi)^{\frac{1}{2}} &\le C_4.
    	\end{aligned}
    \end{equation}
    In the above, $\epsilon$ is the size of the initial data, and we define
   \begin{equation*}
   	\tau^{*}=\sup\{\tau\in(\tau_0,\infty): \eqref{est:Boot_in} \ holds \ on\ [\tau_0,\tau)\}.
   \end{equation*}
We first write down some immediate $L^2$ estimates based on our bootstrap assumptions. 
	\begin{proposition}[Preliminary interior energy estimates]\label{prop:L2estimates_in}
    Under the assumptions \eqref{est:Boot_in}, for $\tau\in[\tau_0,\tau^*)$, the following bounds hold
		\begin{align}
			\sum_{|J|\leq1}\Big\|\frac{\tau}{t+1} \partial\Omega^J \phi\Big\|_{L^2({\mathcal{H}_\tau})}
			&\lesssim C_4 \epsilon^{\frac{1}{2}},\label{est:L2_in1}\\
			\Big\|\frac{\tau}{t+1} \partial \partial \phi\Big\|_{L^2({\mathcal{H}_\tau})}
			+ \sum_{a=1}^{3}\Big\|\frac{\tau}{t+1} \partial L_a \phi\Big\|_{L^2({\mathcal{H}_\tau})}
			&\lesssim C_4, \label{est:L2_in2}\\
			\sum_{|J|\leq1}\sum_{a=1}^3\Big\|\frac{\tau}{t+1} L_a\Omega^J \phi\Big\|_{L^2({\mathcal{H}_\tau})}
			+ \sum_{|J|\leq2}\Big\|\frac{\tau}{t+1} \Omega^J \phi\Big\|_{L^2({\mathcal{H}_\tau})}
			&\lesssim C_4 \epsilon^{\frac{1}{2}},\label{est:L2_in3}\\
			\sum_{a=1}^3\Big\|\frac{\tau}{t+1} L_a \partial \phi\Big\|_{L^2({\mathcal{H}_\tau})}
			+ \sum_{a,b}\Big\|\frac{\tau}{t+1} L_a L_b \phi\Big\|_{L^2({\mathcal{H}_\tau})}
			&\lesssim C_4.\label{est:L2_in4}
		\end{align}
	\end{proposition}
        \begin{proof}
		Estimates \eqref{est:L2_in1} and \eqref{est:L2_in2} follow directly from the interior bootstrap assumptions~\eqref{est:Boot_in} and the definition of the natural energy on hyperboloids. Then, the first term in~\eqref{est:L2_in4} follows from Lemma~\ref{lem:commutator2}, and \eqref{est:L2_in1}--\eqref{est:L2_in2}.

        Next, by applying the Hardy inequality \eqref{est:hardy_in} and noting that $r\leq t+1$ on $\mathcal{H}_\tau$, we have
        \begin{equation*}
            \begin{aligned}
                \Big\|\frac{\tau}{t+1} \phi\Big\|_{L^2({\mathcal{H}_\tau})}
                &\lesssim\Big\|\frac{\tau}{r}  \phi\Big\|_{L^2({\mathcal{H}_\tau})}
                \\& \lesssim\tau r_*^{\frac{1}{2}}\|\phi(\sqrt{\tau^2+r_*^2}-1,r_*\omega)\|_{L^2(\mathbb{S}^2)}+\tau (r^*)^{\frac{1}{2}}\|\phi(r^*+4\epsilon,r^*\omega)\|_{L^2(\mathbb{S}^2)}+\|\tau\slashed{\partial}_r\phi\|_{L^2(\mathcal{H}_\tau)},
            \end{aligned}
        \end{equation*}
        where $r_*, r^*$ are defined in \eqref{eq:def_rstar} and $\omega \in \mathbb{S}^2$. 
	Using \eqref{est:Boot_in}, Propositions~\ref{prop:estimates_local1} and~\ref{prop:estimates_ex1}, we get
    \begin{equation*}
        \Big\|\frac{\tau}{t+1} \phi\Big\|_{L^2({\mathcal{H}_\tau})}\lesssim C_2\epsilon^{\frac{1}{2}}+C_3\epsilon^{\frac{1}{2}}+C_4\epsilon^{\frac{1}{2}}\lesssim C_4\epsilon^{\frac{1}{2}}.
    \end{equation*}
        Using the identities
        \begin{align*}
		      \frac{\tau}{t+1}L_a=\frac{t}{t+1}\tau\slashed{\partial}_a+\frac{x_a}{t+1}\slashed{\partial}_\tau,\qquad
		      \frac{\tau}{t+1}\Omega_{ab}=\frac{x_a}{t+1}\tau\slashed{\partial}_b-\frac{x_b}{t+1}\tau\slashed{\partial}_a,
	    \end{align*}
        together with the energy bounds in~\eqref{est:Boot_in}, we obtain~\eqref{est:L2_in3} and the second term in~\eqref{est:L2_in4}.
    
	\end{proof}

	\begin{lemma}\label{lem:pointwise_in}
	Under the bootstrap assumptions~\eqref{est:Boot_in}, for any $(t,x)\in\Hcali$, we have
			\begin{equation}\label{est:pointwise_in1}
					|\phi(t,x)| \lesssim C_4\epsilon^{\frac{1}{8}}\tau^{-1}t^{-\frac{1}{2}}.
			\end{equation}
	On the other hand, for  any $(t,x)\in\Hcalf$,
	    \begin{equation}\label{est:Lq_in2}
		\|\phi(t,x)\|_{L^q(\Hcalf)}
	    		\lesssim  C_4 \epsilon^{\frac{1}{2}} \tau^{-1}r^{-\frac{1}{2}+\frac{3}{q}},\qquad q>6.
	    \end{equation} 
	\end{lemma}
	\begin{proof}
	$\bullet$ We first prove \eqref{est:pointwise_in1}. Let $(t,x)\in\mathcal{H}_\tau$ and $r\leq t/4$. Applying estimate~\eqref{est:KS_in1} in Lemma~\ref{lem:KS_in}, we get
		\begin{equation*}
			\begin{aligned}
				|\phi(t,x)| &\lesssim \tau^{-1} t^{-\frac{1}{2}}
				\sum_{|I|\leq2}\epsilon^{(|I|-\frac{3}{2})\lambda}\Big\|\frac{\tau}{t+1}Z^I\phi\Big\|_{L^2(\mathcal{H}_\tau)}
                \lesssim \tau^{-1} t^{-\frac{1}{2}} 
				\bigl(C_4 \epsilon^{\frac{1}{2}-\frac{3}{2}\lambda}+ C_4 \epsilon^{\frac{1}{2}-\frac{1}{2}\lambda} + C_4 \epsilon^{\frac{1}{2}\lambda}\bigr),
			\end{aligned}
		\end{equation*}
        where we used the estimates in Proposition~\ref{prop:L2estimates_in}.
		Taking $\lambda=\frac{1}{4}$ in the above, we have
		\begin{equation*}
			|\phi(t,x)|\lesssim C_4\epsilon^{\frac{1}{8}}\tau^{-1}t^{-\frac{1}{2}}.
		\end{equation*}
		
	$\bullet$ Next we prove \eqref{est:Lq_in2}. Let $(t,x)\in\mathcal{H}_\tau$ and $r\geq t/8$. Then $r\leq r^*$. Using estimate~\eqref{est:KS_in2} we derive
		\begin{equation*}
			\begin{aligned}
				\|\phi(t,x)&\|_{L^q(\mathcal{H}_\tau \cap \{r \geq t/8\})}
                \\
				&\lesssim (r^*)^{\frac{3}{q}}\sum_{|J|\leq1}\|\Omega^{J} \phi(t^*, r^*\omega)\|_{L^2(\mathbb{S}^2)}
				+ \tau^{-1}r^{-\frac{1}{2}+\frac{3}{q}} 
				\sum_{|J|\leq1,\ |I|=1}\left\|\frac{\tau}{t}\Omega^{J} \phi\right\|_{L^2(\mathcal{H}_\tau)}^{\frac{1}{2}}
				\left\|\frac{\tau}{t}Z^I\Omega^{J} \phi\right\|_{L^2(\mathcal{H}_\tau)}^{\frac{1}{2}}\\
                &\lesssim C_3\epsilon^{\frac{1}{2}} \tau^{-1}(r^*)^{-\frac{1}{2}+\frac{3}{q}} + C_4 \epsilon^{\frac{1}{2}} \tau^{-1}r^{-\frac{1}{2}+\frac{3}{q}}\lesssim C_4\epsilon^{\frac{1}{2}}\tau^{-1}r^{-\frac{1}{2}+\frac{3}{q}},
			\end{aligned}
		\end{equation*}
        where we used Proposition~\ref{prop:L2estimates_in} and \eqref{est:L2_ex3}. 
	\end{proof}

   \subsection{Combining the boundary flux and interior estimates}
   In this subsection, we close the interior bootstrap assumptions \eqref{est:Boot_in} on $\phi$ up to the maximal hyperboloidal time $\tau^*$. The argument relies on the hyperboloidal energy estimates in the interior region, whose right-hand side contains three types of contributions: initial terms on $\Sigma_{t_1}\cap\mathcal{D}^{in}$, boundary flux terms on $\mathcal{B}_{4\epsilon}$ and spacetime integrals of the nonlinear source. Accordingly, we first derive uniform $L^2$ bounds for the boundary terms on $\mathcal{B}_{4\epsilon}$, and then establish $L^2$ bounds on $\mathcal{H}_\tau$ for the nonlinear difference
   \begin{equation*}
       \mathcal{N}(\psi,\phi)=|\psi+\phi|^{p-1}(\psi+\phi)-|\psi|^{p-1}\psi.
   \end{equation*}
   Combining these with the interior energy estimates yields improved bounds, which closes the bootstrap for $\tau\in[\tau_0,\tau^*)$ once $\epsilon$ is sufficiently small.

   \medskip

   We start with the boundary terms. Since the interior energy identities pick up flux contributions  involving good derivatives, rotations, and the modified conformal combinations, we record the following uniform bounds.

   \begin{lemma}[Boundary estimates]\label{lem:boundary}
   	The following boundary $L^2$ estimates: for good derivatives we have:
   		\begin{equation}\label{est:boundary1}
   			\begin{aligned}
   				\sum_{|J|\leq1}&\left(\int_{\mathcal{B}_{4\epsilon}} |G \Omega^J\phi|^2 dx \right)^{1/2}
   				+\left( \int_{\mathcal{B}_{4\epsilon}} |G \partial\phi|^2 dx \right)^{1/2}+\sum_a\left( \int_{\mathcal{B}_{4\epsilon}} |GL_a \phi|^2 dx \right)^{1/2}\lesssim C_3\epsilon^{1/2}.
   			\end{aligned}
   		\end{equation}
   	    For rotations, we have
   	    \begin{equation}\label{est:boundary2}
   	    	\begin{aligned}
   	    		\sum_{1\leq|J|\leq2}\left( \int_{\mathcal{B}_{4\epsilon}} \Big|\frac{1}{r} \Omega^J\phi\Big|^2 dx \right)^{1/2}+ \sum_{a}\left( \int_{\mathcal{B}_{4\epsilon}} \Big|\frac{1}{r} \Omega L_a \phi\Big|^2 dx \right)^{1/2}\lesssim C_3\epsilon^{1/2}.
   	    	\end{aligned}
   	    \end{equation}
        The following modified conformal boundary terms obey:
        \begin{equation}\label{est:boundary3}
        	\begin{aligned}
        		\sum_{|J|\leq1}\left( \int_{\mathcal{B}_{4\epsilon}} [(r+t+1)L \Omega^J \phi + 2\Omega^J\phi]^2 dx \right)^{1/2} +\sum_a\left( \int_{\mathcal{B}_{4\epsilon}} [(r+t+1)L L_a \phi + 2L_a \phi]^2 dx \right)^{1/2} \lesssim C_3\epsilon^{1/2}.
        	\end{aligned}
        \end{equation}
   \end{lemma}

   \begin{proof}
   Since we work on the boundary $\mathcal B_{4\epsilon}$ where $r=t-4\epsilon\ge t_1-4\epsilon=1-2\epsilon$, we note that $t=r+4\epsilon\sim r$.
   
$\bullet$ We first prove \eqref{est:boundary1} and \eqref{est:boundary2}. We divide the proof into two steps.
\textit{Step 1:} The first step is to derive pointwise estimates on $\mathcal B_{4\epsilon}$. Note the identities
\begin{equation}\label{eq:Gidentity}
    \begin{aligned}
        G_a
&=\frac{\omega_a}{t+r}L_0+\frac{1}{t}L_a-\frac{\omega_ax^b}{t(t+r)}L_b
=\omega_aL-\frac{\omega^b}{r}\Omega_{ab},  \quad
\partial_t=\frac{1}{2}(L+\underline{L}),\\
\partial_a&=\frac{\omega_a}{2}(L-\underline{L})-\frac{\omega^b}{r}\Omega_{ab}, \quad
L_a=\frac{\omega_a}{2}((t+r)L-(t-r)\underline{L})-\frac{t}{r}\omega^b\Omega_{ab}.
    \end{aligned}
\end{equation}
Hence, for $|J|\le 1$, using $t\sim r$ on $\mathcal B_{4\epsilon}$ and Proposition~\ref{prop:decay_ex}, we obtain
\begin{equation}\label{eq:G_OmegaJ_pw}
|G\Omega^J\phi|
\lesssim \frac1t\Big(|L_0\Omega^J\phi|+\sum_a|L_a\Omega^J\phi|\Big)
\lesssim C_3\epsilon^{\frac12}r^{-2}.
\end{equation}
Next, using the standard decomposition and estimates in Proposition~\ref{prop:decay_ex} and Lemma~\ref{lem:refined_ex}, we infer that
\begin{equation*}
           \begin{aligned}
               |G\partial\phi|&\lesssim|L^2\phi|+|L\underline{L}\phi|+\frac{1}{r}|L\Omega\phi|+\frac{1}{r}|L\phi|
               +\frac{1}{r}\sum_{1\leq|J|\leq2}|\Omega^{J}\phi|+\frac{1}{r}|\underline{L}\Omega\phi|+\frac{1}{r}|\underline{L}\phi|\lesssim C_3\epsilon^{\frac{1}{2}}r^{-2},
           \end{aligned}
       \end{equation*}
and similarly   
\begin{equation*}
    \begin{aligned}
        |GL_a\phi|&\lesssim (t+r)|L^2\phi|+|L\phi|+(t-r)|L\underline L\phi|+\frac{t}{r}|L\Omega\phi|
        +\frac{t-r}{r^2}|\Omega\phi|+\frac1r|L_a\Omega\phi|+\frac1r|L_a\phi|
        \lesssim C_3\epsilon^{\frac12}r^{-2}.
    \end{aligned}
\end{equation*}
Moreover, for $1\le |J|\le 2$, it follows from \eqref{est:L2_ex3} that
\begin{equation*}
\Big|\frac1r\Omega^J\phi\Big|\lesssim C_3\epsilon^{\frac12}r^{-2}.
\end{equation*}
Finally, we use \eqref{est:commutators} and Proposition~\ref{prop:decay_ex} to conclude
\begin{equation}\label{eq:rOmegaLa_pw}
\Big|\frac1r\Omega L_a\phi\Big|
\lesssim \frac1t\big(|L_b\Omega\phi|+|L_b\phi|\big)
\lesssim C_3\epsilon^{\frac12}r^{-2}.
\end{equation}

\textit{Step 2:} Our second step is to consider $L^2$-integrals over $\mathcal B_{4\epsilon}$.
Using $dx\sim r^2\,dr\,d\omega$ on $\mathcal B_{4\epsilon}$ and the bound
$|Q|\lesssim C_3\epsilon^{\frac12}r^{-2}$ (where $Q$ stands for any of the quantities in
\eqref{eq:G_OmegaJ_pw}--\eqref{eq:rOmegaLa_pw}), we compute
\begin{equation}\label{eq:int_template}
\begin{aligned}
\Big(\int_{\mathcal B_{4\epsilon}} |Q|^2\,dx\Big)^{\frac12}
&\lesssim \Big(\int_{r\ge 1-2\epsilon} (C_3^2\epsilon\,r^{-4})\,r^2\,dr\Big)^{\frac12}
\lesssim C_3\epsilon^{\frac12}.
\end{aligned}
\end{equation}
Applying \eqref{eq:int_template} to $Q \in \{ G\Omega^{\leq 1}\phi, G\partial\phi, GL_a\phi\}$, we obtain \eqref{est:boundary1}. Applying \eqref{eq:int_template} to
$Q \in \{\frac1r\Omega\Omega^{\leq 1}\phi, \frac1r\Omega L_a\phi\}$ gives \eqref{est:boundary2}.

$\bullet$ We finally turn to \eqref{est:boundary3}.  First, using~\eqref{eq:Gidentity}, we have
\begin{equation*}
    \begin{aligned}
        |LL_a\phi|&\lesssim(t+r)|L^2\phi|+|L\phi|+(t-r)|L\underline L\phi|+\frac{t}{r}|L\Omega\phi|
        +\frac{t-r}{r^2}|\Omega\phi|
        \lesssim C_3\epsilon^{\frac12}r^{-2}.
    \end{aligned}
\end{equation*}
Proposition~\ref{prop:decay_ex} gives us 
\begin{equation*}
    |L\Omega^J\phi|\lesssim C_3\epsilon^{\frac{1}{2}}r^{-2},\qquad|J|\leq1.
\end{equation*}
Applying \eqref{eq:int_template} to $Q \in \{LL_a\phi, L\Omega^{\leq 1}\phi\}$ we obtain
\begin{equation}\label{est:LL_a}
    \sum_a\Big(\int_{\mathcal B_{4\epsilon}}|LL_a\phi|^2\,dx\Big)^{\frac{1}{2}}+\sum_{|J|\leq1}\Big(\int_{\mathcal B_{4\epsilon}}|L\Omega^J\phi|^2\,dx\Big)^{\frac{1}{2}}\lesssim C_3\epsilon^{\frac{1}{2}}.
\end{equation}
Hence, for $|J|\le 1$, applying~\eqref{est:LL_a} and estimates in \eqref{eq:E-con-Lpphi} with $\tilde{u}=4\epsilon\in[2\epsilon,5\epsilon]$, we get
\begin{equation*}
\begin{aligned}
\Big(\int_{\mathcal B_{4\epsilon}}[(r+t+1)L\Omega^J\phi+2\Omega^J\phi]^2\,dx\Big)^{\frac12}
&\lesssim \Big(\int_{\mathcal B_{4\epsilon}}[(r+t)L\Omega^J\phi+2\Omega^J\phi]^2\,dx\Big)^{\frac12}
+\Big(\int_{\mathcal B_{4\epsilon}}|L\Omega^J\phi|^2\,dx\Big)^{\frac12}\\
&\lesssim C_3\epsilon+C_3\epsilon^{\frac12}\Big(\int_{1-2\epsilon}^{\infty} r^{-2}\,dr\Big)^{\frac12}
\lesssim C_3\epsilon^{\frac12}.
\end{aligned}
\end{equation*}

Similarly, for each $a$ we apply the same splitting:
\begin{equation*}
\begin{aligned}
\Big(\int_{\mathcal B_{4\epsilon}}[(r+t+1)LL_a\phi+2L_a\phi]^2\,dx\Big)^{\frac12}
&\lesssim \Big(\int_{\mathcal B_{4\epsilon}}[(r+t)LL_a\phi+2L_a\phi]^2\,dx\Big)^{\frac12}
+\Big(\int_{\mathcal B_{4\epsilon}}|LL_a\phi|^2\,dx\Big)^{\frac12}
\\&\lesssim C_3\epsilon^{\frac{1}{2}}.
\end{aligned}
\end{equation*}

\end{proof}

   Next, we estimate the nonlinear source term. The goal is to control $\mathcal{N}(\psi,\phi)$ and its derivatives in $L^2(\mathcal{H}_\tau)$, both with and without the hyperboloidal weight $\tau$, in a form that is integrable in $\tau$.
  
  \begin{lemma}\label{lem:nonlinear_in}
	Under the bootstrap assumptions~\eqref{est:Boot_in},  the following $L^2$ estimates  hold
          \begin{equation}\label{eq:int_L2-nonlinear}\begin{split}
              \|\partial^{\leq 1} \mathcal{N}(\psi,\phi)\|_{L^2(\mathcal{H}_{\tau})}
              + \|\Omega\mathcal{N}(\psi,\phi)\|_{L^2(\mathcal{H}_{\tau})}
              + \|L_a\mathcal{N}(\psi,\phi)\|_{L^2(\mathcal{H}_{\tau})}
              \lesssim C_4^p\epsilon^{\frac{1}{2}+(\frac{1}{2}-\delta)\frac{p-5}{2}}\tau^{-\frac{p+3}{4}+\frac{p+3}{2}\delta},
          \end{split}\end{equation}
and the following weighted $L^2$ estimates  hold
          \begin{equation}\label{eq:int_weightedL2-nonlinear}\begin{split}
              \|\tau\mathcal{N}(\psi,\phi)\|_{L^2(\mathcal{H}_{\tau})}+ \|\tau\Omega\mathcal{N}(\psi,\phi)\|_{L^2(\mathcal{H}_{\tau})}
              + \|\tau L_a\mathcal{N}(\psi,\phi)\|_{L^2(\mathcal{H}_{\tau})}
              \lesssim C_4^p\epsilon^{\frac{1}{2}+(\frac{1}{2}-\delta)\frac{p-5}{2}}\tau^{-\frac{p-1}{4}+\frac{p+3}{2}\delta}.
          \end{split}\end{equation}
  \end{lemma}
  
  \begin{proof}
  Recall that for any $\tilde{\Gamma}\in\{\Omega,\partial,L_a\}$
  \begin{align}\label{eq:N-alg}
      |\mathcal N(\psi,\phi)|&\lesssim |\phi|^p+|\phi||\psi|^{p-1},\\
      \label{eq:N-der}
      |\tilde{\Gamma}\mathcal N(\psi,\phi)|
      &\lesssim |\phi|^{p-1}|\tilde{\Gamma}\phi|+|\phi|^{p-1}|\tilde{\Gamma}\psi|+|\psi|^{p-1}|\tilde{\Gamma}\phi|
      +|\phi||\psi|^{p-2}|\tilde{\Gamma}\psi|.
  \end{align}
  We now decompose the $L^2$ integrals into two terms coming from an interior region $\Hcali =H_\tau\cap\{r\le t/4\}$ and a `far' or exterior region $\Hcalf = \mathcal H_\tau\cap\{r\ge t/8\}$. 

  \medskip
  $\bullet$ We first prove \eqref{eq:int_L2-nonlinear}. 
  Using~\eqref{eq:N-alg}, we obtain
  \begin{equation}\label{eq:N-L2}
      \|\mathcal N(\psi,\phi)\|_{L^2(\mathcal H_\tau)}
      \lesssim \||\phi|^p\|_{L^2(\mathcal H_\tau)}+\|\phi|\psi|^{p-1}\|_{L^2(\mathcal H_\tau)}
      =: \mathcal A_1+\mathcal A_2.
  \end{equation}

  \smallskip
  \noindent\emph{Interior contribution for $\mathcal A_1$.}
  On $\{r\le t/4\}$ we use the weighted $L^2$ control of $\phi$ together with the $L^\infty$ decay of $\phi$,
  obtaining
  \begin{equation}\label{eq:A11-mid}
      \begin{aligned}
          \||\phi|^p\|_{L^2(\Hcali)}
          &\lesssim
          \Big\|\frac{\tau}{t+1}\phi\Big\|_{L^2(\mathcal{H}_\tau)}
          \Big\|\frac{t+1}{\tau}|\phi|^{p-1}\Big\|_{L^\infty(\Hcali)}
          \lesssim C_4^p\epsilon^{\frac12+\frac{p-1}{8}}\tau^{-\frac32(p-1)}.
      \end{aligned}
    \end{equation}
    Here we used \eqref{est:L2_in3} and \eqref{est:pointwise_in1} for the $L^\infty$ bound.

    \smallskip
    \noindent\emph{Exterior contribution for $\mathcal A_1$.}
    On $\{r\ge t/8\}$ we apply Hölder inequality to obtain
    \begin{equation*}
        \begin{aligned}
            \||\phi|^p\|_{L^2(\Hcalf)} ^2 =\||\phi|^{2p}\|_{L^1(\Hcalf)}
            &\lesssim
            \Big\|\frac{\tau}{t+1}\phi\Big\|_{L^2(\mathcal{H}_\tau)}
            \Big\|\frac{t+1}{\tau}|\phi|^{2p-1}\Big\|_{L^2(\Hcalf)}.
        \end{aligned}
    \end{equation*}
    Since $\tau\leq t+1\leq\tau^2$ on $\mathcal H_\tau$, the second factor can be bounded by
    \[
      \Big\|\frac{t+1}{\tau}|\phi|^{2p-1}\Big\|_{L^2(\Hcalf)}
      \lesssim \tau \|\phi\|_{L^{4p-2}(\Hcalf)}^{2p-1}.
    \]
    We then apply the $L^q$ estimate~\eqref{est:Lq_in2} from Lemma~\ref{lem:pointwise_in} to obtain
    \begin{equation*}
        \|\phi\|_{L^{4p-2}(\Hcalf)}\lesssim C_4\epsilon^{\frac12}\tau^{-\frac32+\frac{3}{4p-2}}\,.
    \end{equation*}
    All together, this implies
    \begin{equation}\label{eq:A12-mid-fin}
        \||\phi|^p\|_{L^2(\Hcalf)}\lesssim (C_4\epsilon^{\frac12})^{p}\tau^{-\frac32p+2}.
    \end{equation}

    \smallskip
    \noindent\emph{Estimate of $\mathcal A_2$.}
    From~\eqref{eq:disp-psi01}, we derive
    \begin{equation}\label{est:psi_H}
        \begin{aligned}
            |\psi|\lesssim C_1\epsilon^{\frac{1}{2}-\delta},\qquad
            |\psi|\lesssim C_1\langle t+r\rangle^{-\frac{1}{2}+\delta}.
        \end{aligned}
    \end{equation}
    Using the Hölder inequality again, we find
    \begin{equation*}
        \mathcal A_2
        \lesssim
        \Big\|\frac{\tau}{t+1}\phi\Big\|_{L^2(\mathcal H_\tau)}
        \Big\|\frac{t+1}{\tau}|\psi|^{p-1}\Big\|_{L^\infty(\mathcal H_\tau)}.
    \end{equation*}
    Since $\tau\leq t+1\leq\tau^2$ on $\mathcal H_\tau$, using~\eqref{est:psi_H}, the second factor can be bounded by
    \begin{equation*}
        \Big\|\frac{t+1}{\tau}|\psi|^{p-1}\Big\|_{L^\infty(\mathcal H_\tau)}\lesssim C_1^{p-1}\epsilon^{(\frac{1}{2}-\delta)\frac{p-5}{2}}\tau^{-\frac{p+3}{4}+\frac{p+3}{2}\delta},
    \end{equation*}
    which implies
    \begin{equation}\label{eq:A2-mid}
        \mathcal{A}_2\lesssim C_4^p\epsilon^{\frac{1}{2}+(\frac{1}{2}-\delta)\frac{p-5}{2}}\tau^{-\frac{p+3}{4}+\frac{p+3}{2}\delta}.
    \end{equation}
    Combining \eqref{eq:N-L2}, \eqref{eq:A11-mid}, \eqref{eq:A12-mid-fin} and \eqref{eq:A2-mid} gives the first bound in \eqref{eq:int_L2-nonlinear}.

    Next, we illustrate the estimate for $\tilde{\Gamma}=\Omega$; the cases $\tilde{\Gamma}=\partial$ and $\tilde{\Gamma}=L_a$ are similar.
    From \eqref{eq:N-der}, we decompose
    \begin{equation*}
        \begin{aligned}
            \|\Omega\mathcal N(\psi,\phi)\|_{L^2(\mathcal{H}_\tau)}
            &\lesssim \||\phi|^{p-1}\Omega\phi\|_{L^2(\mathcal{H}_\tau)}+\||\phi|^{p-1}\Omega\psi\|_{L^2(\mathcal{H}_\tau)}
            +\||\psi|^{p-1}\Omega\phi\|_{L^2(\mathcal{H}_\tau)}+\|\phi|\psi|^{p-2}\Omega\psi\|_{L^2(\mathcal{H}_\tau)}.
        \end{aligned}
    \end{equation*}
    The last two terms are estimated in the same way as $\mathcal A_2$ using $L^\infty$ decay for $\psi$ or $\Omega\psi$ and $L^2$ estimates for $\phi$ or $\Omega\phi$.
    For the first term $\||\phi|^{p-1}\Omega\phi\|_{L^2}$, we again split it into two parts. On $\{r\le t/4\}$, combining \eqref{est:L2_in3} and \eqref{est:pointwise_in1}, we deduce
    \begin{equation*}
        \begin{aligned}
            \||\phi|^{p-1}\Omega\phi\|_{L^2(\Hcali)}
            \lesssim
            \Big\|\frac{\tau}{t+1}\Omega\phi\Big\|_{L^2(\mathcal{H}_\tau)}
            \Big\|\frac{t+1}{\tau}|\phi|^{p-1}\Big\|_{L^\infty(\Hcali)}
            \lesssim C_4^p\epsilon^{\frac12+\frac{p-1}{8}}\tau^{-\frac32(p-1)}.
        \end{aligned}
    \end{equation*}
    On $\{r\ge t/8\}$, the Hölder inequality implies
    \begin{equation*}
        \begin{aligned}
            \||\phi|^{p-1}\Omega\phi\|_{L^2(\Hcalf)}
            &=\||\phi|^{2p-2}|\Omega\phi|^2\|^{\frac{1}{2}}_{L^1(\Hcalf)}
            \lesssim\Big\||\frac{\tau}{t+1}\Omega\phi|^2\Big\|_{L^2(\Hcalf)}^{\frac{1}{2}}\left\|\frac{t^2}{\tau^2}|\phi|^{2p-2}\right\|_{L^2(\Hcalf)}^{\frac{1}{2}}.
        \end{aligned}
    \end{equation*}
    Since $\tau\leq t+1\leq\tau^2$ on $\mathcal{H}_\tau$, the second factor can be bounded by
    \begin{equation}\label{eq:GN_mid1}
        \left\|\frac{t^2}{\tau^2}|\phi|^{2p-2}\right\|_{L^2(\Hcalf)}\lesssim\tau^2\|\phi\|_{L^{4p-4}(\Hcalf)}^{2p-2}\lesssim C_4^{2p-2}\epsilon^{p-1}\tau^{-3p+\frac{11}{2}},
    \end{equation}
    where we used~\eqref{est:Lq_in2} in the last inequality.
    For the first factor, applying Lemmas~\ref{lem:Sobo_embedding} and \ref{lem:inter_inequ}, we derive
    \begin{equation}\label{eq:GN-mid2}
        \begin{aligned}
            \Big\||\frac{\tau}{t+1}\Omega\phi|^2\Big\|_{L^2(\Hcalf)}^{\frac{1}{2}}=\Big\|\frac{\tau}{t+1}\Omega\phi\Big\|_{L^4(\Hcalf)}
            \lesssim\Big\|\frac{\tau}{t+1}\Omega\phi\Big\|_{L^2(\Hcalf)}^{\frac{1}{4}}\Big\|\frac{\tau}{t+1}\Omega\phi\Big\|_{H^1(\Hcalf)}^{\frac{3}{4}}.
        \end{aligned}
    \end{equation}
    Combining~\eqref{eq:GN_mid1} and~\eqref{eq:GN-mid2} and using Proposition~\ref{prop:L2estimates_in}, we obtain
    \begin{equation*}
        \||\phi|^{p-1}\Omega\phi\|_{L^2(\Hcalf)}\lesssim C_4^p\epsilon^{\frac{p}{2}-\frac{3}{8}}\tau^{-\frac{3}{2}p+\frac{11}{4}}.
    \end{equation*}
    For the second term $\||\phi|^{p-1}\Omega\psi\|_{L^2}$, we bound it by
    \begin{equation*}
        \||\phi|^{p-1}\Omega\psi\|_{L^2(\mathcal{H}_\tau)}\lesssim\||\phi|^{p-1}\Omega\psi\|_{L^2(\Hcali)}+\||\phi|^{p-1}\Omega\psi\|_{L^2(\Hcalf)}.
    \end{equation*}
    On $\{r\leq t/4\}$, using~\eqref{est:L2_in3},~\eqref{est:pointwise_in1}, and~\eqref{eq:disp-psi01}, we obtain
    \begin{equation*}
        \begin{aligned}
            \||\phi|^{p-1}\Omega\psi\|_{L^2(\Hcali)}&\lesssim\left\|\frac{\tau}{t+1}\phi\right\|_{L^2(\mathcal{H}_{\tau})}\left\|\frac{t+1}{\tau}|\phi|^{p-2}\Omega\psi\right\|_{L^\infty(\Hcali)}
            \lesssim C_4^p\epsilon^{1-\delta+\frac{p-2}{8}}\tau^{-\frac{3}{2}p+3}.
        \end{aligned}
    \end{equation*}
    On $\{r\geq t/8\}$, the H\"{o}lder inequality yields
    \begin{equation*}
        \begin{aligned}
            \||\phi|^{p-1}\Omega\psi\|_{L^2(\Hcalf)}
            \lesssim\left\||\phi|^{p-1}\right\|_{L^2(\Hcalf)}\|\Omega\psi\|_{L^\infty(\mathcal{H}_{\tau})}.
        \end{aligned}
    \end{equation*}
    From \eqref{est:Lq_in2}, we derive
    \begin{equation*}
        \left\||\phi|^{p-1}\right\|_{L^2(\Hcalf)}=\|\phi\|^{p-1}_{L^{2p-2}(\Hcalf)}\lesssim C_4^{p-1}\epsilon^{\frac{p-1}{2}}\tau^{-\frac{3}{2}p+3}.
    \end{equation*}
    Thus, combining it with~\eqref{eq:disp-psi01}, we get
    \begin{equation*}
        \||\phi|^{p-1}\Omega\psi\|_{L^2(\Hcalf)}\lesssim C_4^p\epsilon^{\frac{p}{2}-\delta}\tau^{-\frac{3}{2}p+3}.
    \end{equation*}
 Therefore, we conclude that
    \begin{equation*}
        \|\Omega\mathcal N(\psi,\phi)\|_{L^2(\mathcal H_\tau)}
        \lesssim C_4^p\epsilon^{\frac{1}{2}+(\frac{1}{2}-\delta)\frac{p-5}{2}}\tau^{-\frac{p+3}{4}+\frac{p+3}{2}\delta}.
    \end{equation*}

    The bounds for $\partial\mathcal N$ and $L_a\mathcal N$ follow in the same way by replacing $\Omega$ with $\partial$ or $L_a$. This completes the proof of \eqref{eq:int_L2-nonlinear}.

    $\bullet$ We now turn to \eqref{eq:int_weightedL2-nonlinear}.
    Multiplying \eqref{eq:N-alg} and \eqref{eq:N-der} by $\tau$, we repeat the above argument.
    For the first bound in \eqref{eq:int_weightedL2-nonlinear}, using~\eqref{eq:N-alg}, we get
    \begin{equation*}
        \|\tau\mathcal N(\psi,\phi)\|_{L^2(\mathcal H_\tau)}
      \lesssim \|\tau|\phi|^p\|_{L^2(\mathcal H_\tau)}+\|\tau\phi|\psi|^{p-1}\|_{L^2(\mathcal H_\tau)}
      =: \mathcal A_3+\mathcal A_4.
    \end{equation*}
    \emph{Estimate of $\mathcal{A}_3$.} Splitting $\mathcal{A}_3$ into two parts
    \begin{equation*}
        \mathcal{A}_3\lesssim\|\tau|\phi|^p\|_{L^2(\Hcalf)}+\|\tau|\phi|^p\|_{L^2(\Hcali)}.
    \end{equation*}
    On $\{r\leq t/4\}$, applying~\eqref{est:L2_in3} and~\eqref{est:pointwise_in1}, we get
    \begin{equation}\label{est:A2-mid1}
        \begin{aligned}
            \|\tau|\phi|^p\|_{L^2(\Hcali)}&\lesssim\left\|\frac{\tau}{t+1} \phi\right\|_{L^2(\Hcali)}
   			\left\|(t+1)|\phi|^{p-1} \right\|_{L^\infty(\Hcali)}
            \lesssim C_4^p\epsilon^{\frac{1}{2}+\frac{p-1}{8}}\tau^{-\frac{3}{2}p+\frac{5}{2}}.
        \end{aligned}
    \end{equation}
    On $\{r\geq t/8\}$, using the H\"{o}lder inequality, we have
    \begin{equation*}
        \begin{aligned}
            &\|\tau|\phi|^p\|_{L^2(\Hcalf)}=\left\|\tau^2
            |\phi|^{2p}\right\|_{L^1(\Hcalf)}^{\frac{1}{2}}
            \lesssim\left\|\frac{\tau}{t+1} \phi\right\|_{L^2(\Hcalf))}^{\frac{1}{2}}
   			\left\|\tau(t+1) |\phi|^{2p-1} \right\|_{L^2(\Hcalf)}^{\frac{1}{2}}.
        \end{aligned}
    \end{equation*}
    For the second factor, combining with~\eqref{est:Lq_in2}, we see
    \begin{equation*}
        \left\|\tau(t+1) |\phi|^{2p-1} \right\|_{L^2(\Hcalf)}\lesssim\tau^3\|\phi\|^{2p-1}_{L^{4p-2}(\Hcalf)}\lesssim C_4^{2p-1}\epsilon^{p-\frac{1}{2}}\tau^{-3p+6},
    \end{equation*}
    where we used the fact that $\tau\leq t+1\leq\tau^2$ on $\mathcal{H}_\tau$.
    Together with~\eqref{est:L2_in3}, we conclude
    \begin{equation}\label{est:A2-mid2}
        \|\tau|\phi|^p\|_{L^2(\Hcalf)}\lesssim C_4^p\epsilon^{\frac{p}{2}}\tau^{-\frac{3}{2}p+3}.
    \end{equation}
    \emph{Estimate of $\mathcal{A}_4$.} Employing~\eqref{est:L2_in3} and~\eqref{eq:disp-psi01}, we have
    \begin{equation}\label{est:A2-mid3}
        \mathcal{A}_4\lesssim
   			\left\|\frac{\tau}{t+1} \phi\right\|_{L^2(\mathcal{H}_{\tau} )}
   			\left\|(t+1) |\psi|^{p-1} \right\|_{L^\infty(\mathcal{H}_{\tau} )}
   			\lesssim C_4^p\epsilon^{\frac{1}{2}+(\frac{1}{2}-\delta)\frac{p-5}{2}}\tau^{-\frac{p-1}{4}+\frac{p+3}{2}\delta}.
    \end{equation}
    Combining \eqref{est:A2-mid1}, \eqref{est:A2-mid2}, and \eqref{est:A2-mid3} gives the first bound in \eqref{eq:int_weightedL2-nonlinear}.

    Next, we illustrate the estimate for $\tilde{\Gamma}=\Omega$; the case $\tilde{\Gamma}=L_a$ is similar.
    From \eqref{eq:N-der}, we decompose
    \begin{equation}\label{est:tau_omega}
        \begin{aligned}
            \|\tau\Omega\mathcal N(\psi,\phi)\|_{L^2(\mathcal{H}_\tau)}
            &\lesssim \|\tau|\phi|^{p-1}\Omega\phi\|_{L^2(\mathcal{H}_\tau)}+\|\tau|\phi|^{p-1}\Omega\psi\|_{L^2(\mathcal{H}_\tau)}\\
            &+\|\tau|\psi|^{p-1}\Omega\phi\|_{L^2(\mathcal{H}_\tau)}+\|\tau\phi|\psi|^{p-2}\Omega\psi\|_{L^2(\mathcal{H}_\tau)}.
        \end{aligned}
    \end{equation}
    The last two terms are estimated in the same way as $\mathcal A_4$, so we omit the details.
    For the first term in~\eqref{est:tau_omega}, we split it into two parts
    \begin{equation*}
        \|\tau|\phi|^{p-1}\Omega\phi\|_{L^2(\mathcal{H}_{\tau})}\lesssim\|\tau|\phi|^{p-1}\Omega\phi\|_{L^2(\Hcali)}+\|\tau|\phi|^{p-1}\Omega\phi\|_{L^2(\Hcalf)}.
    \end{equation*}
    On $\{r\leq t/4\}$, using the H\"{o}lder inequality, we have
   	\begin{equation*}
   		\begin{aligned}
   			\|\tau|\phi|^{p-1}\Omega\phi\|_{L^2(\Hcali)}
            &\lesssim
   			\Big\|\frac{\tau}{t+1}\Omega\phi\Big\|_{L^2(\mathcal{H}_{\tau})}
   			\left\|(t+1)|\phi|^{p-1}\right\|_{L^\infty(\Hcali)}
            \lesssim
   			C_4^p \epsilon^{\frac{1}{2}+\frac{p-1}{8}}\tau^{-\frac{3}{2}p+\frac{5}{2}},
   		\end{aligned}
   	\end{equation*}
    where we used~\eqref{est:L2_in3} and~\eqref{est:pointwise_in1}.
    On $\{r\geq t/8\}$, using the H\"{o}lder inequality again, we deduce
    \begin{equation*}
        \begin{aligned}
           &\|\tau|\phi|^{p-1}\Omega\phi\|_{L^2(\Hcalf)}=\left\|\tau^2|\Omega\phi|^2|\phi|^{2p-2}\right\|_{L^1(\Hcalf)}^{\frac{1}{2}}
           \lesssim\left\||\frac{\tau}{t+1}\Omega\phi|^2\right\|_{L^2(\Hcalf)}^{\frac{1}{2}}\left\|t^2|\phi|^{2p-2}\right\|_{L^2(\Hcalf)}^{\frac{1}{2}}.
        \end{aligned}
    \end{equation*}
    For the first factor, applying Lemmas~\ref{lem:Sobo_embedding}-\ref{lem:inter_inequ} and Proposition~\ref{prop:L2estimates_in}, we obtain
    \begin{equation*}
        \begin{aligned}
            \left\||\frac{\tau}{t+1}\Omega\phi|^2\right\|_{L^2(\Hcalf)}^{\frac{1}{2}}&=\left\|\frac{\tau}{t+1}\Omega\phi\right\|_{L^4(\Hcalf)}
            \lesssim
   			\left\|\frac{\tau}{t+1}\Omega\phi\right\|_{L^2(\Hcalf)}^{1\over4}\left\|\frac{\tau}{t+1}\Omega\phi\right\|_{H^1(\Hcalf)}^{3\over4}\lesssim C_4\epsilon^{\frac{1}{2}}.
        \end{aligned}
    \end{equation*}

    Since $\tau\leq t+1\leq \tau^2$ on $\mathcal{H}_\tau$, we apply~\eqref{est:Lq_in2} to the second factor
    \begin{equation*}
        \left\|t^2|\phi|^{2p-2}\right\|_{L^2(\Hcalf)}^{\frac{1}{2}}\lesssim\tau^2\|\phi\|^{p-1}_{L^{4p-4}(\Hcalf)}\lesssim C_4^{p-1}\epsilon^{\frac{p-1}{2}}\tau^{-\frac{3}{2}p+\frac{17}{4}}.
    \end{equation*}
    Combining the above estimates, we conclude that
    \begin{equation*}
        \|\tau|\phi|^{p-1}\Omega\phi\|_{L^2(\Hcalf)}\lesssim C_4^p\epsilon^{\frac{p}{2}}\tau^{-\frac{3}{2}p+\frac{17}{4}}.
    \end{equation*}
   	For the second term in~\eqref{est:tau_omega}, we decompose it into two parts
    \begin{equation*}
        \begin{aligned}
        \|\tau|\phi|^{p-1}\Omega\psi\|_{L^2(\mathcal{H}_\tau)}&\lesssim\Big\|\frac{\tau}{t+1}\phi\Big\|_{L^2(\mathcal{H}_{\tau})}\left\|(t+1)|\phi|^{p-2}\Omega\psi\right\|_{L^\infty(\Hcali)}
        +\left\||\phi|^{p-1}\right\|_{L^2(\Hcalf)}\|\tau\Omega\psi\|_{L^\infty(\mathcal{H}_{\tau})}=:\mathcal{A}_5+\mathcal{A}_6.
        \end{aligned}
    \end{equation*}
    Employing~\eqref{est:L2_in3},~\eqref{est:pointwise_in1} and~\eqref{eq:disp-001}, we obtain
    \begin{equation*}
        \mathcal{A}_5\lesssim C_4^p\epsilon^{1-\delta+\frac{p-2}{8}}\tau^{-\frac{3}{2}p+4}.
    \end{equation*}
    Using~\eqref{est:Lq_in2}, we have
    \begin{equation*}
        \mathcal{A}_6\lesssim\|\phi\|^{p-1}_{L^{2p-2}(\Hcalf)}\|\tau\Omega\psi\|_{L^\infty(\mathcal{H}_{\tau})}\lesssim C_4^p\epsilon^{\frac{p}{2}-\delta}\tau^{-\frac{3}{2}p+4},
    \end{equation*}
    where we used~\eqref{eq:disp-001} in the last inequality.
    The above estimates yield
    \begin{equation*}
        \|\tau|\phi|^{p-1}\Omega\psi\|_{L^2(\mathcal{H}_\tau)}\lesssim C_4^p\epsilon^{1-\delta+\frac{p-2}{8}}\tau^{-\frac{3}{2}p+4}.
    \end{equation*}
    Therefore, we conclude that
    \begin{equation*}
        \|\tau\Omega\mathcal N(\psi,\phi)\|_{L^2(\mathcal H_\tau)}
        \lesssim C_4^p\epsilon^{\frac{1}{2}+(\frac{1}{2}-\delta)\frac{p-5}{2}}\tau^{-\frac{p-1}{4}+\frac{p+3}{2}\delta}.
    \end{equation*}

    The bound for $L_a\mathcal N$ follows in the same way by replacing $\Omega$ with $L_a$. This proves \eqref{eq:int_weightedL2-nonlinear}.
    \end{proof}

With the boundary estimates and nonlinear estimates at hand, we can now apply the interior natural energy and conformal energy inequalities to $\phi$ and its derivatives. This will improve the bootstrap assumptions made in \eqref{est:Boot_in} and thus close our global existence proof in the interior region. 

   \begin{proof}[Proof of Proposition \ref{prop:global_in_existence}.] \,
   
\noindent\textbf{Step 1: Initial data at $t=t_1$.}
By Proposition~\ref{prop:estimates_local1}, we have the $L^2$ bounds on the initial slice
$\Sigma_{t_1}\cap\mathcal D^{in}$:
\begin{equation}\label{eq:init-in}
\begin{aligned}
\sum_{|J|\le1}\|\Omega^J\phi\|_{L^2(\Sigma_{t_1}\cap\mathcal D^{in})}
+\sum_a\|L_a\phi\|_{L^2(\Sigma_{t_1}\cap\mathcal D^{in})}
+\sum_{|J|\le1}\|\partial\Omega^J\phi\|_{L^2(\Sigma_{t_1}\cap\mathcal D^{in})}
&\lesssim C_2\epsilon,\\
\|\partial\partial\phi\|_{L^2(\Sigma_{t_1}\cap\mathcal D^{in})}
+\sum_a\|\partial L_a\phi\|_{L^2(\Sigma_{t_1}\cap\mathcal D^{in})}
&\lesssim C_2.
\end{aligned}
\end{equation}
\textbf{Step 2: Natural energy for $\Omega^J\phi$ ($|J|\le1$).}
Applying $\Omega^J$ to the wave equation for $\phi$ and using the natural energy estimate
\eqref{est:energy_in1}, we obtain, for $|J|\le1$,
\begin{equation*}
\begin{aligned}
\mathcal E^{in}(\tau,\Omega^J\phi)^{\frac12}
&\lesssim \|\partial\Omega^J\phi\|_{L^2(\Sigma_{t_1}\cap\mathcal D^{in})}
+\Big(\int_{\mathcal B_{4\epsilon}}|G\Omega^J\phi|^2\,dx\Big)^{\frac12}
+\int_{\tau_0}^{\tau}
\|\Omega^J\mathcal N(\psi,\phi)\|_{L^2(\mathcal H_{\tilde\tau})}\,d\tilde\tau.
\end{aligned}
\end{equation*}
Employing~\eqref{eq:init-in},~\eqref{est:boundary1} and Lemma~\ref{lem:nonlinear_in}, we obtain
\begin{equation}\label{eq:Ein-Omega}
    \begin{aligned}
        \mathcal E^{in}(\tau,\Omega^J\phi)^{\frac12}&\lesssim C_2\epsilon+C_3\epsilon^{\frac{1}{2}}+\int_{\tau_0}^\tau C_4^p\epsilon^{\frac{1}{2}+(\frac{1}{2}-\delta)\frac{p-5}{2}}\tilde{\tau}^{-\frac{p+3}{4}+\frac{p+3}{2}\delta}\, d\tilde{\tau}
        \lesssim C_3\epsilon^{\frac{1}{2}}+C_4^p\epsilon^{\frac{1}{2}+(\frac{1}{2}-\delta)\frac{p-5}{2}}.
    \end{aligned}
\end{equation}
\textbf{Step 3: Natural energy for $\partial\phi$ and $L_a\phi$.}
Next, we apply the natural energy estimate~\eqref{est:energy_in1} to the differentiated equations. For $\partial\phi$, we obtain
\begin{equation*}
   	\begin{aligned}
   		\mathcal{E}^{in}(\tau,\partial\phi)^{\frac{1}{2}}&\lesssim\|\partial\partial\phi\|_{L^2(\Sigma_{t_1}\cap\mathcal{D}^{in})}+\big(\int_{\mathcal{B}_{4\epsilon}}|G\partial\phi|^2\, dx\big)^{\frac{1}{2}}
        +\int_{\tau_0}^{\tau}\|\partial\mathcal N(\psi,\phi)\|_{L^2(\mathcal{H}_{\tilde{\tau}})}\, d\tilde{\tau}.
   	\end{aligned}
   \end{equation*}
   Combining~\eqref{eq:init-in},~\eqref{est:boundary1} and Lemma~\ref{lem:nonlinear_in}, we conclude
   \begin{equation}\label{eq:Ein-partial}
       \begin{aligned}
           \mathcal{E}^{in}(\tau,\partial\phi)^{\frac{1}{2}}
           &\lesssim C_2+C_3\epsilon^{\frac{1}{2}}+\int_{\tau_0}^{\tau}C_4^p\epsilon^{\frac{1}{2}+(\frac{1}{2}-\delta)\frac{p-5}{2}}\tilde{\tau}^{-\frac{p+3}{4}+\frac{p+3}{2}\delta}\, d\tilde{\tau}
           \lesssim C_3+C_4^p\epsilon^{\frac{1}{2}+(\frac{1}{2}-\delta)\frac{p-5}{2}}.
       \end{aligned}
   \end{equation}
   Similarly, applying \eqref{est:energy_in1} to $L_a\phi$ $(a=1,2,3)$, we obtain
\begin{equation}\label{eq:Ein-La-final}
    \begin{aligned}
        \mathcal E^{in}(\tau,L_a\phi)^{\frac12}
        &\lesssim \|\partial L_a\phi\|_{L^2(\Sigma_{t_1}\cap\mathcal D^{in})}
        +\Big(\int_{\mathcal B_{4\epsilon}}|GL_a\phi|^2\,dx\Big)^{\frac12}
        +\int_{\tau_0}^{\tau}\|L_a\mathcal N(\psi,\phi)\|_{L^2(\mathcal{H}_{\tilde{\tau}})}\,d\tilde\tau\\
        &\lesssim C_3 + C_4^p\epsilon^{\frac{1}{2}+(\frac{1}{2}-\delta)\frac{p-5}{2}},
    \end{aligned}
\end{equation}
where we used \eqref{eq:init-in}, Lemma~\ref{lem:boundary} and Lemma \ref{lem:nonlinear_in}.

\noindent\textbf{Step 4: Conformal energy for $\Omega^J\phi$ ($|J|\le1$).}
Using the conformal energy estimate \eqref{est:energy_in2} in Lemma~\ref{lem:energy_in}, we have
\begin{equation*}
   \begin{aligned}
   		\mathcal{E}^{in}_{con}(\tau,\Omega^J\phi)^{\frac{1}{2}}&\lesssim\|\partial\Omega^J\phi\|_{L^2(\Sigma_{t_1}\cap\mathcal{D}^{in})}+\|\Omega^J\phi\|_{L^2(\Sigma_{t_1}\cap\mathcal{D}^{in})}\\
        &+\left(\int_{\mathcal{B}_{4\epsilon}}\Big[(t+r+1)L\Omega^J\phi + 2\Omega^J\phi\Big]^2+\Big|\frac{1}{r}\Omega\Omega^J\phi\Big|^{2}\, dx\right)^{\frac{1}{2}}
        +\int_{\tau_0}^{\tau}\|\tilde{\tau}\Omega^J\mathcal N(\psi,\phi)\|_{L^2(\mathcal{H}_{\tilde{\tau}})}\, d\tilde{\tau}.
   	\end{aligned}
   \end{equation*}
Using~\eqref{eq:init-in} and Lemmas~\ref{lem:boundary}-\ref{lem:nonlinear_in}, we have
\begin{equation}\label{eq:Econ-Omega}
    \begin{aligned}
        \mathcal{E}^{in}_{con}(\tau,\Omega^J\phi)^{\frac{1}{2}}
        &\lesssim C_2\epsilon+C_3\epsilon^{\frac{1}{2}}+\int_{\tau_0}^{\tau}C_4^p\epsilon^{\frac{1}{2}+(\frac{1}{2}-\delta)\frac{p-5}{2}}\tilde{\tau}^{-\frac{p-1}{4}+\frac{p+3}{2}\delta}\, d\tilde{\tau}
        \lesssim C_3\epsilon^{\frac{1}{2}}+C_4^p\epsilon^{\frac{1}{2}+(\frac{1}{2}-\delta)\frac{p-5}{2}}.
    \end{aligned}
\end{equation}
\textbf{Step 5: Conformal energy for $L_a\phi$ $(a=1,2,3)$.} Similarly, applying the conformal energy estimate~\eqref{est:energy_in2} to $L_a\phi$ yields
\begin{equation}\label{eq:Econ-La}
\begin{aligned}
\mathcal E^{in}_{con}(\tau,L_a\phi)^{\frac12}
&\lesssim \|\partial L_a\phi\|_{L^2(\Sigma_{t_1}\cap\mathcal D^{in})}
+\|L_a\phi\|_{L^2(\Sigma_{t_1}\cap\mathcal D^{in})}\\
&\quad+\Big(\int_{\mathcal B_{4\epsilon}}\big[(t+r+1)LL_a\phi+2L_a\phi\big]^2
+\Big|\frac1r\Omega L_a\phi\Big|^2\,dx\Big)^{\frac12}
+\int_{\tau_0}^{\tau}\|\tilde\tau\,L_a\mathcal N(\psi,\phi)\|_{L^2(\mathcal H_{\tilde\tau})}\,d\tilde\tau\\
&\lesssim C_2+C_3\epsilon^{\frac{1}{2}}+\int_{\tau_0}^{\tau}C_4^p\epsilon^{\frac{1}{2}+(\frac{1}{2}-\delta)\frac{p-5}{2}}\tilde{\tau}^{-\frac{p-1}{4}+\frac{p+3}{2}\delta}\, d\tilde{\tau}\lesssim C_3+C_4^p\epsilon^{\frac{1}{2}+(\frac{1}{2}-\delta)\frac{p-5}{2}}.
\end{aligned}
\end{equation}
\textbf{Step 6: Closing the bootstrap argument.}
Combining \eqref{eq:Ein-Omega}, \eqref{eq:Ein-partial}, \eqref{eq:Ein-La-final}, \eqref{eq:Econ-Omega} and \eqref{eq:Econ-La}, we have
\begin{align*}
    \sum_{|J|\le1}\mathcal E^{in}(\tau,\Omega^J\phi)^{\frac12}
    +\sum_{|J|\le1}\mathcal E^{in}_{con}(\tau,\Omega^J\phi)^{\frac12}
    &\lesssim C_3\epsilon^{\frac12}+ C_4^p\epsilon^{\frac12+(\frac{1}{2}-\delta)\frac{p-5}{2}},\\
    \mathcal E^{in}(\tau,\partial\phi)^{\frac12}
    +\sum_a\mathcal E^{in}(\tau,L_a\phi)^{\frac12}
    +\sum_a\mathcal E^{in}_{con}(\tau,L_a\phi)^{\frac12}
    &\lesssim C_3 + C_4^p\epsilon^{\frac{1}{2}+(\frac{1}{2}-\delta)\frac{p-5}{2}}.
\end{align*}

Choosing $C_4$ large enough so that $C_3\ll C_4$ and taking $\epsilon$ sufficiently small so that
\begin{equation*}
    C_4^{p-1}\epsilon^{(\frac{1}{2}-\delta)\frac{p-5}{2}} \ll 1
\end{equation*}
yields improved bounds with the factor $\frac12 C_4\epsilon^{\frac{1}{2}}$ (and $\frac12 C_4$) compared to \eqref{est:Boot_in}. This closes the bootstrap and completes the proof.
   
   \end{proof}

\section{Global existence for classical solutions}\label{sec:classical}
In this section, we upgrade the global solution constructed in the previous sections to a global
classical solution. Recall that $\phi$ is already globally defined on
$\mathcal D^{ex}\cup\mathcal D^{in}$ for $t\ge t_1$.
In Sections \ref{subsec:step1}, \ref{subsec:step2}, \ref{subsec:step3} we will (respectively) derive pointwise control of $\phi$, $\partial\phi$ and
$\partial^2\phi$ in $\mathcal D^{in}$. Combined with the corresponding exterior regularity estimates from
Section~\ref{sp:exterior}, we can then conclude that $\phi$ is a global classical solution to the equation~\eqref{eq:model_short-pulse} for $t\ge t_1$.

\subsection{Pointwise bound for $\phi$}\label{subsec:step1}
In the previous section we have already obtained pointwise bound for $\phi$ in the interior part
$\{r\le t/4\}$ of $\mathcal D^{in}$ (see \eqref{est:pointwise_in1}). In this subsection we derive the remaining pointwise control on
$\{r\ge t/8\}$ by first establishing hyperboloidal energy bounds for $\Omega^2\phi$ and then applying
the Sobolev inequality.

\begin{lemma}[Energy bounds for $\Omega^2\phi$]\label{lem:Omega2-energy}
For all $\tau\in[\tau_0,+\infty)$,
\begin{equation}\label{eq:Omega2-energy}
\mathcal E^{in}(\tau,\Omega^2\phi)^{\frac12}+
\mathcal E^{in}_{con}(\tau,\Omega^2\phi)^{\frac12}\lesssim C_4^p.
\end{equation}
\end{lemma}

\begin{proof}
Commuting $\Omega^2$ with \eqref{eq:model_short-pulse} gives $
-\Box\Omega^2\phi=\pm\Omega^2\mathcal N(\psi,\phi)$.
By a direct differentiation of $\mathcal N$ we obtain the pointwise bound
\begin{equation}\label{eq:N-Omega2}
\begin{aligned}
|\Omega^2\mathcal N(\psi,\phi)|
&\lesssim
\Big[
|\phi|^{p-1}|\Omega^2\phi|
+ |\phi|^{p-2}\big(|\Omega\phi|^2+|\Omega\phi||\Omega\psi|\big)
+ |\psi|^{p-2}|\Omega\phi|^2\Big]\\
&+ \Big[|\phi|^{p-1}|\Omega^2\psi| 
+|\phi|^{p-2}|\Omega\psi|^2\Big]\\
&+\Big[|\psi|^{p-1}|\Omega^2\phi|
+|\phi||\psi|^{p-2}|\Omega^2\psi|
+|\psi|^{p-2}|\Omega\phi||\Omega\psi|
+|\phi||\psi|^{p-3}|\Omega\psi|^2 \Big] =: \mathcal{G}_1 + \mathcal{G}_2 + \mathcal{G}_3.
\end{aligned}
\end{equation}

$\bullet$
 We start by applying the interior natural energy inequality \eqref{est:energy_in1} to the commuted equation:
\begin{equation}\label{eq:Omega2-energy-raw}
\begin{aligned}
\mathcal E^{in}(\tau,\Omega^2\phi)^{\frac12}
\lesssim\;&
\|\partial\Omega^2\phi\|_{L^2(\Sigma_{t_1}\cap\mathcal D^{in})}
+\Big(\int_{\mathcal B_{4\epsilon}}|G\Omega^2\phi|^2\,dx\Big)^{\frac12}
+\int_{\tau_0}^{\tau}\|\Omega^2\mathcal N(\psi,\phi)\|_{L^2(\mathcal H_{\tilde\tau})}\,d\tilde\tau.
\end{aligned}
\end{equation}
We now estimate each of these terms in turn. Firstly, the term involving initial data can be controlled by Proposition~\ref{prop:estimates_local1} as
\begin{equation*}
\|\partial\Omega^2\phi\|_{L^2(\Sigma_{t_1}\cap\mathcal D^{in})}\lesssim C_2.
\end{equation*}
Next the boundary flux term can be controlled using~\eqref{eq:E-ex-partialOmega2}, to get
\begin{equation*}
\Big(\int_{\mathcal B_{4\epsilon}}|G\Omega^2\phi|^2\,dx\Big)^{\frac12}\lesssim C_3.
\end{equation*}

We next turn to the spacetime intergral term in \eqref{eq:Omega2-energy-raw} which we will analyse using \eqref{eq:N-Omega2}. We make three groupings of the terms appearing in \eqref{eq:N-Omega2}  based on how they are estimated. 
We will also use repeatedly the decomposition $\mathcal H_\tau=\Hcali\cup\Hcalf$, 
the interior pointwise bound \eqref{est:pointwise_in1}, the $L^q$ bounds on $\{r\ge t/8\}$ from
Lemma~\ref{lem:pointwise_in} (in particular \eqref{est:Lq_in2}), and the following decay estimates for $\psi$,
\begin{equation}\label{eq:psi-decay-used}
|\Omega^{\le 2}\psi|\lesssim C_1\epsilon^{\frac12-\delta},
\qquad
|\Omega^{\le 2}\psi|\lesssim C_1 t^{-\frac12+\delta},
\qquad
|\Omega^{\le 2}\psi|\lesssim C_1\epsilon^{-1} \langle t+r\rangle^{-1}\langle t-r\rangle^{-\frac{1}{2}}.
\end{equation}

\emph{Group $\mathcal{G}_1$:} We start with the term $|\phi|^{p-1}\Omega^2\phi$ as the remainder in this group can be treated in the same way. Using \eqref{eq:Htau-split}, Hölder, and Lemma~\ref{lem:Sobo_embedding}, we find
\begin{equation*}
    \begin{aligned}
        \big\||\phi|^{p-1}\Omega^2\phi\big\|_{L^2(\mathcal H_\tau)}
        &\lesssim\big\|\frac{\tau}{t+1}\Omega^2\phi\big\|_{L^2(\mathcal H_\tau)}\big\|\frac{t+1}{\tau}|\phi|^{p-1}\big\|_{L^\infty(\Hcali)}\\
        &+\big\|\frac{\tau}{t+1}\Omega^2\phi\big\|_{L^6(\Hcalf)}\big\|\frac{t+1}{\tau}|\phi|^{p-1}\big\|_{L^3(\Hcalf)}.
    \end{aligned}
\end{equation*}
By \eqref{est:pointwise_in1} and $\tau\le t+1\le \tau^2$ on $\mathcal H_\tau$ we have
\begin{equation*}
    \big\|\frac{t+1}{\tau}|\phi|^{p-1}\big\|_{L^\infty(\Hcali)}\lesssim C_4^{p-1}\epsilon^{\frac{p-1}{8}}\tau^{-\frac{3}{2}p+\frac{3}{2}},
\end{equation*}
and the $L^q$ bound on $\{r\ge t/8\}$ gives
\begin{equation*}
    \big\|\frac{t+1}{\tau}|\phi|^{p-1}\big\|_{L^3(\Hcalf)}
\lesssim C_4^{p-1}\epsilon^{\frac{p-1}{2}}\tau^{-\frac{3}{2}p+\frac{7}{2}}.
\end{equation*}
Applying Lemma~\ref{lem:Sobo_embedding}, we obtain
\begin{equation*}
    \big\|\frac{\tau}{t+1}\Omega^2\phi\big\|_{L^6(\Hcalf)}\lesssim\big\|\frac{\tau}{t+1}\Omega^2\phi\big\|_{H^1(\Hcalf)}.
\end{equation*}
Combining the above bounds, we conclude
\begin{equation*}
    \big\||\phi|^{p-1}\Omega^2\phi\big\|_{L^2(\mathcal H_\tau)}\lesssim C_4^{p}\epsilon^{\frac{1}{2}+\frac{p-1}{8}}\tau^{-\frac{3}{2}p+\frac{7}{2}}+C_4^{p-1}\epsilon^{\frac{p-1}{2}}\tau^{-\frac{3}{2}p+\frac{7}{2}}\mathcal{E}^{in}(\tau,\Omega^2\phi)^{\frac{1}{2}}.
\end{equation*}
Similarly, we decompose
\begin{equation*}
    \begin{aligned}
        \big\||\psi|^{p-2}\Omega\phi\Omega\phi\big\|_{L^2(\mathcal H_\tau)}
        &\lesssim\big\|\frac{\tau}{t+1}\Omega\phi\big\|^2_{L^4(\mathcal H_\tau)}\big\|\frac{(t+1)^2}{\tau^2}|\psi|^{p-2}\big\|_{L^\infty(\mathcal H_\tau)}.
    \end{aligned}
\end{equation*}
Employing Lemmas~\ref{lem:Sobo_embedding} and~\ref{lem:inter_inequ}, we obtain
\begin{equation*}
    \begin{aligned}
        \big\|\frac{\tau}{t+1}\Omega\phi\big\|_{L^4(\mathcal H_\tau)}&\lesssim\big\|\frac{\tau}{t+1}\Omega\phi\big\|^{\frac{1}{4}}_{L^2(\mathcal H_\tau)}\big\|\frac{\tau}{t+1}\Omega\phi\big\|^{\frac{3}{4}}_{H^1(\mathcal H_\tau)}\lesssim C_4\epsilon^{\frac{1}{2}}.
    \end{aligned}
\end{equation*}
The pointwise bounds in~\eqref{eq:psi-decay-used} yield
\begin{equation*}
    \big\|\frac{(t+1)^2}{\tau^2}|\psi|^{p-2}\big\|_{L^\infty(\mathcal{H}_{\tau})}\lesssim C_1^{p-2}\epsilon^{-1}\tau^{-\frac{p}{2}+\frac{1}{2}+(p-3)\delta},
\end{equation*}
where we used that fact $\tau\leq t+1\leq \tau^2$ on $\mathcal{H}_\tau$.
Using the above estimates, we derive
\begin{equation*}
    \begin{aligned}
         \big\||\psi|^{p-2}\Omega\phi\Omega\phi\big\|_{L^2(\mathcal H_\tau)}\lesssim C_4^p\tau^{-\frac{p}{2}+\frac{1}{2}+(p-3)\delta}.
    \end{aligned}
\end{equation*}
The remaining terms in $\mathcal{G}_1$ are treated in the same way, using the same decomposition~\eqref{eq:Htau-split}, together with Lemmas~\ref{lem:Sobo_embedding}--\ref{lem:inter_inequ}, \eqref{est:pointwise_in1}, \eqref{est:Lq_in2} and \eqref{eq:psi-decay-used}.

\emph{Group $\mathcal{G}_2$:} We start with the term $|\phi|^{p-1}\Omega^2\psi$. Using the decomposition~\eqref{eq:Htau-split}, we have 
\begin{equation*}
    \begin{aligned}
        \||\phi|^{p-1}\Omega^2\psi\|_{L^2(\mathcal{H}_\tau)}
        &\lesssim\big\|\frac{\tau}{t+1}\phi\big\|_{L^2(\mathcal H_\tau)}\big\|\frac{t+1}{\tau}|\phi|^{p-2}\big\|_{L^\infty(\Hcali)}\big\|\Omega^2\psi\big\|_{L^\infty(\mathcal{H}_\tau)}\\
        &+\||\phi|^{p-1}\|_{L^2(\Hcalf)}\|\Omega^2\psi\|_{L^\infty(\mathcal{H}_\tau)}=:\mathcal{A}_1+\mathcal{A}_2.
    \end{aligned}
\end{equation*}
To estimate $\mathcal{A}_1$, we apply \eqref{est:pointwise_in1} and $\tau\le t+1\le \tau^2$ on $\mathcal H_\tau$ to find
\begin{equation*}
    \big\|\frac{t+1}{\tau}|\phi|^{p-2}\big\|_{L^\infty(\Hcali)}\lesssim C_4^{p-2}\epsilon^{\frac{p-2}{8}}\tau^{-\frac{3}{2}p+3}.
\end{equation*}
This implies
\begin{equation*}
    \mathcal{A}_1\lesssim C_4^p\epsilon^{\frac{p+6}{8}-\delta}\tau^{-\frac{3}{2}p+3}.
\end{equation*}
Next to estimate $\mathcal{A}_2$, we take the $L^q$ bound on $\{r\ge t/8\}$~\eqref{est:Lq_in2} with $q=2p-2$. This gives
\begin{equation*}
    \||\phi|^{p-1}\|_{L^2(\Hcalf)}=\|\phi\|^{p-1}_{L^{2p-2}(\Hcalf)}\lesssim C_4^{p-1}\epsilon^{\frac{p-1}{2}}\tau^{-\frac{3}{2}p+3}.
\end{equation*}
Thus, in combination with~\eqref{eq:psi-decay-used}, we have
\begin{equation*}
    \mathcal{A}_2\lesssim C_4^p\epsilon^{\frac{p}{2}-\delta}\tau^{-\frac{3}{2}p+3}.
\end{equation*}
In a similar way, we get
\begin{equation*}
    \begin{aligned}
        \||\phi|^{p-2}\Omega\psi\Omega\psi\|_{L^2(\mathcal{H}_\tau)}
        &\lesssim\big\|\frac{\tau}{t+1}\phi\big\|_{L^2(\mathcal H_\tau)}\big\|\frac{t+1}{\tau}|\phi|^{p-3}\big\|_{L^\infty(\Hcali)}\big\|\Omega\psi\big\|^2_{L^\infty(\mathcal{H}_\tau)}\\
        &+\||\phi|^{p-2}\|_{L^2(\Hcalf)}\|\Omega\psi\|^2_{L^\infty(\mathcal{H}_\tau)}
        \lesssim C_4^p\epsilon^{\frac{p+5}{8}-\delta}\tau^{-\frac{3}{2}p+\frac{9}{2}}.
    \end{aligned}
\end{equation*}

\emph{Group $\mathcal{G}_3$:}  Lastly, we consider $|\psi|^{p-1}\Omega^2\phi$. 
Using \eqref{eq:psi-decay-used},~\eqref{est:L2_in3} and $\tau\le t+1\le \tau^2$ on $\mathcal H_\tau$, we derive
\begin{equation*}
    \begin{aligned}
        \||\psi|^{p-1}\Omega^2\phi\|_{L^2(\mathcal{H}_\tau)}
        &\lesssim\big\|\frac{\tau}{t+1}\Omega^2\phi\big\|_{L^2(\mathcal H_\tau)}\big\|\frac{t+1}{\tau}|\psi|^{p-1}\big\|_{L^\infty(\mathcal H_\tau)}\lesssim C_4^p\epsilon^{1-\delta}\tau^{-\frac{p}{2}+1+(p-2)\delta}.
    \end{aligned}
\end{equation*}
The other terms in $\mathcal{G}_3$ are bounded analogously.

Collecting the above estimates yields
\begin{equation}\label{eq:Omega2N-L2}
\|\Omega^2\mathcal N(\psi,\phi)\|_{L^2(\mathcal H_{\tau})}
\ \lesssim
C_4^{p-1}\epsilon^{\frac{p-1}{2}}\tau^{-\frac{3}{2}p+\frac{7}{2}}\mathcal E^{in}(\tau,\Omega^2\phi)^{\frac12}
 + C_4^p\tau^{-\frac{p}{2}+1+(p-2)\delta}.
\end{equation}
Inserting \eqref{eq:Omega2N-L2} into the energy inequality \eqref{eq:Omega2-energy-raw} and applying Gronwall inequality, we obtain the desired bound in \eqref{eq:Omega2-energy}.

$\bullet$ We conclude by turning to the conformal estimate, which is obtained in a similar manner.  Applying \eqref{est:energy_in2} to the commuted equation gives
\begin{equation*}
\begin{aligned}
    \mathcal E^{in}_{con}(\tau,\Omega^2\phi)^{1/2}
&\lesssim
\|\partial\Omega^2\phi\|_{L^2(\Sigma_{t_1}\cap\mathcal D^{in})}
+\|\Omega^2\phi\|_{L^2(\Sigma_{t_1}\cap\mathcal D^{in})}\\
&+\left(\int_{\mathcal{B}_{4\epsilon}}\Big[(t+r+1)L\Omega^2\phi + 2\Omega^2\phi\Big]^2+\Big|\frac{1}{r}\Omega
^3\phi\Big|^{2}\, dx\right)^{1/2}+\int_{\tau_0}^{\tau}\|\tilde\tau\,\Omega^2\mathcal N(\psi,\phi)\|_{L^2(\mathcal H_{\tilde\tau})}\,d\tilde\tau.
\end{aligned}
\end{equation*}

The initial and boundary terms are controlled exactly as above, using Propositions~\ref{prop:estimates_local1}, \ref{prop:energy_ex} and Lemma~\ref{lem:refined_ex}. Finally the spacetime integral is estimated in the same way as \eqref{eq:Omega2N-L2}, with an additional factor
$\tilde\tau$, which remains integrable. 
\end{proof}

\begin{lemma}[Pointwise bound for $\phi$]\label{lem:phi-pt}
Let $(t,x)\in\mathcal H_\tau$. Then we have
\begin{equation}\label{eq:phi-pt}
    \begin{aligned}
        |\phi(t,x)|&\lesssim C_4\epsilon^{\frac{1}{8}}\tau^{-1}t^{-\frac{1}{2}},\qquad\qquad r\leq t/4,\\
        |\phi(t,x)|&\lesssim C_4^{\frac{p+1}{2}}\epsilon^{\frac{1}{4}}\tau^{-1}r^{-\frac{1}{2}},\qquad\ \ \  r\geq t/8.
    \end{aligned}
\end{equation}
\end{lemma}

\begin{proof}
On $\{r\geq t/8\}$, using~\eqref{est:KS_in3}
    \begin{equation*}
        \begin{aligned}
        |\phi(t,x)|\lesssim|\phi(t,x)|_{\mathcal{H}_\tau\cap\mathcal{B}_{4\epsilon}}+\sum_{\substack{|J_1|\leq 2,|J_2|\leq2\\|I|=1}}\tau^{-1}|x|^{-\frac{1}{2}}\left\|\frac{\tau}{t}\Omega^{J_1}\phi\right\|_{L^2(\mathcal{H}_\tau)}^{\frac{1}{2}}\left\|\frac{\tau}{t}\Omega^{J_2}Z^{I}\phi\right\|_{L^2(\mathcal{H}_\tau)}^{\frac{1}{2}},
        \end{aligned}
    \end{equation*}
where $Z=\{\partial_t, L_1, L_2, L_3\}.$
Applying Proposition~\ref{prop:L2estimates_in}, Lemma~\ref{lem:Omega2-energy} and~\eqref{est:L2_ex3}, we obtain
\begin{equation*}
    |\phi(t,x)|\lesssim C_4^{\frac{p+1}{2}}\epsilon^{\frac{1}{4}}\tau^{-1}r^{-\frac{1}{2}}.
\end{equation*}
Together with the already known pointwise bound in $\{r\le t/4\}$, we obtain the desired global pointwise control
of $\phi$ in $\mathcal D^{in}$.
\end{proof}

\subsection{Pointwise bound for $\partial\phi$}\label{subsec:step2}
In this subsection, we derive the pointwise estimate for $\partial\phi$.

\begin{proposition}[$L^2$ bounds for higher derivatives I]\label{prop:high-L2-1}
For all $\tau\in[\tau_0,+\infty)$, the following bounds hold:
    \begin{equation}\label{eq:high-L2-1}
        \begin{aligned}
            \mathcal{E}^{in}(\tau,\partial^2\phi)^{\frac{1}{2}}+\mathcal{E}_{con}^{in}(\tau,\partial^2\phi)^{\frac{1}{2}}&\lesssim C_3\epsilon^{-2}+C_4^{\frac{p^2+1}{2}}\epsilon^{\frac{p-1}{8}}+C_4^p\epsilon^{-\frac{3}{4}},\\
            \mathcal{E}^{in}(\tau,L_a\partial\phi)^{\frac{1}{2}}+\mathcal{E}_{con}^{in}(\tau,L_a\partial\phi)^{\frac{1}{2}}&\lesssim C_2\epsilon^{-1}+C_3+C_4^{\frac{p^2+1}{2}}\epsilon^{\frac{p-1}{8}}+C_4^p\epsilon^{-\frac{3}{4}}
            .
        \end{aligned}
    \end{equation}
\end{proposition}

\begin{proof}
    Commuting $\partial^2$ with \eqref{eq:model_short-pulse} gives $-\Box\partial^2\phi=\pm\partial^2\mathcal N(\psi,\phi)$.
    By a direct differentiation of $\mathcal N$ we obtain the pointwise bound
    \begin{equation}\label{eq:N-partial2}
        \begin{aligned}
        |\partial^2\mathcal N(\psi,\phi)|
        &\lesssim \Big[
        |\phi|^{p-1}|\partial^2\phi|
        +|\phi|^{p-1}|\partial^2\psi|
        +|\psi|^{p-1}|\partial^2\phi|
        +|\phi||\psi|^{p-2}|\partial^2\psi|\\
        &+|\phi|^{p-2}\big(|\partial\phi||\partial\psi|+|\partial\psi|^2\big)+|\psi|^{p-2}|\partial\phi||\partial\psi|
        +|\phi||\psi|^{p-3}|\partial\psi|^2 \Big]
        \\&+\Big[ |\phi|^{p-2}|\partial\phi|^2 + |\psi|^{p-2}|\partial\phi|^2\Big] =: \mathcal{G}_1 + \mathcal{G}_2.
        \end{aligned}
    \end{equation}

    $\bullet$ We start with the natural energy bound. 
    Applying the energy inequality \eqref{est:energy_in1} to the commuted equation, we get
    \begin{equation}\label{eq:partial2-energy-raw}
        \begin{aligned}
            \mathcal E^{in}(\tau,\partial^2\phi)^{\frac12}
            \lesssim\;&
            \|\partial^3\phi\|_{L^2(\Sigma_{t_1}\cap\mathcal D^{in})}
            +\Big(\int_{\mathcal B_{4\epsilon}}|G\partial^2\phi|^2\,dx\Big)^{\frac12}+\int_{\tau_0}^{\tau}\|\partial^2\mathcal N(\psi,\phi)\|_{L^2(\mathcal H_{\tilde\tau})}\,d\tilde\tau.
        \end{aligned}
    \end{equation}
    Again, we estimate these three components separately. Firstly, the initial term can be controlled by Proposition~\ref{prop:estimates_local1} as
    \begin{equation*}
        \|\partial^3\phi\|_{L^2(\Sigma_{t_1}\cap\mathcal D^{in})}\lesssim C_2\epsilon^{-2}.
    \end{equation*}
    Next, using Proposition~\ref{prop:decay_ex}, the boundary flux term is controlled by 
    \begin{equation*}
        \Big(\int_{\mathcal B_{4\epsilon}}|G\partial^2\phi|^2\,dx\Big)^{\frac12}\lesssim C_3\epsilon^{-\frac{3}{2}}.
    \end{equation*}
    Lastly we turn to the spacetime integral term, which we estimate by making two groupings of the terms appearing in \eqref{eq:N-partial2} (the top two lines, and the bottom line).
    We use the $L^2$ bounds in Proposition~\ref{prop:L2estimates_in}, the interior pointwise bounds in Lemma~\ref{lem:phi-pt} and the decay estimates for $\psi$,
        \begin{alignat}{3} \label{eq:partial-psi-decay-used}
            |\psi|&\lesssim C_1 t^{-\frac12+\delta},&\qquad&&|\psi|&\lesssim C_1\epsilon^{-1}\langle t+r\rangle^{-1}\langle t-r\rangle^{-\frac{1}{2}},\\
            |\partial\psi|&\lesssim C_1\langle t+r\rangle^{-1}\langle t-r\rangle^{-\frac{1}{2}},&\qquad
            &&|\partial\partial\psi|&\lesssim C_1\epsilon\langle t+r\rangle^{-1}\langle t-r\rangle^{-\frac{1}{2}}.
        \end{alignat}

    \emph{Group $\mathcal{G}_1$:} 
    For $|\phi|^{p-1}\partial^2\phi$, using Proposition~\ref{prop:L2estimates_in}
    and Lemma~\ref{lem:phi-pt},
    \begin{equation*}
    \begin{aligned}
        \big\||\phi|^{p-1}\partial^2\phi\big\|_{L^2(\mathcal H_\tau)}
        &\lesssim\big\|\frac{\tau}{t+1}\partial^2\phi\big\|_{L^2(\mathcal H_\tau)}\big\|\frac{t+1}{\tau}|\phi|^{p-1}\big\|_{L^\infty(\mathcal H_\tau)}\lesssim C_4^{\frac{p^2+1}{2}}\epsilon^{\frac{p-1}{8}}\tau^{-\frac{3}{2}p+\frac{3}{2}}.
    \end{aligned}
    \end{equation*}
    For $|\psi|^{p-1}\partial^2\phi$, Proposition~\ref{prop:L2estimates_in}
    and~\eqref{eq:partial-psi-decay-used} give
    \begin{equation*}
    \begin{aligned}
        \big\||\psi|^{p-1}\partial^2\phi\big\|_{L^2(\mathcal H_\tau)}
        &\lesssim\big\|\frac{\tau}{t+1}\partial^2\phi\big\|_{L^2(\mathcal H_\tau)}\big\|\frac{t+1}{\tau}|\psi|^{p-1}\big\|_{L^\infty(\mathcal H_\tau)}\lesssim C_4^{p}\tau^{-\frac{p}{2}+\frac{1}{2}+(p-1)\delta}.
    \end{aligned}
    \end{equation*}
    The remaining terms in $\mathcal{G}_1$ are treated in the same way, using  Proposition~\ref{prop:L2estimates_in}, together with Lemma~\ref{lem:phi-pt} and \eqref{eq:partial-psi-decay-used}.

    \emph{Group $\mathcal{G}_2$:} 
    For $|\phi|^{p-2}\partial\phi\partial\phi$, using the H\"{o}lder inequality, we obtain
    \begin{equation*}
    \begin{aligned}
        \big\||\phi|^{p-2}\partial\phi\partial\phi\big\|_{L^2(\mathcal H_\tau)}
        &\lesssim\big\|\frac{\tau}{t+1}\partial\phi\big\|^2_{L^4(\mathcal H_\tau)}\big\|\frac{(t+1)^2}{\tau^2}|\phi|^{p-2}\big\|_{L^\infty(\mathcal H_\tau)}.
    \end{aligned}
\end{equation*}
Employing Lemmas~\ref{lem:Sobo_embedding} and~\ref{lem:inter_inequ}, we obtain
\begin{equation*}
    \begin{aligned}
        \big\|\frac{\tau}{t+1}\partial\phi\big\|_{L^4(\mathcal H_\tau)}&\lesssim\big\|\frac{\tau}{t+1}\partial\phi\big\|^{\frac{1}{4}}_{L^2(\mathcal H_\tau)}\big\|\frac{\tau}{t+1}\partial\phi\big\|^{\frac{3}{4}}_{H^1(\mathcal H_\tau)}\lesssim C_4\epsilon^{\frac{1}{8}},
    \end{aligned}
\end{equation*}
together with Lemma~\ref{lem:phi-pt}, yielding
\begin{equation*}
    \big\||\phi|^{p-2}\partial\phi\partial\phi\big\|_{L^2(\mathcal H_\tau)}\lesssim C_4^{\frac{p}{2}(p-1)+1}\epsilon^{\frac{p}{8}}\tau^{-p+3}.
\end{equation*}
Similarly, we bound $|\psi|^{p-2}\partial\phi\partial\phi$ by
\begin{equation*}
    \begin{aligned}
        \big\||\psi|^{p-2}\partial\phi\partial\phi\big\|_{L^2(\mathcal H_\tau)}
        &\lesssim\big\|\frac{\tau}{t+1}\partial\phi\big\|^2_{L^4(\mathcal H_\tau)}\big\|\frac{(t+1)^2}{\tau^2}|\psi|^{p-2}\big\|_{L^\infty(\mathcal H_\tau)},
    \end{aligned}
\end{equation*}
together with~\eqref{eq:partial-psi-decay-used}, giving
\begin{equation*}
    \big\||\psi|^{p-2}\partial\phi\partial\phi\big\|_{L^2(\mathcal H_\tau)}\lesssim C_4^p\epsilon^{-\frac{3}{4}}\tau^{-\frac{p}{2}+\frac{1}{2}+(p-3)\delta}.
\end{equation*}

Collecting the $\mathcal{G}_1$ and $\mathcal{G}_2$ estimates together yields
\[
\|\partial^2\mathcal N(\psi,\phi)\|_{L^2(\mathcal H_{\tau})}
\ \lesssim (C_4^{\frac{p^2+1}{2}}\epsilon^{\frac{p-1}{8}}+C_4^p\epsilon^{-\frac{3}{4}})\tau^{-\frac{p}{2}+\frac{1}{2}+(p-1)\delta}.
\]
Inserting this into the energy inequality \eqref{eq:partial2-energy-raw}, we obtain
\begin{equation*}
    \mathcal{E}^{in}(\tau,\partial^2\phi)^{\frac{1}{2}}\lesssim C_3\epsilon^{-2}+C_4^{\frac{p^2+1}{2}}\epsilon^{\frac{p-1}{8}}+C_4^p\epsilon^{-\frac{3}{4}}.
\end{equation*}

$\bullet$ Finally we prove the conformal energy estimate, which is obtained in a similar way to the natural energy. Applying the energy inequality \eqref{est:energy_in2} to the equation \eqref{eq:model_short-pulse} commuted with $\partial^2$  gives
\begin{equation*}
\begin{aligned}
    \mathcal E^{in}_{con}(\tau,\partial^2\phi)^{1/2}
&\lesssim
\|\partial^3\phi\|_{L^2(\Sigma_{t_1}\cap\mathcal D^{in})}
+\|\partial^2\phi\|_{L^2(\Sigma_{t_1}\cap\mathcal D^{in})}\\
&+\left(\int_{\mathcal{B}_{4\epsilon}}\Big[(t+r+1)L\partial^2\phi + 2\partial^2\phi\Big]^2+\Big|\frac{1}{r}\Omega
\partial^2\phi\Big|^{2}\, dx\right)^{1/2}
+\int_{\tau_0}^{\tau}\|\tilde\tau\,\partial^2\mathcal N(\psi,\phi)\|_{L^2(\mathcal H_{\tilde\tau})}\,d\tilde\tau.
\end{aligned}
\end{equation*}

The initial and boundary terms are controlled exactly as above, using Propositions~\ref{prop:estimates_local1} and \ref{prop:energy_ex}. The spacetime integral is estimated in the same way as above, with an additional factor $\tilde\tau$, which remains integrable.
For $\mathcal{E}^{in}(\tau,L_a\partial\phi)$ and $\mathcal{E}_{con}^{in}(\tau,L_a\partial\phi)$, with $a=1,2,3$, applying the energy inequalities \eqref{est:energy_in1}--\eqref{est:energy_in2}, we obtain
\begin{equation*}
        \begin{aligned}
            \mathcal E^{in}(\tau,L_a\partial\phi)^{\frac12}
            &\lesssim
            \|\partial L_a\partial\phi\|_{L^2(\Sigma_{t_1}\cap\mathcal D^{in})}
            +\Big(\int_{\mathcal B_{4\epsilon}}|GL_a\partial\phi|^2\,dx\Big)^{\frac12}+\int_{\tau_0}^{\tau}\|L_a\partial\mathcal N(\psi,\phi)\|_{L^2(\mathcal H_{\tilde\tau})}\,d\tilde\tau,\\
            \mathcal E^{in}_{con}(\tau,L_a\partial\phi)^{\frac12}
            &\lesssim
            \|\partial L_a\partial\phi\|_{L^2(\Sigma_{t_1}\cap\mathcal D^{in})}
            +\|L_a\partial\phi\|_{L^2(\Sigma_{t_1}\cap\mathcal D^{in})}\\
            &\quad +\left(\int_{\mathcal{B}_{4\epsilon}}\Big[(t+r+1)LL_a\partial\phi + 2L_a\partial\phi\Big]^2+\Big|\frac{1}{r}\Omega
            L_a\partial\phi\Big|^{2}\, dx\right)^{\frac{1}{2}}
            +\int_{\tau_0}^{\tau}\|\tilde\tau\,L_a\partial\mathcal N(\psi,\phi)\|_{L^2(\mathcal H_{\tilde\tau})}\,d\tilde\tau.
        \end{aligned}
    \end{equation*}
By Proposition~\ref{prop:estimates_local1}, we deduce
\begin{equation*}
    \|\partial L_a\partial\phi\|_{L^2(\Sigma_{t_1}\cap\mathcal{D}^{in})}\lesssim C_2\epsilon^{-1},\qquad
    \|L_a\partial\phi\|_{L^2(\Sigma_{t_1}\cap\mathcal{D}^{in})}\lesssim C_2.
\end{equation*}
Using Proposition~\ref{prop:energy_ex}, we obtain
\begin{equation*}
    \begin{aligned}
        &\Big(\int_{\mathcal B_{4\epsilon}}|LL_a\partial\phi|^2\,dx\Big)^{\frac12} \lesssim
        \Big(\int_{\mathcal B_{4\epsilon}}|GL_a\partial\phi|^2\,dx\Big)^{\frac12}
        \lesssim C_2\epsilon^{-1}+C_3,\\
        &\left(\int_{\mathcal{B}_{4\epsilon}}\Big[(t+r)LL_a\partial\phi + 2L_a\partial\phi\Big]^2\, dx\right)^{\frac{1}{2}}
        +\epsilon\left(\int_{\mathcal{B}_{4\epsilon}}\Big|\frac{1}{r}\Omega
        L_a\partial\phi\Big|^{2}\, dx\right)^{\frac{1}{2}}
        \lesssim C_2\epsilon^{-1}+C_3.
    \end{aligned}
\end{equation*}

The bound for $L_a\partial\mathcal N(\psi,\phi)$ follows in the same way, except that one of the $\partial$ is replaced by an $L_a$. This completes the proof.
\end{proof}

\begin{proposition}[$L^2$ bounds for higher derivatives II]\label{prop:high-L2-2}
For all $\tau\in[\tau_0,+\infty)$, the following bounds hold:
    \begin{equation*}
        \begin{aligned}
            \mathcal{E}^{in}(\tau,\partial\Omega^2\phi)^{\frac{1}{2}}\lesssim C_4^{p^2}\epsilon^{-2},\qquad
            \mathcal{E}_{con}^{in}(\tau,\partial\Omega^2\phi)^{\frac{1}{2}}\lesssim \left\{\begin{aligned}
                C_4^{p^2}\epsilon^{-2},\quad p>\frac{17}{3},\\ C_4^{p^2}\epsilon^{-2}\log\tau ,\quad p=\frac{17}{3},\\
                C_4^{p^2}\epsilon^{-2}\tau^{\frac{17-3p}{2}},\quad p<\frac{17}{3}.
            \end{aligned}\right.
        \end{aligned}
    \end{equation*}
\end{proposition}

\begin{proof}
    Applying the interior natural energy inequality \eqref{est:energy_in1} to the commuted equation $-\Box\partial\Omega^2\phi=\pm\partial\Omega^2\mathcal N(\psi,\phi)$, we get
    \begin{equation}\label{eq:partialOmega2-energy-raw}
        \begin{aligned}
            \mathcal E^{in}(\tau,\partial\Omega^2\phi)^{\frac12}
            \lesssim\;&
            \|\partial^2\Omega^2\phi\|_{L^2(\Sigma_{t_1}\cap\mathcal D^{in})}
            +\Big(\int_{\mathcal B_{4\epsilon}}|G\partial\Omega^2\phi|^2\,dx\Big)^{\frac12}
            +\int_{\tau_0}^{\tau}\|\partial\Omega^2\mathcal N(\psi,\phi)\|_{L^2(\mathcal H_{\tilde\tau})}\,d\tilde\tau.
        \end{aligned}
    \end{equation}
Firstly, the initial data term can be controlled by Proposition~\ref{prop:estimates_local1},
    \begin{equation*}
        \|\partial^2\Omega^2\phi\|_{L^2(\Sigma_{t_1}\cap\mathcal D^{in})}\lesssim C_2\epsilon^{-1},
    \end{equation*}
and the boundary flux term is estimated using \eqref{eq:E-ex-partialOmega2}, to obtain
    \begin{equation*}
        \Big(\int_{\mathcal B_{4\epsilon}}|G\partial\Omega^2\phi|^2\,dx\Big)^{\frac12}\lesssim C_3\epsilon^{-1}.
    \end{equation*}
    
    We now estimate the spacetime term $\|\Gamma^I\mathcal N(\psi,\phi)\|_{L^2(\mathcal H_\tau)}$, where $\Gamma^I=\partial\Omega^2$, by splitting it into four groups. The terms are grouped according to the number of differentiated $\phi$ factors, namely factors of the form $\Gamma^{I_k}\phi$ with $|I_k|\ge 1$: 
$\mathcal{G}_1$ and $\mathcal{G}_3$ contain at most one such factor, whereas $\mathcal{G}_2$ and $\mathcal{G}_4$ contain at least two.
    \begin{equation*}
        \begin{aligned}
            \mathcal{G}_1&:\ \{|\phi|^{p-1}\Gamma^I\phi,\ |\phi|^{p-1}\Gamma^I\psi,\ |\psi|^{p-1}\Gamma^I\phi,\ \phi|\psi|^{p-2}\Gamma^I\psi\},\\
            \mathcal{G}_2&:\ \{|\phi|^{p-2}\Gamma^{I_1}\phi\Gamma^{I_2}\phi,\ |\psi|^{p-2}\Gamma^{I_1}\phi\Gamma^{I_2}\phi\},\\
            \mathcal{G}_3&:\ \{|\phi|^{p-2}\Gamma^{I_1}\phi\Gamma^{I_2}\psi,\ |\phi|^{p-2}\Gamma^{I_1}\psi\Gamma^{I_2}\psi,\ |\psi|^{p-2}\Gamma^{I_1}\phi\Gamma^{I_2}\psi,\ \phi|\psi|^{p-3}\Gamma^{I_1}\psi\Gamma^{I_2}\psi, \\ &\qquad|\phi|^{p-3}\Gamma^{I_1}\phi\Gamma^{I_2}\psi\Gamma^{I_3}\psi,\ |\phi|^{p-3}\Gamma^{I_1}\psi\Gamma^{I_2}\psi\Gamma^{I_3}\psi, |\psi|^{p-3}\Gamma^{I_1}\phi\Gamma^{I_2}\psi\Gamma^{I_3}\psi,\ \phi|\psi|^{p-4}\Gamma^{I_1}\psi\Gamma^{I_2}\psi\Gamma^{I_3}\psi\},\\
            \mathcal{G}_4&:\ \{|\phi|^{p-3}\Gamma^{I_1}\phi\Gamma^{I_2}\phi\Gamma^{I_3}\phi,\ |\phi|^{p-3}\Gamma^{I_1}\phi\Gamma^{I_2}\phi\Gamma^{I_3}\psi, |\psi|^{p-3}\Gamma^{I_1}\phi\Gamma^{I_2}\phi\Gamma^{I_3}\phi,\ |\psi|^{p-3}\Gamma^{I_1}\phi\Gamma^{I_2}\phi\Gamma^{I_3}\psi\} \,.
        \end{aligned}
    \end{equation*}
    Throughout we use the $L^2$ bounds in Propositions~\ref{prop:L2estimates_in},~\ref{prop:high-L2-1}, Lemma~\ref{lem:Omega2-energy}, the interior pointwise bounds in Lemma~\ref{lem:phi-pt} and the decay estimates in Lemma~\ref{lem:disp-psi}.

\noindent
    \emph{Estimates for group $\mathcal{G}_1$.}
    For $|\phi|^{p-1}\partial\Omega^2\phi$, using Lemmas~\ref{lem:Omega2-energy}--\ref{lem:phi-pt}, we have
    \begin{equation*}
    \begin{aligned}
        \big\||\phi|^{p-1}\partial\Omega^2\phi\big\|_{L^2(\mathcal H_\tau)}
        &\lesssim\big\|\frac{\tau}{t+1}\partial\Omega^2\phi\big\|_{L^2(\mathcal H_\tau)}\big\|\frac{t+1}{\tau}|\phi|^{p-1}\big\|_{L^\infty(\mathcal H_\tau)}\lesssim C_4^{\frac{p^2+2p-1}{2}}\epsilon^{\frac{p-1}{8}}\tau^{-\frac{3}{2}p+\frac{3}{2}}.
    \end{aligned}
    \end{equation*}
    For $|\psi|^{p-1}\partial\Omega^2\phi$, Lemma~\ref{lem:Omega2-energy}
    and~\eqref{eq:disp-psi01} give
    \begin{equation*}
    \begin{aligned}
        \big\||\psi|^{p-1}\partial\Omega^2\phi\big\|_{L^2(\mathcal H_\tau)}
        &\lesssim\big\|\frac{\tau}{t+1}\partial\Omega^2\phi\big\|_{L^2(\mathcal H_\tau)}\big\|\frac{t+1}{\tau}|\psi|^{p-1}\big\|_{L^\infty(\mathcal H_\tau)}\lesssim C_4^{2p-1}\tau^{-\frac{p}{2}+\frac{1}{2}+(p-1)\delta}.
    \end{aligned}
    \end{equation*}
    The remaining terms in~$\mathcal{G}_1$ are treated in the same way.

\noindent
    \emph{Estimates for group  $\mathcal{G}_2$.}
    For $|\phi|^{p-2}\Gamma^{I_1}\phi\Gamma^{I_2}\phi$, using the H\"{o}lder inequality, we obtain
    \begin{equation*}
    \begin{aligned}
        \big\||\phi|^{p-2}\Gamma^{I_1}\phi\Gamma^{I_2}\phi\big\|_{L^2(\mathcal H_\tau)}
        &\lesssim\big\|\frac{\tau}{t+1}\Gamma^{I_1}\phi\big\|_{L^4(\mathcal H_\tau)}\big\|\frac{\tau}{t+1}\Gamma^{I_2}\phi\big\|_{L^4(\mathcal H_\tau)}\big\|\frac{(t+1)^2}{\tau^2}|\phi|^{p-2}\big\|_{L^\infty(\mathcal H_\tau)}.
    \end{aligned}
\end{equation*}
Employing Lemmas~\ref{lem:Sobo_embedding} and~\ref{lem:inter_inequ}, we obtain
\begin{equation*}
    \begin{aligned}
        \big\|\frac{\tau}{t+1}\Gamma^{I_1}\phi\big\|_{L^4(\mathcal H_\tau)}&\lesssim\big\|\frac{\tau}{t+1}\Gamma^{I_1}\phi\big\|^{\frac{1}{4}}_{L^2(\mathcal H_\tau)}\big\|\frac{\tau}{t+1}\Gamma^{I_1}\phi\big\|^{\frac{3}{4}}_{H^1(\mathcal H_\tau)},
    \end{aligned}
\end{equation*}
together with Proposition~\ref{prop:high-L2-1}, Lemmas~\ref{lem:Omega2-energy}-\ref{lem:phi-pt}, yielding
\begin{equation*}
    \big\||\phi|^{p-2}\Gamma^{I_1}\phi\Gamma^{I_2}\phi\big\|_{L^2(\mathcal H_\tau)}\lesssim C_4^{\frac{7}{8}p^2-\frac{p}{2}+\frac{5}{8}}\epsilon^{-\frac{1}{2}}\tau^{-p+3}.
\end{equation*}
Similarly, we bound $|\psi|^{p-2}\Gamma^{I_1}\phi\Gamma^{I_2}\phi$ by
\begin{equation*}
    \begin{aligned}
        \big\||\psi|^{p-2}\Gamma^{I_1}\phi\Gamma^{I_2}\phi\big\|_{L^2(\mathcal H_\tau)}
        &\lesssim C_4^{\frac{3}{8}p^2+p}\epsilon^{-\frac{15}{8}}\tau^{-\frac{p}{2}+\frac{1}{2}+(p-3)\delta}.
    \end{aligned}
\end{equation*}

\noindent
    \emph{Estimates for group $\mathcal{G}_3$.} For $|\phi|^{p-2}\Gamma^{I_1}\phi\Gamma^{I_2}\psi$, using the H\"{o}lder inequality, we obtain
    \begin{equation*}
    \begin{aligned}
        \big\||\phi|^{p-2}\Gamma^{I_1}\phi\Gamma^{I_2}\psi\big\|_{L^2(\mathcal H_\tau)}
        &\lesssim\big\|\frac{\tau}{t+1}\Gamma^{I_1}\phi\big\|_{L^2(\mathcal H_\tau)}\big\|\Gamma^{I_2}\psi\big\|_{L^\infty(\mathcal H_\tau)}\big\|\frac{t+1}{\tau}|\phi|^{p-2}\big\|_{L^\infty(\mathcal H_\tau)}.
    \end{aligned}
    \end{equation*}
    Using Proposition~\ref{prop:L2estimates_in}, Lemmas~\ref{lem:disp-psi} and~\ref{lem:phi-pt}, we derive
    \begin{equation*}
        \big\||\phi|^{p-2}\Gamma^{I_1}\phi\Gamma^{I_2}\psi\big\|_{L^2(\mathcal H_\tau)}\lesssim C_4^{\frac{p^2-p+2}{2}}\epsilon^{\frac{p+2}{8}}\tau^{-\frac{3}{2}p+5+\delta}.
    \end{equation*}
    For $\phi|\psi|^{p-3}\Gamma^{I_1}\psi\Gamma^{I_2}\psi$, we have
    \begin{equation*}
        \begin{aligned}
            \big\|\phi&|\psi|^{p-3}\Gamma^{I_1}\psi\Gamma^{I_2}\psi\big\|_{L^2(\mathcal H_\tau)}\\
            &\lesssim\big\|\frac{\tau}{t+1}\phi\big\|_{L^2(\mathcal H_\tau)}\big\|\Gamma^{I_1}\psi\big\|_{L^\infty(\mathcal H_\tau)}\big\|\Gamma^{I_2}\psi\big\|_{L^\infty(\mathcal H_\tau)}\big\|\frac{t+1}{\tau}|\psi|^{p-3}\big\|_{L^\infty(\mathcal H_\tau)}
            \lesssim C_4^{p}\epsilon^{\frac{1}{2}}\tau^{-\frac{p}{2}+(p-2)\delta}.
        \end{aligned}
    \end{equation*}
    The remaining terms in~$\mathcal{G}_3$ are treated in the same way.

\noindent
    \emph{Estimates for group $\mathcal{G}_4$.} For $|\phi|^{p-3}\Gamma^{I_1}\phi\Gamma^{I_2}\phi\Gamma^{I_3}\phi$, using the H\"{o}lder inequality, we obtain
    \begin{equation*}
    \begin{aligned}
        \big\|&|\phi|^{p-3}\Gamma^{I_1}\phi\Gamma^{I_2}\phi\Gamma^{I_3}\phi\big\|_{L^2(\mathcal H_\tau)}\\
        &\lesssim\big\|\frac{\tau}{t+1}\Gamma^{I_1}\phi\big\|_{L^6(\mathcal H_\tau)}\big\|\frac{\tau}{t+1}\Gamma^{I_2}\phi\big\|_{L^6(\mathcal H_\tau)}\big\|\frac{\tau}{t+1}\Gamma^{I_3}\phi\big\|_{L^6(\mathcal H_\tau)}\big\|\frac{(t+1)^3}{\tau^3}|\phi|^{p-3}\big\|_{L^\infty(\mathcal H_\tau)}.
    \end{aligned}
    \end{equation*}
    Employing Lemma~\ref{lem:Sobo_embedding}, we obtain
    \begin{equation*}
        \begin{aligned}
            \big\|\frac{\tau}{t+1}\Gamma^{I_1}\phi\big\|_{L^6(\mathcal H_\tau)}&\lesssim\big\|\frac{\tau}{t+1}\Gamma^{I_1}\phi\big\|_{H^1(\mathcal H_\tau)}.
        \end{aligned}
    \end{equation*}
    Applying Proposition~\ref{prop:L2estimates_in} and Lemma~\ref{lem:phi-pt}, we get
    \begin{equation*}
        \big\||\phi|^{p-3}\Gamma^{I_1}\phi\Gamma^{I_2}\phi\Gamma^{I_3}\phi\big\|_{L^2(\mathcal H_\tau)}\lesssim C_4^{3+\frac{(p+1)(p-3)}{2}}\epsilon^{\frac{p+5}{8}}\tau^{-\frac{3}{2}p+\frac{13}{2}}.
    \end{equation*}
    Similarly, for $|\phi|^{p-3}\Gamma^{I_1}\phi\Gamma^{I_2}\phi\Gamma^{I_3}\psi$, we obtain
    \begin{equation*}
    \begin{aligned}
        \big\|&|\phi|^{p-3}\Gamma^{I_1}\phi\Gamma^{I_2}\phi\Gamma^{I_3}\psi\big\|_{L^2(\mathcal H_\tau)}\\
        &\lesssim\big\|\frac{\tau}{t+1}\Gamma^{I_1}\phi\big\|_{L^4(\mathcal H_\tau)}\big\|\frac{\tau}{t+1}\Gamma^{I_2}\phi\big\|_{L^4(\mathcal H_\tau)}\big\|\Gamma^{I_3}\psi\big\|_{L^\infty(\mathcal H_\tau)}\big\|\frac{(t+1)^2}{\tau^2}|\phi|^{p-3}\big\|_{L^\infty(\mathcal H_\tau)},
    \end{aligned}
    \end{equation*}
    together with Lemmas~\ref{lem:Sobo_embedding} and~\ref{lem:inter_inequ}, \eqref{eq:disp-psi01}, giving
    \begin{equation*}
        \big\||\phi|^{p-3}\Gamma^{I_1}\phi\Gamma^{I_2}\phi\Gamma^{I_3}\psi\big\|_{L^2(\mathcal H_\tau)}\lesssim C_4^{3+\frac{(p+1)(p-3)}{2}}\epsilon^{\frac{p+2}{8}}\tau^{-\frac{3}{2}p+\frac{7}{2}}.
    \end{equation*}
    The remaining terms in~$\mathcal{G}_4$ are treated in the same way.

   Inserting all the above estimates into the energy inequality \eqref{eq:partialOmega2-energy-raw}, we obtain the desired bound.

$\bullet$ Next we turn to the conformal estimate, which proceeds in a very similar manner. Applying \eqref{est:energy_in2} to the equation \eqref{eq:model_short-pulse} commuted with $\partial \Omega^2$ gives
\begin{equation*}
\begin{aligned}
    \mathcal E^{in}_{con}&(\tau,\partial\Omega^2\phi)^{\frac12}
\lesssim
\|\partial^2\Omega^2\phi\|_{L^2(\Sigma_{t_1}\cap\mathcal D^{in})}
+\|\partial\Omega^2\phi\|_{L^2(\Sigma_{t_1}\cap\mathcal D^{in})}\\
&+\left(\int_{\mathcal{B}_{4\epsilon}}\Big[(t+r+1)L\partial\Omega^2\phi + 2\partial\Omega^2\phi\Big]^2+\Big|\frac{1}{r}\Omega\partial\Omega
^2\phi\Big|^{2}\, dx\right)^{\frac{1}{2}}
+\int_{\tau_0}^{\tau}\|\tilde\tau\,\partial\Omega^2\mathcal N(\psi,\phi)\|_{L^2(\mathcal H_{\tilde\tau})}\,d\tilde\tau.
\end{aligned}
\end{equation*}

The initial and boundary terms are controlled exactly as above, using
Propositions~\ref{prop:estimates_local1},~\ref{prop:energy_ex} and~\eqref{eq:E-ex-partialOmega2}. The spacetime term is estimated in the same way, except that the conformal energy inequality introduces an additional factor $\tilde\tau$ in the source term. Consequently, the integral is finite if and only if $p>\frac{17}{3}$. 
\end{proof}

\begin{lemma}[Pointwise bound for $\partial\phi$]\label{lem:der-pt}
For $(t,x)\in\mathcal D^{in}$, we have
\begin{equation}\label{eq:der-pt}
    \begin{aligned}
        |\partial\phi(t,x)|&\lesssim C_4^{\frac{p^2+1}{2}}\epsilon^{-\frac{11}{8}}\tau^{-1}t^{-\frac{1}{2}},\qquad \text{ for } r\leq t/4,\\
        |\partial\phi(t,x)|&\lesssim
        \left\{
	\begin{array}{ll}
	C_4^{\frac{p^2+p}{2}}\epsilon^{-1}\tau^{\frac{13-3p}{4}}r^{-\frac{1}{2}},  & p<\frac{17}{3}, \\
	C_4^{\frac{p^2+p}{2}}\epsilon^{-1}\tau^{-1}(\log\tau)^{\frac{1}{2}} r^{-\frac{1}{2}}, & p=\frac{17}{3},
    \\
    C_4^{\frac{p^2+p}{2}}\epsilon^{-1}\tau^{-1}r^{-\frac{1}{2}}, & p>\frac{17}{3},
	\end{array}
\right.\qquad \text{ for } r\geq t/8.
\end{aligned}
\end{equation}
\end{lemma}

\begin{proof}
Recall that $Z=\{\partial_t, L_1, L_2, L_3\}.$
On $\{r\le t/4\}$, using~\eqref{est:KS_in1} with $\lambda=\frac{5}{4}$, we have
\begin{equation*}
    \begin{aligned}
        |\partial\phi(t,x)|
        \lesssim\tau^{-1}t^{-\frac{1}{2}}\sum_{|I|\leq2}\epsilon^{(|I|-\frac{3}{2})\lambda}\big\|\frac{\tau}{t+1}Z^{I}\partial\phi\big\|_{L^2(\mathcal H_\tau)}.
    \end{aligned}
\end{equation*}
 Applying Propositions~\ref{prop:L2estimates_in} and~\ref{prop:high-L2-1}, we get
 \begin{equation*}
     |\partial\phi(t,x)|\lesssim C_4^{\frac{p^2+1}{2}}\epsilon^{-\frac{11}{8}}\tau^{-1}t^{-\frac{1}{2}}.
 \end{equation*}
On $\{r\geq t/8\}$, using~\eqref{est:KS_in3}, we obtain
    \begin{equation*}
        \begin{aligned}
            &|\partial\phi(t,x)|\lesssim|\partial\phi(t,x)|_{\mathcal{H}_{\tau}\cap\mathcal{B}_{4\epsilon}}
            +\sum_{\substack{|J_1|\leq 2,|J_2|\leq2\\|I|=1}}\tau^{-1}|x|^{-\frac{1}{2}}\left\|\frac{\tau}{t}\Omega^{J_1}\partial\phi\right\|_{L^2(\mathcal{H}_\tau)}^{\frac{1}{2}}\left\|\frac{\tau}{t}\Omega^{J_2}Z^{I}\partial\phi\right\|_{L^2(\mathcal{H}_\tau)}^{\frac{1}{2}},
        \end{aligned}
    \end{equation*}
together with~\eqref{est:L2_ex3}, Propositions~\ref{prop:L2estimates_in},~\ref{prop:high-L2-1},~\ref{prop:high-L2-2}, giving the desired result. 
\end{proof}

\subsection{Pointwise bound for $\partial^2\phi$}\label{subsec:step3}
In this subsection, we derive the pointwise estimate for $\partial^2\phi$.

\begin{proposition}[$L^2$ bounds for higher derivatives III]\label{prop:high-L2-3}
For all $\tau\in[\tau_0,+\infty)$, the following bounds hold:
    \begin{equation}\label{eq:high-L2-3}
        \begin{split}
            \mathcal{E}^{in}(\tau,\partial^3\phi)^{\frac{1}{2}}+\mathcal{E}_{con}^{in}(\tau,\partial^3\phi)^{\frac{1}{2}}&\lesssim C_3\epsilon^{-3}+C_4^{p^2}\epsilon^{-\frac{17}{8}},\\
            \mathcal{E}^{in}(\tau,L_a\partial^2\phi)^{\frac{1}{2}}+\mathcal{E}_{con}^{in}(\tau,L_a\partial^2\phi)^{\frac{1}{2}}&\lesssim C_2\epsilon^{-3}+C_4^{p^2}\epsilon^{-\frac{17}{8}}.
        \end{split}
    \end{equation}
\end{proposition}

\begin{proof}
    Commuting $\Gamma^I=\partial^3\ (\mbox{or}\ L_a\partial^2)$ with \eqref{eq:model_short-pulse} and applying the interior natural energy inequality \eqref{est:energy_in1} produces
  \begin{equation*}
          \mathcal E^{in}(\tau,\Gamma^I\phi)^{\frac12}
          \lesssim\;
          \|\partial\Gamma^I\phi\|_{L^2(\Sigma_{t_1}\cap\mathcal D^{in})}
          +\Big(\int_{\mathcal B_{4\epsilon}}|G\Gamma^I\phi|^2\,dx\Big)^{\frac12}
          +\int_{\tau_0}^{\tau}\|\Gamma^I\mathcal N(\psi,\phi)\|_{L^2(\mathcal H_{\tilde\tau})}\,d\tilde\tau.
    \end{equation*}
Similarly the conformal estimate becomes
\begin{equation*}
\begin{aligned}
    \mathcal E^{in}_{con}(\tau,\Gamma^I\phi)^{\frac12}
&\lesssim
\|\partial\Gamma^I\phi\|_{L^2(\Sigma_{t_1}\cap\mathcal D^{in})}
+\|\Gamma^I\phi\|_{L^2(\Sigma_{t_1}\cap\mathcal D^{in})}\\
&+\left(\int_{\mathcal{B}_{4\epsilon}}\Big[(t+r+1)L\Gamma^I\phi + 2\Gamma^I\phi\Big]^2+\Big|\frac{1}{r}\Omega\Gamma^I\phi\Big|^{2}\, dx\right)^{\frac{1}{2}}
+\int_{\tau_0}^{\tau}\|\tilde\tau\,\Gamma^I\mathcal N(\psi,\phi)\|_{L^2(\mathcal H_{\tilde\tau})}\,d\tilde\tau.
\end{aligned}
\end{equation*}
In both inequalities, the initial and boundary terms are controlled exactly as above, using Propositions~\ref{prop:estimates_local1},~\ref{prop:energy_ex} and~\eqref{eq:E-ex-partialOmega2}.
The estimate for $\partial^3\mathcal N(\psi,\phi)$ is proved in the same way as for $\partial\Omega^2\mathcal N(\psi,\phi)$
in Proposition~\ref{prop:high-L2-2}, with the two $\Omega$--derivatives replaced by two $\partial$--derivatives.
The only difference is that, in the product estimates for $\partial^3\mathcal N(\psi,\phi)$, we may place one factor
$\partial\phi$ in $L^\infty$ (see \eqref{eq:der-pt}), and estimate the remaining
factors in $L^2$. The bound for $L_a\partial^2\mathcal N(\psi,\phi)$ is estimated similarly.
\end{proof}

\begin{proposition}[$L^2$ bounds for higher derivatives IV]\label{prop:high-L2-4}
For all $\tau\in[\tau_0,+\infty)$, the following bounds hold:
    \begin{equation}\label{eq:high-L2-4}
        \begin{aligned}
            \mathcal{E}^{in}(\tau,\Omega^2\partial^2\phi)^{\frac{1}{2}}&\lesssim C_4^{\frac{7}{4}p^2}\epsilon^{-\frac{41}{8}}+C_4^{p^2}\epsilon^{\frac{3}{2}-p},\\
            \mathcal{E}_{con}^{in}(\tau,\Omega^2\partial^2\phi)^{\frac{1}{2}}&\lesssim 
            \left\{
	\begin{array}{ll}
	C_4^{\frac{7}{4}p^2}\epsilon^{-\frac{41}{8}}+C_4^{p^2}\epsilon^{\frac{3}{2}-p}, & p>\frac{17}{3}, \\
	C_4^{\frac{7}{4}p^2}\epsilon^{-\frac{41}{8}}\log\tau,& p=\frac{17}{3}, \\
    C_4^{\frac{7}{4}p^2}\epsilon^{-\frac{41}{8}}\tau^{\frac{17-3p}{2}}, & p<\frac{17}{3}.
	\end{array}
\right.
        \end{aligned}
    \end{equation}
\end{proposition}

\begin{proof}
    $\bullet$ We start with proving the natural energy estimate. Commuting $\Omega^2\partial^2$ with \eqref{eq:model_short-pulse} and applying the interior natural energy inequality \eqref{est:energy_in1} to this commuted equation gives 
  \begin{equation}\label{eq:partial2OMega2-energy-raw}
      \begin{aligned}
          \mathcal E^{in}(\tau,\Omega^2\partial^2\phi)^{\frac12}
          \lesssim\;&
          \|\partial\Omega^2\partial^2\phi\|_{L^2(\Sigma_{t_1}\cap\mathcal D^{in})}
          +\Big(\int_{\mathcal B_{4\epsilon}}|G\Omega^2\partial^2\phi|^2\,dx\Big)^{\frac12}
          +\int_{\tau_0}^{\tau}\|\Omega^2\partial^2\mathcal N(\psi,\phi)\|_{L^2(\mathcal H_{\tilde\tau})}\,d\tilde\tau.
      \end{aligned}
    \end{equation}
   We first estimate the initial data term using  Proposition~\ref{prop:estimates_local1} to find
    \begin{equation*}
        \|\partial\Omega^2\partial^2\phi\|_{L^2(\Sigma_{t_1}\cap\mathcal D^{in})}\lesssim C_2\epsilon^{-2}.
    \end{equation*}
 Next, the boundary flux term is controlled using the exterior bound~\eqref{eq:E-ex-partialOmega2}, giving 
    \begin{equation*}
        \Big(\int_{\mathcal B_{4\epsilon}}|G\Omega^2\partial^2\phi|^2\,dx\Big)^{\frac12}\lesssim C_3\epsilon^{-2}.
    \end{equation*}
    
Finally we turn to the spacetime term. 
We estimate $\|\Gamma^I\mathcal N(\psi,\phi)\|_{L^2(\mathcal H_\tau)}$, where  $\Gamma^I=\Omega^2\partial^2$,  by splitting it into two  groups according to the number of differentiated $\phi$ factors, namely factors of the form $\Gamma^{I_k}\phi$ with $|I_k|\ge 1$: 
$\mathcal{G}_1$ contains at most one such factor, whereas $\mathcal{G}_2$ contains at least two.
\begin{equation*}
    \begin{aligned}
        \mathcal{G}_1&:\{|\phi|^{p-1}\Gamma^I\phi,\ |\phi|^{p-1}\Gamma^I\psi,\ |\psi|^{p-1}\Gamma^I\phi,\ \phi|\psi|^{p-2}\Gamma^I\psi,
        \ |\phi|^{p-2}\Gamma^{I_1}\phi\Gamma^{I_2}\psi,\ |\phi|^{p-2}\Gamma^{I_1}\psi\Gamma^{I_2}\psi,\
        \\&\qquad|\psi|^{p-2}\Gamma^{I_1}\phi\Gamma^{I_2}\psi,\ 
        \phi|\psi|^{p-3}\Gamma^{I_1}\psi\Gamma^{I_2}\psi,\
        |\phi|^{p-3}\Gamma^{I_1}\phi\Gamma^{I_2}\psi\Gamma^{I_3}\psi,\ |\phi|^{p-3}\Gamma^{I_1}\psi\Gamma^{I_2}\psi\Gamma^{I_3}\psi,\ 
        \\&\qquad|\psi|^{p-3}\Gamma^{I_1}\phi\Gamma^{I_2}\psi\Gamma^{I_3}\psi,\ 
        \phi|\psi|^{p-4}\Gamma^{I_1}\psi\Gamma^{I_2}\psi\Gamma^{I_3}\psi,\ |\phi|^{p-4}\Gamma^{I_1}\phi\Gamma^{I_2}\psi\Gamma^{I_3}\psi\Gamma^{I_4}\psi,\ 
        \\&\qquad|\phi|^{p-4}\Gamma^{I_1}\psi\Gamma^{I_2}\psi\Gamma^{I_3}\psi\Gamma^{I_4}\psi,\ |\psi|^{p-4}\Gamma^{I_1}\phi\Gamma^{I_2}\psi\Gamma^{I_3}\psi\Gamma^{I_4}\psi,\ \phi|\psi|^{p-5}\Gamma^{I_1}\psi\Gamma^{I_2}\psi\Gamma^{I_3}\psi\Gamma^{I_4}\psi \},\\
        \mathcal{G}_2&:\{|\phi|^{p-2}\Gamma^{I_1}\phi\Gamma^{I_2}\phi,\ |\psi|^{p-2}\Gamma^{I_1}\phi\Gamma^{I_2}\phi,\ |\phi|^{p-3}\Gamma^{I_1}\phi\Gamma^{I_2}\phi\Gamma^{I_3}\phi,\ |\psi|^{p-3}\Gamma^{I_1}\phi\Gamma^{I_2}\phi\Gamma^{I_3}\phi,\ 
        |\phi|^{p-3}\Gamma^{I_1}\phi\Gamma^{I_2}\phi\Gamma^{I_3}\psi,\
        \\&\qquad|\psi|^{p-3}\Gamma^{I_1}\phi\Gamma^{I_2}\phi\Gamma^{I_3}\psi,\ 
        |\phi|^{p-4}\Gamma^{I_1}\phi\Gamma^{I_2}\phi\Gamma^{I_3}\phi\Gamma^{I_4}\phi,\ |\phi|^{p-4}\Gamma^{I_1}\phi\Gamma^{I_2}\phi\Gamma^{I_3}\phi\Gamma^{I_4}\psi,\ 
        |\phi|^{p-4}\Gamma^{I_1}\phi\Gamma^{I_2}\phi\Gamma^{I_3}\psi\Gamma^{I_4}\psi,\ 
        \\&\qquad|\psi|^{p-4}\Gamma^{I_1}\phi\Gamma^{I_2}\phi\Gamma^{I_3}\phi\Gamma^{I_4}\phi,\ |\psi|^{p-4}\Gamma^{I_1}\phi\Gamma^{I_2}\phi\Gamma^{I_3}\phi\Gamma^{I_4}\psi,\ |\psi|^{p-4}\Gamma^{I_1}\phi\Gamma^{I_2}\phi\Gamma^{I_3}\psi\Gamma^{I_4}\psi\},\\
    \end{aligned}
\end{equation*}
    where in each term $\sum_k{I_k}=I$ and $|I_k|\geq1$.
    Throughout we use the $L^2$ bounds in Propositions~\ref{prop:L2estimates_in},~\ref{prop:high-L2-1}, Lemma~\ref{lem:Omega2-energy}, the interior pointwise bounds in Lemmas~\ref{lem:phi-pt},~\ref{lem:der-pt} and the decay estimates in Lemma~\ref{lem:disp-psi}.

    \emph{Estimates for group $\mathcal{G}_1$.} Start by considering  $|\phi|^{p-1}\Omega^2\partial^2\phi$. Using Proposition~\ref{prop:high-L2-2} and Lemma \ref{lem:phi-pt}, we have
    \begin{equation*}
    \begin{aligned}
        \big\||\phi|^{p-1}\Omega^2\partial^2\phi\big\|_{L^2(\mathcal H_\tau)}
        &\lesssim\big\|\frac{\tau}{t+1}\Omega^2\partial^2\phi\big\|_{L^2(\mathcal H_\tau)}\big\|\frac{t+1}{\tau}|\phi|^{p-1}\big\|_{L^\infty(\mathcal H_\tau)}\\
        &\lesssim C_4^{\frac{3p^2-1}{2}}\epsilon^{-2+\frac{p-1}{8}}\tau^{-\frac{3}{2}p+\frac{3}{2}}.
    \end{aligned}
    \end{equation*}
    For $|\psi|^{p-1}\Omega^2\partial^2\phi$, Lemma~\ref{lem:Omega2-energy}
    and~\eqref{eq:disp-psi01} give
    \begin{equation*}
    \begin{aligned}
        \big\||\psi|^{p-1}\Omega^2\partial^2\phi\big\|_{L^2(\mathcal H_\tau)}
        &\lesssim\big\|\frac{\tau}{t+1}\Omega^2\partial^2\phi\big\|_{L^2(\mathcal H_\tau)}\big\|\frac{t+1}{\tau}|\psi|^{p-1}\big\|_{L^\infty(\mathcal H_\tau)}\\
        &\lesssim C_4^{p^2+p-1}\epsilon^{-2}\tau^{-\frac{p}{2}+\frac{1}{2}+(p-1)\delta}.
    \end{aligned}
    \end{equation*}

    Next consider $|\phi|^{p-2}\Gamma^{I_1}\phi\Gamma^{I_2}\psi$. Using the H\"{o}lder inequality, we obtain
    \begin{equation*}
    \begin{aligned}
        \big\||\phi|^{p-2}\Gamma^{I_1}\phi\Gamma^{I_2}\psi\big\|_{L^2(\mathcal H_\tau)}
        &\lesssim\big\|\frac{\tau}{t+1}\Gamma^{I_1}\phi\big\|_{L^2(\mathcal H_\tau)}\big\|\Gamma^{I_2}\psi\big\|_{L^\infty(\mathcal H_\tau)}\big\|\frac{t+1}{\tau}|\phi|^{p-2}\big\|_{L^\infty(\mathcal H_\tau)}.
    \end{aligned}
    \end{equation*}
    Using Proposition~\ref{prop:L2estimates_in}, Lemmas~\ref{lem:disp-psi} and~\ref{lem:phi-pt}, we derive
    \begin{equation*}
        \big\||\phi|^{p-2}\Gamma^{I_1}\phi\Gamma^{I_2}\psi\big\|_{L^2(\mathcal H_\tau)}\lesssim C_4^{p^2}\epsilon^{-\frac{13}{8}}\tau^{-\frac{3}{2}p+\frac{5}{2}+\delta}.
    \end{equation*}
    The remaining terms in~$\mathcal{G}_1$ are treated in the same way.

\noindent
    \emph{Estimates for group $\mathcal{G}_2$.} Consider first the term $|\phi|^{p-2}\Gamma^{I_1}\phi\Gamma^{I_2}\phi$. Using the H\"{o}lder inequality, we obtain
    \begin{equation*}
    \begin{aligned}
        \big\||\phi|^{p-2}\Gamma^{I_1}\phi\Gamma^{I_2}\phi\big\|_{L^2(\mathcal H_\tau)}
        &\lesssim\big\|\frac{\tau}{t+1}\Gamma^{I_1}\phi\big\|_{L^4(\mathcal H_\tau)}\big\|\frac{\tau}{t+1}\Gamma^{I_2}\phi\big\|_{L^4(\mathcal H_\tau)}\big\|\frac{(t+1)^2}{\tau^2}|\phi|^{p-2}\big\|_{L^\infty(\mathcal H_\tau)}.
    \end{aligned}
\end{equation*}
Employing Lemmas~\ref{lem:Sobo_embedding} and~\ref{lem:inter_inequ}, we obtain
\begin{equation*}
    \begin{aligned}
        \big\|\frac{\tau}{t+1}\Gamma^{I_1}\phi\big\|_{L^4(\mathcal H_\tau)}&\lesssim\big\|\frac{\tau}{t+1}\Gamma^{I_1}\phi\big\|^{\frac{1}{4}}_{L^2(\mathcal H_\tau)}\big\|\frac{\tau}{t+1}\Gamma^{I_1}\phi\big\|^{\frac{3}{4}}_{H^1(\mathcal H_\tau)},
    \end{aligned}
\end{equation*}
together with Proposition~\ref{prop:high-L2-1}, Lemmas~\ref{lem:Omega2-energy}-\ref{lem:phi-pt}, yielding
\begin{equation*}
    \big\||\phi|^{p-2}\Gamma^{I_1}\phi\Gamma^{I_2}\phi\big\|_{L^2(\mathcal H_\tau)}\lesssim C_4^{\frac{7}{4}p^2}\epsilon^{-\frac{33}{8}}\tau^{-\frac{3}{2}p+\frac{7}{2}}.
\end{equation*}
Similarly, we bound $|\psi|^{p-2}\Gamma^{I_1}\phi\Gamma^{I_2}\phi$ by
\begin{equation*}
    \begin{aligned}
        \big\||\psi|^{p-2}\Gamma^{I_1}\phi\Gamma^{I_2}\phi\big\|_{L^2(\mathcal H_\tau)}
        &\lesssim C_4^{\frac{3}{2}p^2}\epsilon^{-\frac{41}{8}}\tau^{-\frac{p}{2}+\frac{1}{2}+(p-3)\delta}.
    \end{aligned}
\end{equation*}

    Finally consider $|\phi|^{p-3}\Gamma^{I_1}\phi\Gamma^{I_2}\phi\Gamma^{I_3}\phi$. Using the H\"{o}lder inequality, we find
    \begin{equation*}
    \begin{aligned}
        &\big\||\phi|^{p-3}\Gamma^{I_1}\phi\Gamma^{I_2}\phi\Gamma^{I_3}\phi\big\|_{L^2(\mathcal H_\tau)}\\
        &\lesssim\big\|\frac{\tau}{t+1}\Gamma^{I_1}\phi\big\|_{L^6(\mathcal H_\tau)}\big\|\frac{\tau}{t+1}\Gamma^{I_2}\phi\big\|_{L^6(\mathcal H_\tau)}\big\|\frac{\tau}{t+1}\Gamma^{I_3}\phi\big\|_{L^6(\mathcal H_\tau)}\big\|\frac{(t+1)^3}{\tau^3}|\phi|^{p-3}\big\|_{L^\infty(\mathcal H_\tau)}.
    \end{aligned}
    \end{equation*}
    Employing Lemma~\ref{lem:Sobo_embedding}, we obtain
    \begin{equation*}
        \begin{aligned}
            \big\|\frac{\tau}{t+1}\Gamma^{I_1}\phi\big\|_{L^6(\mathcal H_\tau)}&\lesssim\big\|\frac{\tau}{t+1}\Gamma^{I_1}\phi\big\|_{H^1(\mathcal H_\tau)}.
        \end{aligned}
    \end{equation*}
    Applying Proposition~\ref{prop:L2estimates_in} and Lemma~\ref{lem:phi-pt}, we get
    \begin{equation*}
        \big\||\phi|^{p-3}\Gamma^{I_1}\phi\Gamma^{I_2}\phi\Gamma^{I_3}\phi\big\|_{L^2(\mathcal H_\tau)}\lesssim C_4^{p^2}\epsilon^{-\frac{5}{4}}\tau^{-\frac{3}{2}p+\frac{13}{2}}.
    \end{equation*}
    Similarly, for $|\psi|^{p-3}\Gamma^{I_1}\phi\Gamma^{I_2}\phi\Gamma^{I_3}\phi$, using~\eqref{eq:disp-0001}, we obtain
    \begin{equation*}
        \big\||\psi|^{p-3}\Gamma^{I_1}\phi\Gamma^{I_2}\phi\Gamma^{I_3}\phi\big\|_{L^2(\mathcal H_\tau)}\lesssim C_4^{p^2}\epsilon^{\frac{3}{2}-p}\tau^{-\frac{3}{2}p+\frac{13}{2}}.
    \end{equation*}
    The remaining terms in the second group are treated in a similar way.

   Inserting the above estimates into the energy inequality \eqref{eq:partial2OMega2-energy-raw}, we obtain the desired bound.

  $\bullet$ Finally we turn to the conformal estimate, which is obtained in a similar manner. Applying the conformal energy inequality \eqref{est:energy_in2} to the commuted equation gives
\begin{equation*}
\begin{aligned}
    \mathcal E^{in}_{con}(\tau,\Omega^2\partial^2\phi)^{\frac12}
&\lesssim
\|\partial\Omega^2\partial^2\phi\|_{L^2(\Sigma_{t_1}\cap\mathcal D^{in})}
+\|\Omega^2\partial^2\phi\|_{L^2(\Sigma_{t_1}\cap\mathcal D^{in})}\\
&\quad +\left(\int_{\mathcal{B}_{4\epsilon}}\Big[(t+r+1)L\Omega^2\partial^2\phi + 2\Omega^2\partial^2\phi\Big]^2+\Big|\frac{1}{r}\Omega\Omega
^2\partial^2\phi\Big|^{2}\, dx\right)^{\frac{1}{2}}\\
&\quad +\int_{\tau_0}^{\tau}\|\tilde\tau\,\Omega^2\partial^2\mathcal N(\psi,\phi)\|_{L^2(\mathcal H_{\tilde\tau})}\,d\tilde\tau.
\end{aligned}
\end{equation*}

The initial and boundary terms are controlled exactly as above, using
Propositions~\ref{prop:estimates_local1} and~\ref{prop:energy_ex}. The spacetime term is estimated in the same way, except that the conformal energy inequality introduces an additional factor $\tilde\tau$ in the source term. Once again, the integral is finite if and only if $p>\frac{17}{3}$. 
\end{proof}

\begin{lemma}[Pointwise bound for $\partial^2\phi$]\label{lem:dd-phi}
    For $(t,x)\in\mathcal{D}^{in}$, we have
    \begin{equation}\label{est:dd-phi}
        \begin{aligned}
            |\partial^2\phi(t,x)|&\lesssim C_4^{p^2}\epsilon^{-\frac{5}{2}}\tau^{-1}t^{-\frac{1}{2}},\qquad r\leq t/4,\\
            |\partial^2\phi(t,x)|&\lesssim\left\{\begin{aligned}
            C_4^{\frac{11}{8}p^2}\epsilon^{-\frac{57}{16}}\tau^{\frac{-3p+13}{4}}r^{-\frac{1}{2}},\qquad p<17/3,\\
            C_4^{\frac{11}{8}p^2}\epsilon^{-\frac{57}{16}}\tau^{-1}(\log\tau)^{\frac{1}{2}} r^{-\frac{1}{2}},\qquad p=17/3,\\
            (C_4^{\frac{11}{8}p^2}\epsilon^{-\frac{57}{16}}+C_4^{p^2}\epsilon^{-\frac{1}{4}-\frac{p}{2}})\tau^{-1}r^{-\frac{1}{2}},\qquad p>17/3,
     \end{aligned}\qquad r\geq t/8. \right.
        \end{aligned}
    \end{equation}
\end{lemma}
\begin{proof}
   $\bullet$ We first consider the case when $r\leq t/4$. From~\eqref{est:KS_in1} it follows that
    \begin{equation*}
        \begin{aligned}
            |\partial^2\phi(t,x)|\lesssim\tau^{-1}t^{-\frac{1}{2}}\sum_{|I|\leq 2}\epsilon^{(|I|-\frac{3}{2})\lambda}\Big\|\frac{\tau}{t+1}Z^{I}\partial^2\phi\Big\|_{L^2(\mathcal{H}_\tau)}.
        \end{aligned}
    \end{equation*}
    Choosing $\lambda=1$ and applying Propositions~\ref{prop:L2estimates_in},~\ref{prop:high-L2-1} and~\ref{prop:high-L2-3}, we get the desired estimate
    \begin{equation*}
        |\partial^2\phi|\lesssim C_4^{p^2}\epsilon^{-\frac{5}{2}}\tau^{-1}t^{-\frac{1}{2}}.
    \end{equation*}
    
    $\bullet$ Next we turn to the case $r\geq t/8$. Using~\eqref{est:KS_in3}, we have
    \begin{equation*}
        \begin{aligned}
            &|\partial^2\phi(t,x)|\lesssim|\partial^2\phi(t,x)|_{\mathcal{H}_{\tau}\cap\mathcal{B}_{4\epsilon}}
            +\sum_{\substack{|J_1|\leq 2,|J_2|\leq2\\|I|=1}}\tau^{-1}|x|^{-\frac{1}{2}}\left\|\frac{\tau}{t}\Omega^{J_1}\partial^2\phi\right\|_{L^2(\mathcal{H}_\tau)}^{\frac{1}{2}}\left\|\frac{\tau}{t}\Omega^{J_2}Z^{I}\partial^2\phi\right\|_{L^2(\mathcal{H}_\tau)}^{\frac{1}{2}}.
        \end{aligned}
    \end{equation*}
    Utilising the bounds from  Propositions~\ref{prop:L2estimates_in},~\ref{prop:high-L2-1},~\ref{prop:high-L2-2},~\ref{prop:high-L2-3},~\ref{prop:high-L2-4}, gives the desired estimates.
\end{proof}

In conclusion, Lemmas~\ref{lem:phi-pt},~\ref{lem:der-pt} and~\ref{lem:dd-phi} provide control of $\phi, \partial \phi$ and $\partial^2 \phi$ in $\mathcal{D}^{in}$. Combined with our exterior regularity estimates, and the good pointwise control on $\psi$, we can conclude that $\mainfield$ is a global classical solution to \eqref{intro-eq}.



\end{document}